\documentclass[12pt]{amsart}

\usepackage[all,cmtip]{xy}
\usepackage{graphicx, overpic,float}
\usepackage[normalem]{ulem}
\usepackage{amssymb}
\usepackage{amsmath,amscd}
\usepackage{setspace}
\usepackage{color}
\usepackage[colorlinks=true, citecolor=darkblue, linkcolor=darkblue]{hyperref}
\usepackage{fancyhdr}

\usepackage[usenames,dvipsnames]{xcolor}
\definecolor{dblue}{cmyk}{1,0, 0,.6}
\definecolor{darkblue}{cmyk}{1,0, 0,.7}
\usepackage[mathscr]{euscript}

\definecolor{blue(pigment)}{rgb}{0.2, 0.2, 0.6}

\let\oldmarginpar\marginpar
\renewcommand\marginpar[1]{\-\oldmarginpar[\raggedleft\footnotesize #1]%
{\raggedright\footnotesize #1}}

\newtheorem{lemma}{Lemma}[section]
\newtheorem{proposition}[lemma]{Proposition}
\newtheorem{theorem}[lemma]{Theorem}
\newtheorem{claim}[lemma]{Claim}

\newtheorem{notation}[lemma]{Notation}
\newtheorem{question}[lemma]{Question}

\newtheorem{remark}[lemma]{Remark}

\newtheorem{corollary}[lemma]{Corollary}
\newtheorem{definition}[lemma]{Definition}

\newcommand{\Proof}{{\noindent \it Proof. }}
\newcommand{\proofof}[1]{{\noindent \it Proof of #1. }}
\newcommand{\pf}[1]{{\noindent \it Proof of \ref{#1}. }}

\newcommand{\cmm}[1]{\marginpar{ {\it #1}}}   
\renewcommand{\cmm}[1]{}    

\newcommand{\Qed}[1]{\nopagebreak[4]{\tiny \hfill\fbox{\ref{#1}} \linebreak }\pagebreak[2]}

\renewcommand{\Im}{\operatorname{Im}}
\newcommand{\Hol}{\operatorname{Hol}}

\newcommand{\dist}{\operatorname{dist}}
\newcommand{\Conv}{\operatorname{Conv}}
\newcommand{\Core}{\operatorname{Core}}
\newcommand{\Gr}{\operatorname{Gr}}

\newcommand{\mc}{\mathcal}
\newcommand{\td}{\tilde}
\newcommand{\on}{\operatorname}

\newcommand{\ul}{\underline}
\newcommand{\pt}{\partial}
\newcommand{\bd}{\partial}
\newcommand{\bdr}{\partial}

\newcommand{\cc}{\circ}
\newcommand{\iv}{^{-1}}

\newcommand{\h}{\mathbb{H}}
\renewcommand{\H}{\mathbb{H}}
\newcommand{\s}{\mathbb{S}}
\newcommand{\C}{\mathbb{C}}
\newcommand{\R}{\mathbb{R}}
\newcommand{\N}{\mathbb{N}}
\newcommand{\n}{\mathbb{N}}
\newcommand{\Z}{\mathbb{Z}}

\newcommand{\minus}{\setminus}
\newcommand{\sm}{\setminus}

\newcommand{\st}{\subset}
\newcommand{\sub}{\subset}
\newcommand{\dt}{\ldots}
\newcommand{\cn}{\colon}

\newcommand{\If}{\infty}
\newcommand{\ify}{\infty}
\newcommand{\In}{\infty}

\newcommand{\po}{\pi_1}
\newcommand{\PSL}{\operatorname{PSL(2, \mathbb{C})}}
\newcommand{\SL}{\operatorname{SL(2, \mathbb{C})}}
\newcommand{\psl}{\operatorname{PSL(2, \mathbb{C})}}

\newcommand{\rs}{\hat{\mathbb{C}}}
\newcommand{\PML}{\mathscr{PML}}
\newcommand{\pml}{\mathscr{PML}}
\newcommand{\ml}{\mathscr{ML}}
\newcommand{\ML}{\mathscr{ML}}

\newcommand{\gl}{\mathscr{GL}}
\newcommand{\GL}{\mathscr{GL}}

\newcommand{\Tr}{\operatorname{Tr}}
\newcommand{\TM}{\mathscr{TM}}

\renewcommand{\AA}{\mathcal{A}}

\newcommand{\NN}{\mathcal{N}}

\newcommand{\CC}{\mathcal{C}}

\newcommand{\UU}{\mathcal{U}}
\newcommand{\mP}{\mathcal{P}}

\newcommand{\mL}{\mathcal{L}}

\newcommand{\PP}{\mathscr{P}}
\newcommand{\TT}{\mathcal{T}}
\newcommand{\RR}{\mathcal{R}}
\newcommand{\LL}{\mathcal{L}}
\newcommand{\VV}{\mathcal{V}}
\newcommand{\KK}{\mathcal{K}}

\newcommand{\WW}{\mathcal{W}}

\newcommand{\mCt}{\tilde{\mathcal{C}}}

\newcommand{\dl}{\delta}
\newcommand{\del}{\delta}

\newcommand{\kap}{\kappa}
\newcommand{\kp}{\kappa}

\newcommand{\gm}{\gamma}
\newcommand{\gam}{\gamma}

\newcommand{\ld}{\lambda} 
\newcommand{\lam}{\lambda}

\newcommand{\ap}{\alpha}

\newcommand{\ep}{\epsilon}

\newcommand{\Ct}{\tilde{C}}

\newcommand{\St}{\tilde{S}}

\newcommand{\psit}{\tilde{\psi}}

\newcommand{\kpt}{\tilde{\kappa}}

\newcommand{\Ch}{\hat{C}}

 \newcommand{\nuh}{\hat{\nu}}

\newcommand{\Cc}{C_\flat}

\newcommand{\Cs}{C_\sharp}
\newcommand{\Cf}{C_\flat}

\newcommand{\tcr}{\textcolor{red}}
\newcommand{\tcb}{\textcolor{blue}}

\date{\today}

\usepackage{xr}
\newcommand\note[1]{\mbox{}\marginpar{\scriptsize\raggedright\hspace{0pt}\color{darkblue}#1}}
\renewcommand\note[1]{\mbox{}\marginpar{\scriptsize\raggedright\hspace{0pt}\color{darkblue}}}

\newcommand{\Label}[1]{\label{#1}\note{#1}}
\renewcommand{\Label}[1]{\label{#1}}

\newcommand{\infi}{\infty}
\newcommand{\bdry}{\partial}

\newcommand{\til}{\tilde}

\newcommand{\col}{\colon}

\newcommand{\pr}{\operatorname{pr}}

\newcommand{\PG}{\mathcal{PG}}
\definecolor{orange}{rgb}{1,0.5,0}
\definecolor{darkgreen}{cmyk}{1,0,1,.2}

\newcommand{\com}[1]{}

\usepackage{mathtools}

\newcommand{\TTT}{\mathsf{T}}
\newcommand{\RRR}{\mathsf{R}}

\newcommand{\mr}{\mathring}

\newcommand{\length}{{\rm lenth}}

 \title[\today]{ 2$\pi$-Grafting and complex projective structures with generic holonomy}
\author{Shinpei Baba}

\address{Universit\"at Heidelberg 
}

\email{shinpei@mathi.uni-heidelberg.de}
\date{\today}

\begin{document}

\maketitle

\begin{abstract}
Let  $S$ be an oriented closed surface of genus at least two.
We show that, given a generic representation $\rho\cn \pi_1(S) \to \psl$ in the character variety, ($2\pi$-)grafting produces all  projective structures on $S$ with holonomy $\rho$.\end{abstract}

\setcounter{tocdepth}{1}  
\tableofcontents
\section{Introduction}\Label{S:intro}
Let $S$ be a closed oriented surface of genus at least two throughout this paper.
A (complex) projective structure on $S$ is a $(\rs, \PSL)$-structure (see \S \ref{projectivestructure}). 
It induces a (holonomy) representation $\rho\cn \pi_1(S) \to \psl$ unique up to conjugation by an element of $\psl$.

 Let  $\PP_\rho$ be the set of all marked projective structures on $S$ with fixed holonomy $\rho\col \pi_1(S) \to \psl$.
 It is a basic question to understand $\PP_\rho$, in order to understand geometry behind the homomorphism $\rho$, which is not necessarily discrete.
 This question goes back to a foundational paper of Heijal  \cite[p2, (B)]{Hejhal-75}, and it also appeared in various articles (\cite[p 274]{Hubbard-81}, \cite[\S7.1]{Kapovich-95}, \cite[\S12.1]{Gallo-Kapovich-Marden}, \cite[\S 1]{Dumas-08}; see also \cite[\S 1.10]{Goldman-thesis}).

 A {\it ($2\pi$-)graft} is a surgery operation that transforms a  projective structure in $\PP_\rho$ to another  in $\PP_\rho$ (\S \ref{grafting}).
  In the preceding paper (\cite{Baba-15gt}), the author showed that projective structures in $\PP_\rho$ are related by grafting if they are ``close'' in the space of geodesic laminations $\GL$ in Thurston coordinates. 
In this paper,  we aim to relate, without the ``closeness'' assumption, all projective structures in $\PP_\rho$ by  grafting. 
      
Let $\PP$ be the set of all marked projective structures on $S$. 
Then $\PP$ is diffeomorphic to $\C^{6g-6}$. 
Let $\chi$ be the $\psl$-character variety of $S$, that is, the set of all representations $\rho\cn \pi_1(S) \to \psl$, roughly, up to conjugation (\S \ref{purelyloxo}).
Then $\chi$ is a complex affine algebraic variety, and it consists of exactly two connected components (\cite{Goldman-88t}). 
Let $\chi_0$ be the \textit{canonical component} of $\chi$ consisting of representations that lift to $\po(S) \to \operatorname{SL}(2, \C)$.
Let  $$\rm{Hol}\cn \PP \to \chi$$ be  the \textit{holonomy map}, which  takes each projective structure to its holonomy representation. 
Then the image of $\Hol$ is contained in  $\chi_0$, and moreover $\Hol$ is almost onto $\chi_0$ (\cite{Gallo-Kapovich-Marden}).  
For example, there are many holonomy representations whose images are dense in $\psl$.

Noting $\PP_\rho = \Hol^{-1}(\rho)$,  we are interested in understanding fibers of $\Hol$. 
The holonomy map  $\Hol$ is a local homeomorphism \cite{Hejhal-75} (moreover a local biholomorphism \cite{Hubbard-81, Earle-81});
 however  is not a covering map onto its image.
Thus each fiber $\Hol^{-1}(\rho)$ is a discrete subset of $\PP$, but $\PP_\rho$ may possibly be quite different depending on $\rho \in \chi_0$. 

A graft of a projective surface inserts a projective cylinder along an appropriate loop, called an admissible loop, on the surface. 
Then  an ungraft is the opposite of a graft, which removes such a projective cylinder; thus it also preserves holonomy.  
Then 
\begin{question}[\cite{Gallo-Kapovich-Marden}; Grafting Conjecture]\Label{GKM}
Given two projective structures with holonomy $\rho\cn \pi_1(S) \to \psl$, is there a composition of grafts and ungrafts that transforms one to the other?
\end{question}

A basic known case is when $\rho$ is a discrete and faithful representation onto a quasifuchsian group. 
Then $\PP_\rho$ contains a unique \textit{uniformizable} projective structure (i.e. its developing map is an embedding into $\rs$);
Then every projective structure in $\PP_\rho$ is moreover obtained by  grafting the uniformizable structure along a {\it multiloop}, a union of disjoint essential simple closed curves  \cite{Goldman-87}. 
On the other hand, if $\rho \in \chi_0$ is a generic representation  outside the quasifuchsian space, then $\rho$ has a dense image in $\psl$; in particular,
there is no uniformizable structure with holonomy $\rho$.
Thus, for general holonomy, Question \ref{GKM} seems an appropriate analogy of the quasifuchsian case.

In this paper we answer Question \ref{GKM} in the affirmative for generic representations in $\chi_0$, namely,  of the following type:
An element $\ap \in \psl$ is \textit{loxodromic} if its trace $\Tr(\ap) \in \C$,  which is well-defined up to a sign, is not contained in $[-2,2] \st \R$.
A representation $\rho\cn \pi_1(S) \to \psl$ is called \textit{purely loxodromic} if $\rho(\gam)$ is loxodromic for all $\gam \in \pi_1(S)$. 
Then almost all elements of $\chi_0$ are purely loxodromic (Proposition \ref{040213}).
\begin{theorem}\Label{main}
Let $\rho\cn \po(S) \to \psl$ be a purely loxodromic representation in $\chi_0$. 
Then, given any $C_\sharp, \Cc$ in $\PP_\rho$, there is a composition of grafts and ungrafts that transforms $C_\sharp$ to $C_\flat$.
 Namely there is a finite composition of grafts $\Gr_{\ell_i}$ along loops $\ell_i$ starting from $\Cs$, 
 \begin{eqnarray}
\Cs = C_0  \xrightarrow{Gr_{\ell_1}} C_1  \xrightarrow{Gr_{\ell_2}} C_2 \to \ldots  \xrightarrow{Gr_{\ell_n}} C_n, \label{030713n2}
\end{eqnarray}
such that the last projective structure $C_n$ is a graft of $\Cc$ along a multiloop $M$,
\begin{eqnarray}
\Cc  \xrightarrow{Gr_M} C_n. \label{030713n3}
\end{eqnarray}
\end{theorem}
Here a ``graft along a multiloop $M$'' means simultaneous grafts along all loops of $M$.

In the case where $\rho$ is a quasifuchsian representation, \cite[Theorem 3]{Ito07} implies Theorem \ref{main} even  in a stronger form:
 Namely the sequence (\ref{030713n2})  can be replace by a single graft along a multiloop. 
(See also \cite{Calsamiglia4DeroinFrancaviglia14, Baba-15gt}.)

Although Theorem \ref{main} answers Question \ref{GKM} in a generic setting,  the question in full generality remains open. 
Nonetheless many techniques in this paper, including Theorem \ref{ThmB} and Theorem \ref{thmC} below,   applies to arbitrary representations in $\Im \Hol$.

In our proof of Theorem \ref{main}, 
we utilize Thurston coordinates on $\PP$, which are given by (not necessarily $2\pi$-)grafting of hyperbolic surfaces.  
Namely there is a natural homeomorphism 
\begin{displaymath}\label{030713n1}
\PP \cong \mathscr{T} \times \ML,
\end{displaymath}
where $\TT$ is the space of marked hyperbolic structures on $S$ and $\ML$ the space of measured laminations on $S$ (see \S \ref{thurston}).
Note that $\TT$ is diffeomorphic to $\R^{6g-6}$ and $\ML$ is PL diffeomorphic to $\R^{6g-6}$.
We denote Thurston coordinates of a projective structure $C$ using $``\cong"$ as $C \cong (\tau, L) \in \mathscr{T} \times \ML$. 

Let $\GL$ be the space of geodesic laminations on $S$. 
Then we obtain an obvious projection $\ML \to \GL$, forgetting transversal measures. 
 In the preceding paper \cite{Baba-15gt}, the author shows that any $C \in \PP_\rho$ is related, by grafting, to all projective structures in $\PP_\rho$ that are, in  Thurston coordinates,  ``close'' to $C$  in $\GL$ by the projection map (see Theorem \ref{8-14-12no1}).
This local relation yields the graft  (\ref{030713n3}).
Thus our main work in this paper is to construct the sequence (\ref{030713n2}) so that $C_n$ is ``close'' to $\Cf$ in $\GL$.

In order to have a control on geodesic laminations of projective structures, we observe an asymptotic, in Thurston coordinates,  of projective structures given by the iteration of grafts along a fixed loop. In particular

\begin{theorem}\Label{ThmB}
Let $C \cong(\tau, L)$ be a projective structure on $S$ in  Thurston coordinates, where $(\tau, L) \in  \mathscr{T} \times \ML$.
Let $\ell$ be an admissible loop on $C$. 
 For $i \in \Z_{> 0}$, let $C_i \cong (\tau_i, L_i)$ be the projective structure obtained by $i$-times grafting  $C$ along $\ell$ (i.e. $2\pi i$-graft). 
Then $\tau_i$ converges in $\mathscr{T}$, and $L_i$ converges to a (heavy) measured lamination $L_\infty$ as $i \to \In$ such that $\ell$ is  a unique leaf of $L_\In$ of  weight infinity.
 (See  Theorem \ref{3-10no2}.)
\end{theorem}
In this paper, some closed leaves of  laminations may have weight infinity if stated, as in Theorem \ref{ThmB}  (see \S \ref{ML}).

In the special case that $C$ is a hyperbolic surface (i.e.  the developing map of $C$ is an embedding onto a round disk),  Theorem \ref{ThmB} is clear.
Namely, $\tau_i = \tau$ for all $i$, and $L$ is equal to $\ell$ with weight $2\pi i$ (\cite{Goldman-87}).
Thus Theorem \ref{ThmB} asserts that, asymptotically,  $C_i$  behaves similarly to the iteration of grafts of a hyperbolic surface. 
(In contrast,  the conformal structure of $C_i$ diverges but converges to a point in the Thurston boundary of  $\mathscr{T}$ along every grafting ray starting from a hyperbolic structure \cite{ChoiDumasRafi12, Diaz-Kim-12, Hensel11, Gupta14}.)

By projectivizing transversal measures, $\ML$ minus the empty lamination projects onto the space of projective measured laminations, $\pml\, (\cong \s^{6g -7})$. 
In the appendix, we prove 
\begin{theorem}\Label{thmC}
Given arbitrary $\rho \in \Im \Hol$, in Thurston coordinates, $\PP_\rho$ projects onto a dense subset of $\PML$, unless $\PP_\rho$ is empty.
(see  Theorem \ref{8-15-12no1}.) 	
\end{theorem}

With Theorem  \ref{thmC}, it seems quite natural to use Thurston coordinates in order to answer Question \ref{GKM}. 
A similar density is well-known for  geodesic  laminations realized by homotopic pleated surfaces in a fixed hyperbolic three-manifold (see \cite{CanaryEpsteinGreen84}). 

Theorem \ref{thmC} is  obtained by carefully observing the construction of a projective structure with given holonomy 
in \cite{Gallo-Kapovich-Marden} and applying Theorem \ref{ThmB}.

\subsection{Outline of the proof of Theorem \ref{main} }
{\it Part 1.}
Each projective structure on the surface $S$ with holonomy $\rho$ corresponds to a $\rho$-equivariant pleated surface $\h^2 \to \h^3$ (\S\ref{thurston}).
Thus, given two projective structures with the same purely loxodromic holonomy, we first consider their $\rho$-equivariant pleated surfaces $\beta_\sharp$ and $\beta_{\flat}$. 
In \S \ref{5-14},  we construct a ordered family of  $\rho$-equivariant pleated surfaces, so that $\beta_\sharp$ is transformed to  $\beta_\flat$  through a sequence  pleated surfaces in this family, by composition of certain type of simple changes (up to very small perturbations).

In \S \ref{admissible}, given a $\rho$-equivariant pleated surface $\beta$ in the family and a projective structure $C$ with holonomy $\rho$, if the pleating lamination of the pleated surface of $C$ is sufficiently close to the pleating lamination of $\beta$, then every loop $\ell$ close  to the pleating lamination of the next pleated surface is admissible (Proposition \ref{3-2no1}). 
Note we can graft $C$ along the admissible loop $\ell$ as many as we want.
In Section \ref{sTravelingInGL}, using such graftings, we prove Theorem \ref{main}, modulo the result in Part 2 regarding the limit of iterated grafting.

{\it  Part 2.}
Let  $\ell$ be an admissible loop on a projective structure $C$ on $S$ with (arbitrary) holonomy $\rho$. 
We consider the  $n$-times grafting of $C$ along an admissible loop $\ell$ and characterize its limit, as $n \to \infi$, in Thurston coordinates (\S \ref{thurston}). 

The as a projective structure
$\Gr^n_\ell(C)$ converges, in a certain sense, to a projective structure $\CC_\infi$ on $S \minus \ell$. 
In \S \ref{6-8no1},  we  show that the Thurston coordinates of   $\CC_\infi$ are a hyperbolic structure $\sigma_\infi$ on $S \minus \ell$ and a measured lamination $N_\infi$ on it.  

In \S \ref{identification}, we identify  the boundary components of $\sigma_\infi$ naturally so that it corresponds to a $\rho$-equivariant pleated surface.  
Then we have a hyperbolic structure $\tau_\infi$ on $S$ and a measured lamination $L_\infi$ on $\tau_\infi$, which is the expected limit of Thurston coordinates $(\tau_i, L_i)$  of $C_i$ as $i \to \infi$. 
Here the measured lamination $L_\infi$ is a bit more generalized than its usual notion:  $\ell$ is a leaf of $L_\infi$ with weight infinity.

Given a point on a projective surface $C$,
A {\it canonical neighborhood} (\S \ref{5-4no1}) of $p$ is a nice neighborhood homeomorphic to an open disk:  which is embedded in $\rs$ and yet  large enough to capture  the Thurston coordinates of $C$ near the point.
By embedding $\ell$ isomorphically to each $C_i$ appropriately, so that we have the inclusions $C_1 \minus \ell \sub C_2 \minus \ell \sub C_3 \minus \ell \dots $ and $C_i \minus \ell$ converges to $\CC_\infi$.
In \S \ref{sNeighborhoodsConverging}, given a point $p \in \CC_\infi$,   we show the convergence of the canonical neighborhoods of $p$ in $C_i$ when $i \to \infi$. 

In \S \ref{6-30no1}, given a converging sequence of  projective structures on a open disk which are embedded in $\rs$, we  show the convergence of the sequence in Thurston coordinates. 

In \S \ref{Proof} we prove the convergence of $(\tau_i, L_i) \to (\tau_\infi, L_\infi)$, combining the results above.

\subsection{Acknowledgements:}
I would like thank  Ken Bromberg,  Bill Goldman, Ursula Hamenst\"adt, Sebastian Hensel, Albert Marden,  Yair Minsky, and Saul Schleimer. 
  Special thanks to Misha Kapovich. 

I acknowledge support from the GEAR Network (U.S. National Science Foundation grants DMS 1107452, 1107263, 1107367), the European Research Council (ERC-Consolidator grant no. 614733),  and the German Research Foundation (BA 5805/1-1).  
I thank the referees for reading this paper carefully  and giving me valuable comments.

\section{Preliminaries}
  \cite{Kapovich-01} is a general background reference.  
For hyperbolic geometry in particular, see \cite{CanaryEpsteinGreen84, Epstein-Marden-87}.
See also the preceding paper   \cite{Baba-15gt}.

\subsection{Projective structures}\Label{projectivestructure} (c.f. \cite{Thurston-97}.) 
Let $F$ be an oriented connected surface, and let $\td{F}$ be the universal cover of $F$.
Let $\rs$ denote the Riemann sphere $\C \cup \{ \infi \}$.
A {\it projective structure} $C$ on $F$ is a $(\rs, \psl)$-structure, i.e. it is an atlas modeled on $\rs$ with transition maps in $\psl$. 
(In particular $C$ is a refinement of a complex structure.)
In this paper all projective structures $C$ are marked by a homeomorphism $F \to C$. 
Then, equivalently, a projective structure on $F$ is a pair $(f, \rho)$,
where $f\cn \td{F} \to \rs$ is an immersion and $\rho\cn \pi_1(F) \to \psl$ is a homomorphism such that $f$ is $\rho$-equivariant. 
The immersion $f$ is called the \textit{developing map}, which we denote by $dev(C)$,  and $\rho$ the \textit{holonomy representation} of $C$. 
The equivalence of  projective structures on $F$ is given by the isotopies of $F$ and   $(f, \rho) \sim (\gm \cc f, \gm \rho \gm^{-1})$ for all $\gam \in \PSL$.

\subsection{Grafting} (\Label{grafting}\cite{Goldman-87}.)
Let $C = (f, \rho)$ be a  projective structure on $F$. 
A loop $\ell$ on $C$ is \textit{admissible} if $\rho(\ell)$ is loxodromic and  $f$ embeds $\td{\ell}$ into $\rs$, where $\td{\ell}$ is   a lift of $\ell$  to $\td{F}$.
Then $\rho(\ell)$ fixes exactly two points on $\rs$, and $\rho(\ell)$ generates an infinite cyclic group in $\PSL$.
Its domain of discontinuity  is  $\rs$ minus the two points, and its quotient by the cyclic group  is a two-dimensional torus $T_\ell$ has a projective structure.
Then $\ell$ is naturally embedded in $T_\ell$. 
Therefore we can naturally combine two projective surfaces $C$ and $T_\ell$ by cutting and pasting along $\ell$, so that it results a new projective structure $\Gr_\ell(C)$ on $F$. 
(Namely we identify boundary components $C \minus \ell$ and $T_\ell \minus \ell$ by the identification of $\ell$ on $C$ and on $T_\ell$ in an alternating manner.)
 Then it turns out that $\rho$ is also the holonomy of $\Gr_\ell(C)$.

\subsection{Measured laminations}\Label{ML}
Let $F$ be a surface (possibly with boundary), and let $\tau$ be a hyperbolic surface homeomorphic to $F$  (with geodesic boundary).
A \textit{geodesic lamination} $\lam$ is a set of disjoint geodesics whose union is a closed subset of $\tau$\,;  we denote this closed subset by $|\lam|$. 
Those geodesics are call {\it leaves} of the lamination.
A geodesic lamination $\lam$ is \textit{maximal} if its complement is a union of disjoint ideal triangles. 
A \textit{stratum} of $\lam$ is either a leaf of $\lam$ or the closure of a complementary region of $\lam$.

Let $\AA(\lam)$ be the set of all smooth simple arcs $\ap$ on $\tau$,  containing their endpoints, such that $\ap$ is transversal to $\lam$ at  its interior points and  not tangent to $\lam$ at its endpoints.
Let $\mathring{\AA}(\lam)$ be the set of all smooth simple arcs $\ap$ on $\tau$ such that $\ap$ is transversal to $\lam$ and the end points of $\ap$ are not in $|\lam|$. 
Then $\mathring{\AA}(\lam)$ is a dense subset of $\AA(\lam)$ in the Hausdorff topology.

A {\it transversal measure} on $\lam$ is a function $\mu\cn \mathring{\AA}(\lam) \to \R_{\geq 0}$  such that 
\begin{itemize}
\item $\mu(\ap) > 0$ if and only if $\ap \in \mathring{\AA}(\lam)$ intersects $\lam$,
\item if $\ap \in \mr{\AA}(\lam)$  is a composition of two arcs $\ap_1, \ap_2 \in \mr{\AA}(\lam)$, then $\mu(\ap) = \mu(\ap_1) + \mu(\ap_2)$, and
\item
 $\mu(\ap)$ is invariant under any isotopy of $\ap$ through arcs in $\mr{\AA}(\lam)$. 
\end{itemize}

For $p,q \in \tau$, if there is a unique shortest geodesic segment connecting $p$ to $q$, then let  $\mu(p, q)$ denote the transversal measure of the segment. 
The {\it measured lamination} $L$ is a pair $(\lam, \mu)$ of a geodesic lamination $\lam$ and the transversal measure  $\mu$ supported on $\lam$. 
Let $\TM(\lam)$ denote the set of all transversal measures supported on $\lam$.
Let $\ML(F)$ be the set of all measured laminations on $(F, \tau)$.
Note that we do not need to specify $\tau$, since, for different hyperbolic structures on $F$, the corresponding  spaces $\ML(F)$ are naturally isomorphic (see \cite[\S 1]{Bonahon97}).  \note{better reference?}
Suppose that $\lam$ contains a closed leaf $\ell$.
Then $\ell$ carries an atomic measure  (\textit{weight}), which is a positive real number. 

\begin{notation}
 $[a, b]$ denotes the geodesic segment connecting $a$ and $b$.
\end{notation}

Let  $L = (\lam, \mu)$ be a measured lamination on $\tau$.  
Then, for $\ap \in \AA(\lam) \minus \mathring{\AA}(\lam)$,   we can naturally define its transversal measure $\mu(\ap)$ to be a closed interval in $\R_{\geq 0}$ as follows. 
Let $(a_i) \st \mathring{\AA}(\lam)$ be a sequence converging to $\ap$ with $a_i \st \ap$. 
Similarly let  $(b_i) \sub \mathring{\AA}(\lam)$ be a sequence converging to $\ap$ with $\ap \st b_i$.   
Then, the transversal measure $\mu(\ap)$ is the closed interval
$$[\lim_{i \to \In} \mu(a_i), \lim_{i \to \In} \mu(b_i)].$$
Note that the width of the interval is  the sum of the atomic measures on  leaves through the endpoints of $\ap$\,:

In this paper, if stated, we allow closed leaves $\ell$ of $\lam$ to have weight infinity (\textit{heavy leaves}), i.e. if $\ap \in \mathring{\AA}(\lam)$ transversally intersects $\ell$, then $\mu(\ap) = \In$. 
A measured lamination is $\textit{heavy}$, if it has a leaf with weight infinity.

\subsection{Pleated surfaces}

A continuous map $\beta\cn \h^2 \to \h^3$ is a pleated plane if there exists a geodesic lamination $\lam$ on $\h^2$ such that 
\begin{itemize}
\item for  each stratum $P$ of $(\h^2, \lam)$, the map $\beta$ isometrically embeds $P$ into a (totally geodesic) copy of $\h^2$ in $\h^3$, and
\item $\beta$ preserves the length of (rectifiable) paths.
\end{itemize}
Then  we say that the geodesic lamination $\lam$ is \textit{realized} by the pleated surface $\beta$.
In this paper, we in addition assume that the realizing lamination is minimal, i.e. there is no proper sublimation of $\lam$ satisfying the two conditions above.

\begin{definition}[Total lift]
 Let $Y \to X$ be a covering map, and let $Z$ be a subset of $X$. Then the {total lift} of $Z$ to $Y$ is the inverse image of $Z$ by the map. 
\end{definition}

Then, suppose, in addition, that $\lam$ is the total lift of a geodesic lamination $\nu$ on a complete hyperbolic surface $\tau$. 
Let $F$ be the underlying topological surface of $\tau$.   
Then the $\pi_1(F)$-action on $\h^2$ preserves  $\lam$.
Let $\rho\cn \pi_1(F) \to \psl$ be a homomorphism.
Then the pleated surface $\beta\cn \h^2 \to \h^3$ is \textit{$\rho$-equivariant} if $\beta \cc \gam = \rho(\gam) \cc \beta$ for all $\gam \in \pi_1(F)$.
Then we say that the pair $(\tau, \nu)$ is \textit{realized} by the $\rho$-equivariant pleated surface $\beta$.

\begin{definition} 
Let $\psi\colon X \to Y$ be a map between metric spaces $(X, d_X)$ and $(Y, d_Y)$.
Then, for $\ep > 0$,  the map $\psi$ is an \ul{$\ep$-rough isometric} \ul{embedding} if   
$$d_Y(\psi(a), \psi(b)) -  \ep < d_X(a, b) <  d_Y(\psi(a), \psi(b)) + \ep,$$
for all $a, b \in X$.
Then $\psi$ is an $\ep$-rough isometry if, in addition, $Y$ is the $\ep$-neighborhood of the image of $\psi$.
\end{definition}

Two geodesic laminations are, in a sense,  ``close'' if they possibly intersect at angles very close to zero (see \S\ref{sAngle}).
Then, in the preceding paper, we proved that a please surfaces change a little when realizing laminations change a little. Namely 

 \begin{theorem}[\cite{Baba-15gt}, Theorem C]\Label{12-20}
 Suppose that there are a representation $\rho\cn \po(S) \to \psl$ and  
 a $\rho$-equivariant pleated surface $\beta_0\cn \h^2 \to \h^3$ realizing $(\sigma_0, \nu_0) \in \mathscr{T} \times \GL$.
 Then,
for every $\ep > 0$, there exists $\dl > 0$, such that if  there is another $\rho$-equivariant pleated surface $\beta\cn \h^2 \to \h^3$ realizing $(\sigma, \nu) \in \mathscr{T} \times \GL$ with $\angle_{\sigma_0}(\nu_0, \nu) < \dl$, then
 $\beta_0$ and $\beta$ are $\ep$-close: 
Namely there is a marking-preserving $\ep$-rough isometry $\psi\cn\sigma_0 \to \sigma$ such that, letting $\td{\psi} \cn \h^2 \to \h^2$ is the lift of $\psi$,  $~\beta_0$ and $\beta \cc \td{\psi}\cn \h^2 \to \h^3$ are $\ep$-close in the $C^0$-topology on $\h^2$ and,  moreover, in the $C^\In$-topology in the complement of the total lift to the  $\ep$-neighborhood of $|\nu| \cup |\nu_0|$ in $\sigma$. 
\end{theorem}

\subsection{Traintracks}\label{traintracks} (See \cite{Kapovich-01}. Also  \cite{Penner-Harer-92})
Given a rectangle $R$,  pick a pair of opposite edges as \textit{horizontal edges} and the other pair \textit{vertical edges}. 
A \textit{(fat) traintrack} $T$ is a collection  $ \{R_i\}_i$ of  rectangles, called \textit{branches},  embedded in a surface $F$ so that $R_i$ are disjoint except overlaps of their vertical edges in a particular manner:
Each vertical edge $e$ may contain at most finitely many points that are, on $F$, identified with some vertices of the rectangles $R_i$, and  those points divide $e$ into finitely many (sub)edges;  
After applying this decomposition to all vertical edges, all vertical edges are uniquely divided into pairs that are homeomorphically identified on $F$. 
The points dividing (original) vertical edges are called \textit{branch points} of $T$.
Let $|T| \st F$ denote the union of the branches $R_i$ over all $i$. 
Then the boundary of $|T|$ is the union of the horizontal edges of $R_i$, and it contains the branch points of $T$.
 In this paper, we assume that  traintracks are at most trivalent, i.e. for all $i$,  each vertical edge of  $R_i$ is a union of, at most, two other vertical edges. 

A lamination $\lam$ on $F$ is \textit{carried} by a traintrack $T$ if 
\begin{itemize}
\item $|\lam|$ is in the interior of $|T|$, 
\item leaves of  $\lam$ are transversal to the vertical edges of the branches of $T$, and  
\item if $R$ is a branch of $T$, then $R \cap \lam$ is a lamination on $R$ consisting of arcs property embedded in $R$ connecting the vertical edges of $R$\,;
\end{itemize} 
then we say that $T$ is a \textit{traintrack neighborhood} of $\lam$. 

In addition, suppose that the surface $F$ is a  hyperbolic surface and that the branches $R_i$ are smooth rectangles. 
Then the boundary of $|T|$ is the disjoint union of  piecewise-smooth curves, and its non-smooth points are endpoints of vertical edges. 
In particular,  the branch points are non-smooth points.

 For $\ep > 0$, a (smooth) traintrack $T = \{R_i\}$ on  $F$ is \textit{$\ep$-nearly straight}, if each rectangle $R_i$ is smoothly $(1 + \ep)$-bilipschitz to a Euclidean rectangle and at each  branch point, the angle of  the boundary curve of $|T|$ is  $\ep$-close to $0$. 
For $K > 0$,  a traintrack $T = \{R_i\}$ is \textit{$(\ep, K)$-nearly straight}, if in addition, a horizontal edge of each branch has length at least $K$. 
Note that for fixed $K > 0$, if $\ep > 0$ is sufficiently small, every  $(\ep, K)$-nearly straight traintrack is hausdorff close to a geodesic lamination.

(See also \cite{Baba-15gt}.) 
A {\it round circle}  is, by identifying $\rs$ with the unite sphere in $\R^3$, a circle which is the intersection of  $\rs$ with a hyperplane in $\R^3$.
A  \textit{round cylinder}  in the Riemann sphere $\rs$ is a cylinder  bounded by disjoint round circles. 
The \textit{axis} of a round cylinder $A$ in $\rs$ is the geodesic in $\h^3$ orthogonal to both hyperbolic planes bounded by the boundary circles of $A$. 
Then $A$ admits a canonical \textit{circular foliation} by one parameter family of round circles bounding disjoint hyperbolic planes  orthogonal to the axis of $A$. 

Let $C$ be a projective structure on a surface $F$. 
Suppose that $T = \{R_i\}$ is a smooth traintrack on $C$.
Then $C$ induces a projective structure on each rectangle $R_i$.
 Then $dev(R_i)$ is an immersion of $R_i$ to $\rs$, which is defined up to a postcomposition with an element of $\PSL$.
Then a branch $R_i$ of $T$ is a \textit{supported} on a round cylinder $A$ on $\rs$ if
\begin{itemize}
\item  $dev(R_i)$ maps into $A$,
\item  different vertical edges of $R_i$ immerse into different boundary circles of $A$, and
\item   the horizontal edges of $R_i$  immerses transversally to the circular foliation of $A$.
\end{itemize}
 
 A curve on a projective surface is {\it circular} if its lift (to the universal cover)  immerses into a round circle in $\rs$ by the developing map.
Then, the circular foliation of $A$ induces a foliation on $R_i$ by {\it circular arcs} connecting the horizontal edges.
Suppose that each branch $R_i$ of $T$ is supposed on a round cylinder. 
Then the circular foliations on the branches $R_i$ yield a foliation of $|T|$ by circular arcs.
Note that, if a loop $\ell$ is carried by $T$, then we can isotope $\ell$ through loops carried by $T$ so that $\ell$ is transversal to the circular foliation on $T$. 
Then 
\begin{lemma}\Label{123112}
Let $T$ be a traintrack on a projective surface $C$ such that the branches on $T$ are supported on round cylinders on $\rs$. 
Then, if a loop $\ell$ is carried by $T$ and $\ell$ is transversal to the circular foliation of $T$, then $\ell$ is admissible.   
(Lemma 7.2 in \cite{Baba-15gt}.)  
\end{lemma}

\subsection{Pants Graph}\Label{8-20-12no1}
(\cite{HatcherThurston80, Brock03}) 
Recall that $S$ is a closed oriented surface of genus $g \geq 2$. 
Then a \textit{ maximal multiloop} $M$ on $S$ is a multiloop such that  $S \minus M$ is a union of disjoint pairs of pants. 
Then $M$ consists of exactly $3 (g -1)$ non-parallel loops.

An \textit{elementary move} transforms a maximal multiloop $M$ to a different maximal multiloop by removing a loop $\ell$ of $M$ and adding another loop $m$ disjoint from the multiloop $M \minus \ell$ such that $m$ intersects $\ell$ minimally. 
Namely $m$ intersects $\ell$ in either one or two points. 
Then there is a unique connected component $F$ of $S$ minus $M \minus \ell$ such that $F$ contains $\ell$. 
Then either
\begin{itemize}
\item  $F$ is a one-holed torus, and $m$ intersects $\ell$ in a single point, or
\item  $F$ is a four-holed sphere, and $m$  intersects $\ell$ in two points. 
\end{itemize}
The \textit{pants graph} $\PG$ of $S$ is a one-dimensional simplicial complex whose vertices bijectively correspond to (the isotopy classes of) the maximal multiloops on $S$ and the edges to the elementary moves connecting different maximal multiloops.
Then it turns out that $\PG$ is connected (\cite{HatcherThurston80}).

\subsection{Purely loxodromic representations}\Label{purelyloxo} 

\begin{lemma}\Label{7-31-12no1}
Let $\rho\cn \pi_1(S) \to \psl$ be a purely loxodromic representation (see \S\ref{S:intro}). 
Then, if $\gam, \eta \in \pi_1(S)$ are non-commuting elements, the axes of the loxodromics $\rho(\gam), \rho(\eta)$ share no endpoint. 
\end{lemma}

\begin{proof}
Since $\gam$ and  $\eta$ do not commute, $\gam, \eta \neq id$ and $[\gam, \eta] \neq id$.
Suppose that, to the contrary,  the axes of $\rho(\gam)$ and $\rho(\eta)$ share an endpoint.  
Then we can show that their commuter  $[\rho(\gam), \rho(\eta)]$ is parabolic element, by computing its trace.
This is a contradiction since $\rho$ is purely loxodromic.  
\end{proof}

(See \cite[\S 4.3]{Kapovich-01} for example.)
Recall that $S$ is a closed oriented surface of genus at least two. 
Let  $\{\gm_1, \cdots, \gm_m\}$ be a generating set  of $\pi_1(S)$.
Then, by the adjoint representation, $\PSL$ embeds into $\on{GL}(3, \C)$ as a complex affine group. 
Since $\on{GL}(3, \C) \st \C^9$,  representations $\rho \cn\po(S) \to \on{GL}(3, \C)$ injectively correspond to tuples  $\{\rho(\gam_1), \dots, \rho(\gam_m)\}$ in $\PSL^m \st \C^{9m}$.
Thus  we can regard the space $\mathscr{R}$ of representations $\rho \cn\po(S) \to \on{GL}(3, \C)$  as an affine algebraic variety. 
This variety is called  the \textit{$\PSL$ representation variety} of $S$.

    \fontsize{12pt}{12pt}\selectfont
Then $\psl$ acts on $\mathscr{R}$ by conjugation, and its orbits give the equivalent classes of representations. 
By quotienting out $\mathscr{R}$ by a slightly stronger equivalent relation,  we obtain the $\PSL$-\textit{character variety} $\chi$ of $S$ (see \cite{BoyerZhang98, HeusenerPorti04} about the $\PSL$-character varieties and the quotient). 
It turns out that two representations $\rho_1, \rho_2\col \pi_1(S) \to \PSL$ are equivalent if and only if $tr^2(\rho_1 (\gamma)) = tr^2(\rho_2 (\gamma))$ for all $\gamma \in \pi_1(S)$ (see Theorem \cite[Theorem 1.3]{ HeusenerPorti04}).

Then $\chi$ has exactly two connected components (\cite{Goldman-88t}).
Let $\chi_0$ be the component consisting of representations $\pi_1(S) \to \PSL$ that lift to $\pi_1(S) \to \SL$.
Then $\chi_0$ contains the quasifuchsian space. 

For every $\gam \in \po(S)$, let $\Tr^2_\gam\cn \chi \to \C$ denote the square trace function of $\gam$ given by $\rho \to  \Tr^2(\rho(\gam))$.
Then  $\Tr^2_\gam$  is a regular function (see (\cite{BoyerZhang98}).
Note that the singular part of the $\chi$ has complex codimension at least one.
In addition the image  of the holonomy map $\Hol\cn \PP \to \chi$ is contained in the smooth part of $
\chi_0$. 
\begin{lemma}\Label{040213}
Almost all elements of $\chi_0$ are purely loxodromic. 
\end{lemma}
\begin{proof}
Since $\chi_0$ contains the quasifuchsian space,  if $\gam \in \pi_1(S) \minus \{id\}$, then $\Tr^2_\gm$ is nonconstant on $\chi_0$. 
If $\Tr^2_\gam(\rho) = 4$ if $\rho(\gam)$ is parabolic and $\Tr^2_\gam(\rho) \in [0, 4)$ if $\rho(\gam)$ is elliptic. 
Since $[0,4] \sub \R$ has measure zero in $\C$ and $\Tr^2_\gam$ is regular,  almost every element of $\chi_0$ takes $\gam$ to a loxodromic element. 
Since $\pi_1(S)$ contains only countably many elements, if $\rho$ is a generic representation in $\chi_0$, then $\rho(\gam)$ is loxodromic for all $\gam \in \pi_1(S)$. 
\end{proof}

\subsection{Local characterization of projective structures in $\GL(S)$.}\Label{sAngle}
Let $\tau$ be a hyperbolic surface homeomorphic to $S$.
If two geodesics $\ell$ and $m$ on $\tau$ intersect at a point $p$, then
let $\angle_p(\ell, m)$ denote the angle between $\ell$ and $m$ at $p$ that takes a value in $[0, \pi/2]$.
Let $\lam$ and $\nu$ are geodesic laminations on $\tau$. 
Then the \textit{angle} between $\lam$ and $\nu$ is 
 \begin{displaymath}
 \sup \angle_p(\ell_p, m_p),
\end{displaymath} 
where the supremum runs over all points $p \in |\lam| \cap |\nu|$ and $\ell_p$ and $m_p$ are the leaves of $\lam$ and $\nu$, respectively, intersecting at $p$. 
If $\lam$ and $\nu$ are  laminations on the topological surface $S$ or  different hyperbolic surfaces homeomorphic to $S$, then  $\angle_\tau(\lam, \nu)$ is given by taking their geodesic representatives on $\tau$.

Let $\nu$ be a geodesic lamination on $\tau$, and let $(\lam_i)$ be a sequence of geodesic laminations on $\tau$. 
Suppose that $\angle_\tau(\lam_i, \nu) \to 0$ as $i \to \In$. 
Note that this convergence is independent on the choice of  the hyperbolic structure $\tau \in \mathscr{T}$.
Since, in this paper, we typically require such an angle to be sufficiently small,  we may denote $\angle_\tau(\lambda_i, \nu)$ simply by $\angle(\lambda_i, \nu)$ without specifying $\tau$. 
If $\sigma$ is a subsurface of $\tau$, then  let $\angle_\sigma(\lambda, \nu)$ be 
$\sup \angle_p(\ell, m)$ over all leaves $\ell \in \lam$ and $m \in \nu$ intersecting at points $p$ contained in $\sigma$.

Given measured geodesic laminations  $M$ and $L$ on $(S, \tau)$, their angle $\angle_\tau(M, L)$ is the angle of the under lying geodesic laminations $|M|$ and $|L|$.

\begin{theorem}[\cite{Baba-15gt}, Theorem B]\Label{8-14-12no1}
Let $C \cong(\tau, L)$ be a projective structure on $S$ with holonomy $\rho \col \pi_1(S) \to \PSL$.
Then there is $\del > 0$ such that, if  another projective structure $C' \cong (\tau', L')$ with holonomy $\rho$ satisfies $\angle_\tau(L, L') < \del$, then we can graft $C$ and $C'$ along multiloops to a common projective structure. 
That is, there are admissible multiloops $M$ on $C$ and $M'$ on $C'$ such that 
\begin{displaymath}
\Gr_M(C) \cong \Gr_{M'}(C').
\end{displaymath}
\end{theorem}

\section{Thurston's grafting coordinates  on $\PP$}\Label{thurston}
(\cite{Kamishima-Tan-92, Kullkani-Pinkall-94} are general references of this section.)
The space $\PP$ of all (marked) projective structures on  $S$ is naturally homeomorphic to the product of   the Teichm\"uller space $\mathscr{T}$ of $S$ and  the space of measured laminations $\ML$ on $S$:
\begin{eqnarray} 
\PP \cong \mathscr{T} \times \ML \label{040813}
\end{eqnarray}

Let  $C  = (f, \rho) \in \PP$, and  let $(\tau, L) \in  \mathscr{T} \times \ML$ be its Thurston coordinates, i.e. the corresponding pair via (\ref{040813}).
We briefly describe the correspondence between $\PP$ and $\mathscr{T} \times \ML$.
First there is a measured lamination $\LL = (\nu, \omega)$ on $C$, called the \textit{(canonical) circular lamination}, and a marking-preserving continuous map 
\begin{displaymath}
\kap \cn C \to \tau,
\end{displaymath}
 called  \textit{collapsing map}, such that $\kap$ descends $\LL$  to $L$. 
  The leaves of $\LL$ are circular and they have no atomic measure, in comparison to $L$. 
 The pair $(\tau, L)$ corresponds to a pleated surface $\beta\col \h^2 \to \h^3$ equivariant with respect to $\rho$, constructed as follows.  
Let $\til{L}$ be the total lift of $L$ under the covering map $\h^2 \to \tau$. 
Then $\til{L}$ is the $\pi_1(S)$-invariant measured lamination on $\H^2$.
Thus, intuitively speaking by bending $\h^2$ along $\til{L}$ by the angle given by the transversal measure of $L$, $(\tau, L)$ yields a pleated surface $\beta\col \h^2 \to \h^3$. 
(To be precise, there is  a sequences $(L_i)$ of measured laminations with finitely many leaves that converges to $\til{L}$ uniformly on compacts in $\h^2$. Then the pleated surface $\beta$ is given by the limit of pleated surfaces $\beta_i\col \h^2 \to \h^3$ corresponding to $L_i$; see \cite[3.11.6]{Epstein-Marden-87}).

Note that $f\col \til{C} \to \rs$ and $\beta\col \h^2 \to \h^3$ are both $\rho$-equivariant and $\rs$ is the ideal boundary of $\H^3$.
In fact, for every $x \in \til{C}$, its image under  $f\col \til{C} \to \rs$ maps to its image under $\beta \cc \kap \col \til{C} \to \h^3$ by a certain nearest point projection given by the {\it maximal ball} in $\til{C}$ associated with $x$ (see \S \ref{sMaximalBalls} for the precise correspondence).

The collapsing map $\kap$ takes each stratum of $(C, \LL)$ diffeomorphically onto a stratum of $(\tau, L)$.
If $\ell$ is a closed leaves of $L$, which  carries positive atomic measure, then $\kap^{-1}(\ell)$ is a cylinder $\AA_\ell$ foliated by closed leaves of $\LL$.
Conversely $\kap$ takes each closed leaf of $\LL$ on $\AA_\ell$ onto $\ell$  diffeomorphically. 
Let $\AA$ be the union of the disjoint cylinders $\AA_\ell$ over all closed leaves $\ell$ of $L$.
Then, on the other hand, the strata of $(C, \LL)$ not in $\AA$ bijectively correspond,  via $\kap$,  to the strata of  $(\tau, L)$ that are not the closed leaves of $L$.
For different closed leaves $\ell$ of $L$, their corresponding $\AA_\ell$ are disjoint on $C$. 


Recall that $\kap \col C \to \tau$ preserves its marking (and thus  $\kappa$ is homotopic to a homeomorphism).
Then, given a cover of $C$, its Thurston coordinates are given by the corresponding cover of $(\tau, L)$. 
In particular, the universal cover of $C$ is a projective structure on an open disk, and its Thurston coordinates are $\h^2$  and the total lift of $L$ to $\h^2$.

More generally,
we say that  a  projective structure $C = (f, \rho)$ on a connected orientable  surface $F$ {\it has Thurston coordinates} $(X, L)$, where the universal cover $\til{X}$ of $X$ is a convex subset of $\h^2$ bounded by geodesics and a measured lamination $L$ on $X$, 
 if the maximal balls in the universal cover $\til{C}$ yields a $\pi_1(C)$-invariant a stratification of $C$ and it descends, by the construction in \S \ref{sMaximalBalls},  to a $\rho$-equivariant pleated surface from $\til{X} \to \H^3$ given by $(X, L)$.

Indeed, a  projective structure on an open disk has Thurston coordinates unless it is isomorphic to $\C$ as a projective surface (see \S \ref{BendingDisk}).
However its first coordinate is not necessarily the entire hyperbolic space. 
Let $X$ be a convex subset of $\h^2$ bounded by disjoint geodesics.
Note that $X$ can be  the entire hyperbolic  plane or a single (biinfinite) geodesic.
In addition, we suppose that  each boundary geodesic of $X$ is either a subset of $X$ or its complement $\h^2  \minus X$. 
Let $L = (\lam, \mu)$ be a measured lamination on $X$.
Then $L$ induces a pleated surface $\beta\cn X \to \h^3$ by bending $X$, inside $\h^3$, which is unique up to a post-composition with an element of $\psl$. 

Let $(X, L)$ be the Thurston coordinates of  a projective structure $C$  of a open disk. 
Then, if $X$ has a boundary geodesic $\ell$, the transversal measure of $L$ must be infinite near $\ell$.
More precisely,
\begin{itemize}
\item if $\ell$ is a subset of $X$, then $\ell$ has weight infinity, and 
\item if $\ell$ is a subset of $\h^2 \minus X$, then the transversal measure of $L$ is infinite ``near $\ell$", i.e.  if an arc $\ap$ on $X$ is  transversal to  $L$ and it has an open end point at  $\ell$, then  $\mu(\ap)$ is infinite.  
\end{itemize}


\subsection{Maximal balls}\Label{sMaximalBalls}
Let $C$ be a projective structure on an open disk. 
Let $f\cn C \to \rs$ be  $dev(C)$.
Conformally identifying $\rs$ with $\s^2$, we fix a spherical metric on $\rs$, which is unique up to an element of $\psl$. 
Pullback this metric to $C$ by $f$ and obtain an incomplete spherical metric on $C$.
The metric completion of $C$ minus $C$ is called the \textit{ideal boundary} of $C$ and denoted by $\pt_\In C$.
Note that this completion $C \cup \bd_\infi C$  is (topologically) independent of the choice of the spherical metric on $\rs$.

A \textit{maximal ball} in $C$ is a topological open ball $B$ such that 
\begin{itemize}
\item $B$ is round, i.e.  $f$ embeds $B$ onto a round open ball in $\rs$, and
\item $B$ is maximal, i.e. there is no round open ball in $C$ strictly containing $B$.
\end{itemize}

The \textit{ideal boundary} $\pt_\infi B$ of a maximal ball $B$ in $C$ is the intersection of $\pt_\In C$ and $\pt B$ in $C \cup \pt_\In C$.
Then, by identifying $B$ with $\h^2$ conformally, the ideal boundary of $B$ is a subset of the ideal boundary of $\h^2$ (which is $\s^1$). 
The \textit{core} $\Core(B)$ of a maximal ball $B$ in $C$ is  the convex hull of the ideal boundary of $B$ in $\h^2$. 
It turns out that $\Core(B)$ is a stratum of $(C, \LL)$.  
In particular, for different maximal balls $B$ in $C$, their corresponding cores are disjoint \cite[Proposition 4.3]{Kullkani-Pinkall-94}. 
Moreover, taking the cores of  all maximal balls $B$, we obtain the stratification of $(C, \LL)$. 
In other words, for every point $p \in C$, there is a unique maximal ball $B$ in $C$ such that $p \in \Core(B)$ \cite[Proposition 4.4]{Kullkani-Pinkall-94}.
Then we say that $B$ is the maximal ball \textit{centered} at $p$.

Let $H$ be the hyperbolic plane in $\H^3$ bounded by the boundary circle of $B$. 
Then the nearest point projection from $\H^3$ to $H$ extend to $p$. 
 Let $\beta\col X \to \H^3$ be the pleated surface for $C$, and $\kap \col C \to X$. 
Then $\beta \cc \til{\kap} (p)$ is the projection of $p$ to $H$, where $\til{\kap}\col \til{C} \to \til{X}$ is the lift (\S 8). \tcb{\cite{Kullkani-Pinkall-94}}

\subsubsection{Thurston metric} (\cite{Kullkani-Pinkall-94, Tanigawa-97})
Every projective structure $C$ on a surface, unless its universal cover is $\C$, admits a canonical $C^1$-smooth Riemannian metric, called {\it Thurston metric}. 
It is given by a $\pi_1(S)$-invariant Riemannian metric on its universal  $\til{C}$ defined as follows.
Let  $x \in \til{C}$. 
For every maximal ball $B$ in $\til{C}$ containing $x$, by conformally identifying $B$ with $\h^2$, it defines a Riemannian metric tensor at $x$.
Taking the infimum of the metric tensors over all maximal ball $B$ containing $x$, we obtain Thurston metric at $x$. 

Let $C \cong (\tau, L)$ be the Thurston coordinates.
The Thurston metric is isometric to the hyperbolic metric on $\tau$ by $\kap$ on each stratum of $(C, \LL)$.
 For each closed leaf $\ell$ of $L$,  the Thurston metric on the cylinder $\AA_\ell$ is Euclidean, so that each leaf of $\LL$ in $\AA_\ell$ is a closed geodesic whose length is  $\length_\tau(\ell)$ and the height of the cylinder is twice as much as the weight of $\ell$ given by $L$. 
We call $\AA$ the {\it Euclidean region} of $C$.
The Thurston metric changes continuously in the deformation space $\PP$ of projective structures on $S$.

\subsection{Existence of Thurston coordinates on disks}\Label{BendingDisk}
\cite[Theorem 11.6]{Kullkani-Pinkall-94} implies
\begin{theorem}\Label{6-3no1}
Let $C$ be a projective structure on a simply connected surface not isomorphic to $C \neq \C, \rs$ as a projective surface. 
Then $C$ admits unique Thurston coordinates $(X, L)$ such that 
\begin{itemize}
\item  $X$ is a closed convex subset of $\h^2$ bounded by geodesics and  each boundary geodesic  of $X$ is either contained in $X$ ({\rm closed boundary}) or in $\h^2 \minus X$ ({\rm open boundary}), and
\item  $L$ is a measured lamination on $X$, and if a geodesic boundary of $X$ is contained in $X$, then it is a leaf of $L$ with weight $\infi$ (and no leaves in the interior of $X$ have weight infinity).
\end{itemize}
\end{theorem}

\begin{remark}
The boundary leaf with weight infinity corresponds to a complex affine half infinite cylinder in the projective surface, and this cylinder  is foliated by round circles which descend to the boundary leaf.
\end{remark}

\begin{corollary}\Label{6-3no2}
Let $F$ be a connected surface (possibly with open boundary).
Let $C$ be a projective structure on $F$ with holonomy $\rho\cn \po(F)  \to \psl$. 
Suppose that  $\Im \rho$ is non-elementary. 
Then  $C$ has  Thurston coordinates 
$$C \cong (\tau, L),$$
where $\tau$ is a convex hyperbolic surface possibly with geodesic boundary such that the interior of $\tau$ is homeomorphic to $F$ and $L$ is a measured geodesic lamination on $\tau$. 
In addition each boundary component of $\tau$ is either open or closed;  therefore the closed boundary components of $\tau$ are the only leaves of $L$  with weight infinity. 
\end{corollary}
\begin{remark}
If $\Im \rho$ is elementary but the limit set of $\Im \rho$  has cardinality two, then $C$ still has  Thurston coordinates $(\tau, L)$.
However $\tau$ may be a single close geodesic with weight infinity, and the interior of $\tau$ is not homeomorphic to $F$.
Nonetheless  the regular neighborhood of $\tau$ is still homeomorphic to $F$.
\end{remark}

\proof[Proof of Corollary \ref{6-3no2}]
Since the limit set of $\Im(\rho)$ has cardinality more than one,  the universal cover of $C$ can not be isomorphic to  $\C$ or $\rs$ (as a projective structure).
Thus applying Theorem \ref{6-3no1}, we obtain the Thurston coordinates $(\td{\tau}, \td{L})$ of  the universal cover of $C$ so that, if exists, the boundary geodesics of $\td{\tau}$ are the only leaves of $L$ with weight infinity.
Then $\pi_1(F)$ acts on $\td{\tau}$.
Then $\td{\tau}$ is the convex hull of the limit set of $\rho(\pi_1(F))$. 
Thus, since $\Im(\rho)$ is non-elementary,  $\td{\tau}$ has interior whose closure is $\td{\tau}$.   
Since $\po(F)$ preserves $(\td{\tau}, \td{L})$, it descends to the Thurston coordinates $(\tau, L)$ of $C$.
Then the interior of $\tau$ is homeomorphic to $F$.
\Qed{6-3no2}

\subsection{Canonical neighborhoods}\Label{5-4no1}

Suppose that $C$ is a projective structure on an open disk with $C \neq \C$.
Then let $(X, L)$ denote its Thurston coordinates, where $X$ is a convex subset of $\h^2$ bounded by geodesics and $L$ is a (possibly heavy) measured lamination on $X$ (Proposition \ref{6-3no1}).
Let $\beta\cn X \to \h^3$ be the corresponding pleated surface and  $\kp\cn C \to X$ be the collapsing map. 
Let $\mL$ be the measured lamination on $C$ that descends to $L$ by $\kp$.

Let $p$ be a point on $C$, and
let $B(p)$ be the maximal ball in $C$ centered at $p$.
Let $U(p)$ denote the union of all maximal balls in $C$ containing $p$.  
Then $U(p)$ is an open neighborhood of $p$, and it is called the \textit{canonical neighborhood} of $p$ in $C$.
It turns out that $U(p)$ is homeomorphic to an open disk and $dev(C)$ embeds $U(p)$  into $\rs$ (see \cite{Kullkani-Pinkall-94} \cite{Kamishima-Tan-92}).
Note, since $C \neq \C$, thus $U(p) \neq \C$.  
The ideal boundary $\bd_\infi C$ intersects the closure of $U(p)$ in the completion $C \cup \partial_\infty C$, and points in the intersection are called 
 {\it ideal points} of $U(p)$.

Recall that $C$ decomposes into strata by $\LL$, which are leaves of $\LL$ and closures of the complementary regions of $C \minus |\LL|$.
Then we can take a quotient $T$ of $C$ by collapsing each stratum to a point. 
Let  $\Psi\cn C \to T$ be the quotient map.
Then, for two strata $P, Q$ of $(C, \LL)$,  the distance between  $\Psi(P)$ and $\Psi(Q)$ on $T$ is  the infimum of the  measures, given by $\LL$, over all transversal arcs connecting $P$ and $Q$ in $C$. 
 Then, as $C$ is a disk,  it is well-known that the quotient $T$ is a metric $\R$-tree (see, for example, \cite[\S11.12]{Kapovich-01}).

\begingroup

 Let $W(p)$ be the union of $\Core(B(x))$ for all $x \in C$ with $p \in  B(x)$.
Then $W(p)$ is the open neighborhood of $p$ bounded by  the leaves $\ell$ of $\LL$ such that  their corresponding maximal balls $B_\ell$ satisfy $\pt B_\ell \ni p$ (Figure \ref{f:4-23no3}). 
 If $B_1$ and $B_2$ are different maximal balls  in $C$, then $B_1$ intersects exactly one connected component of $C \minus B_2$\,; 
	this implies $W(p)$ is connected.

\begin{figure}[h]
\begin{overpic}[scale=.3
]{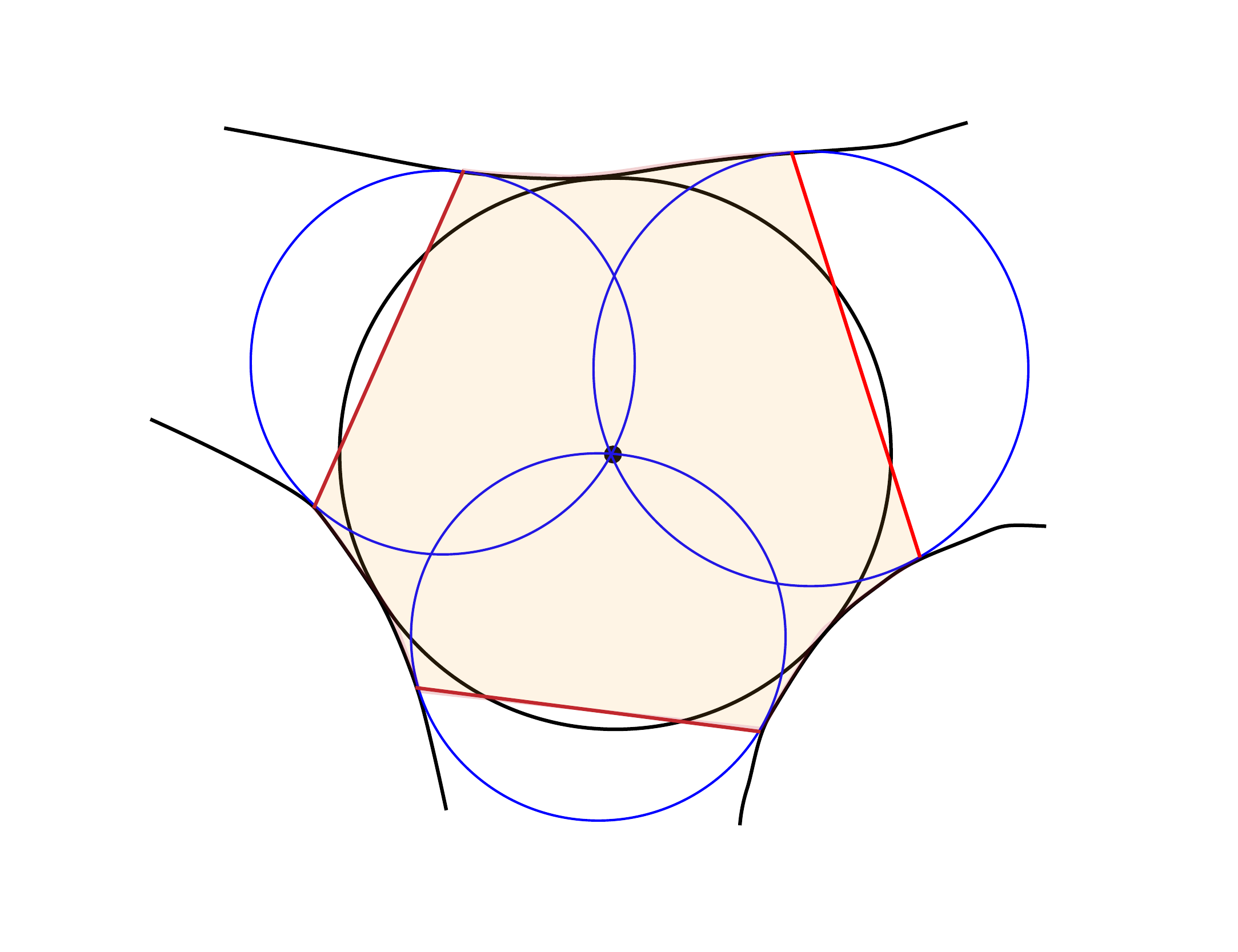}
\put(48,22){$\textcolor{BurntOrange}{W(p)}$}
\put(52,41){$p$}
      \end{overpic}
\caption{}\label{f:4-23no3}
\end{figure}

 $U(P) \minus W(p)$ are disjoint half  disks, which are cores of the Thurston coordinates of $U(P)$.
 However, typically those half disks are not strata of Thurston coordinates of $C$ via the inclusion $U(P) \subset C$. 
   \endgroup

\begin{lemma}\Label{5-31no1}
For  $p \in C$, if a neighborhood $V_p$ of $\Psi(p)$ in $T$ is contained in $\Psi(W(p))$,
then the ideal boundary $\pt_\If B(p)$ is  the boundary circle  $\pt B(p)$ minus the union of maximal balls of $C$ whose  cores map into  $V_p$ by $\Psi$. 
\end{lemma}

\begin{proof}
If $x \in W(p)$, then $B(x)$ is contained in $U(p)$. 
In particular,   if the neighborhood $V_p$ of $p$ is contained in $ \Psi(W_p)$, 
for all  $x \in C$ with $\Psi(x) \in V_p$,   $B(x)$ is contained in $U(p)$.
 By the definition of $U(p)$,  the maximal ball $B(x)$ contains $p$. 

The ideal boundary $\pt_\In B(p)$ is naturally embedded  in the boundary of $U(p)$ in $\rs$.
Therefore,  since   $B(x)$ is in the interior of $U(p)$, the ideal boundary $\pt_\If B(p)$ is contained in $\pt B(p) \sm \cup_x B(x)$ over all $x \in C$ with $\Psi(x) \in V_p$. 

To show the opposite inclusion, 
let $s$ be the connected component of $\pt B(p) \sm \pt_\If B(p)$;  see Figure \ref{6-1no1}.
Then $s$ is a circular arc on $\rs$ with open ends. 
Then there is  a unique leaf $\ell$ of $\LL$ connecting the endpoints of $s$.
 Namely $\ell$ is a boundary leaf  of $\Core B(p)$.

Consider the connected  component of  $C \minus \Core B(x)$ bounded by $\ell$.
Pick  a sequence  of points $x_i$ in the component limiting to an interior point of $\ell$.
Then $B(x_i)$ converges to $B(p)$ as $i \to \In$, and $B(x_i) \cap s$ converges to $s$. 
Since $\Psi$ takes $\Core B(x_i)$ to a point in $V_p$ for sufficiently large $i$,
	we have $s \st \cup_x B(x)$ over $x \in C$ with $\Psi(x) \in V_p$.
\begin{figure}[h]
            \begin{overpic}[scale=.3
]{6-1no1.pdf}
\put(61,42){$\tcr{\ell}$}
\put(78,50){$s$}
\put(31,27){\small $\Core B(x)$}
\put(42, 14){$B(x)$}
\put(10, 50){$C$}
      \end{overpic}
\caption{}\label{6-1no1}
\end{figure}
\end{proof}
Let $(\h^2, L_p)$ be the Thurston coordinates of $U(p)$.
As $U(p)$ is embedded in $\rs$, its Thurston coordinates correspond to the boundary of the convex hull of $\rs \minus U(p)$.
In particular the first coordinate is the entire hyperbolic plane since $U(p)$ is embedded in $\rs$.

Let $\beta_p\cn \h^2 \to \h^3$ be the pleated surface corresponding to  $(\h^2, L_p)$.  
Let $\mL_p$ be the circular measured lamination on $U(p)$ that descends to $L_p$ by the collapsing map $\kp_p\cn U(p) \to \h^2$.

Note that  $W(p) \st U(p) \st C$. 
Then we next show that
the Thurston coordinates of $U(p)$ on $W(0)$ coincide with that of $C$ (which typically fails on $W(p) \minus U(p)$).  Namely
\begin{proposition}\Label{4-23no3}
\begin{itemize}
\item In $W(p)$,   $\mL_p$ is isomorphic to  $ \mL$, and thus the Thurston metric on $W(p)$ is isometric to that on $C$.
\item There exists a natural isometry $\psi\cn \kp_p(W) \to \kp(W)$  such that $\psi \cc \kp_p = \kp$ on $W$ and $\beta \cc \psi = \beta_p$ on $\kp_p(W)$. 
\[
\xymatrix{ 
{{    (U_p, \mL_p)  }} \supset W\, \ar[d]^{\kp_p} \ar@{^{(}->}[r]^{} &  (C,L)  \ar[d]^{{\kp}} \\
  (\h^2, L_p) \ar[r]^{{\psi  }} \ar[rd]^{\beta_p} \supset \kp_p(W)	&   \kp(W) \subset  {{(X, L) }} \ar[d]^{\beta}  \\  
& \h^3}
   \]
\end{itemize}
\end{proposition}
\proof
Recall that $U(p) = \cup_x B(x)$ where $x$ runs over all points  $C$ with  $p \in B(x)$.
Since $U(p) \sup C$,  such a maximal ball $B(x)$ in $C$ is also maximal in $U(p)$.
Since $W(p)$ is connected and it contains no boundary leaves  (Figure \ref{f:4-23no3}), $\Psi(W(p))$ is an open  connected subset of $T$. 
Therefore, by Lemma \ref{5-31no1}, if $q \in W(p)$, the ideal boundary of the maximal ball $B(q)$ in $C$ is  equal to that  in $U(p)$.
Since the maximal balls and their ideal boundary determine the circular laminations,  $\mL$ is isomorphic to $\mL_p$ on $W(p)$ by the inclusion $U_p \st C$.
The second assertion similarly holds. 

\Qed{4-23no3}


\part{Grafting Conjecture for purely loxodromic holonomy}

\section{Sequence of pleated surfaces}\Label{5-14}
Fix an arbitrary representation $\rho\cn \po(S) \to \psl$ that is  purely loxodromic.
 Let $C_\sharp  \cong (\tau_\sharp , L_\sharp)$ and $\Cc \cong (\tau_\flat, L_\flat)$ be projective structures with holonomy $\rho$.
Then they correspond to $\rho$-equivariant pleated surfaces realizing  $(\tau_\sharp , L_\sharp)$ and $(\tau_\flat, L_\flat)$. 
In this section, we construct an infinite family of pleated surfaces, in a coarse sense, ``connecting'' those pleated surfaces corresponding to $C_\sharp$ and $C_\flat$. 

Pleated surfaces are invented by William Thurston for the study of three-dimensional hyperbolic manifolds, and in particular he used a sequence of homotopy equivalent pleated surfaces in order to understand geometry of the convex hull of the manifolds (see \cite[Chapter 9]{Thurston-78}), and it was been widely used (for example see \cite{Brock03} \cite{Minsky99}).

The family of pleated surfaces  in this paper is motivated by the study hyperbolic 3-manifolds, but 
 on the other hand, the homomorphisms $\rho\col \pi_1(S) \to \PSL$ of our interest here are  not necessarily discrete or faithful. 
In order to adapt the theory of pleated surfaces for hyperbolic three-manifolds,  we carefully use the assumption of $\rho$ being purely loxodromic.

\begin{lemma}
Let $L$ be a measured (geodesic) lamination on a hyperbolic surface $\tau$. 
For every $\ep > 0$, there is a neighborhood $U$ of $[L]$ in $\PML(S)$, such that if $L' \in U$ then $\angle_\tau (L, L') < \ep$.

\end{lemma}   
\begin{proof}
Suppose, to the contrary, that there is a sequence of measured laminations $L_i$ converging $L$ as $i \to \infi$, but $\limsup_{i \to \infi} \angle_\tau (L, L_i) > 0$. 
The space of geodesic laminations on $\tau$ is compact. 
Thus, up to a subsequence,  the underlying geodesic laminations $|L_i|$ converge, as $i \to \infi$, to a geodesic lamination $\lambda_\infi$ which contains a leaf transversally intersecting a leaf of $L$. 
This contradicts to the assumption $L_i \to L$.
\end{proof}

 \begin{theorem}[\cite{FLP79}]
In the space of measured laminations $\ML(S)$, a weighted loop is dense.  
\end{theorem}   

By this theorem,
    we can pick maximal multiloops $M_\sharp$ and $M_\flat$ on $C_\sharp $ and $\Cc$, respectively, so that
\begin{itemize}
\item  $\angle_{\tau_\sharp}(M_\sharp,L_\sharp)$ and $\angle_{\tau_\flat}(M_\flat, L_\flat)$ are sufficiently small, and  
\item  sufficiently small neighborhoods of $M_\sharp$ and $M_\flat$ contain $L_\sharp$  on $\tau_\sharp$  and $L_\flat$ and $\tau_\flat$, respectively.  
\end{itemize}
Since the pants graph of $S$ is connected  (\S \ref{8-20-12no1}), there is  a simplicial path in the graph connecting  $M_\sharp$ and $M_\flat$. 
Let $(M_i)_{i =0}^n$ be the corresponding sequence of maximal multiloops on $S$ with $M_0 = M_\sharp$ and $M_n = M_\flat$, so that $M_i$ and $M_{i +1}$ are adjacent vertices of the pants graph for all $i = 0, \dots, n-1$. 

Each connected component $P$ of $S \sm M_i$ is a pair of pants. 
Pick a maximal geodesic lamination on $P$. 
Then it consists of three isolated geodesics,  and each geodesic ray in the lamination (\textit{half-leaf}) is asymptomatic to a boundary component of $P$, spiraling towards it. 
We can, in addition, assume that such half-leaves spiral towards left (with respect to the orientation of $S$) when they approach towards boundary components and that, on each leaf of the lamination, the rays in the opposite directions are asymptotic to different boundary components of $P$. 
For each $i$, let $\nu_i$ be the maximal lamination of $S$ that is the union of the maximal multiloop $M_i$ and the above maximal laminations on all connected components $P$ of $S \sm M_i$. 

The lamination $\nu_i$ is obtained as the Hausdorff limit of the iteration of the left Dehn twist along $M_i$ of some multiloop $N_i$ on $S$ (Figure \ref{f_nu_i}).
Indeed we can take the multiloop $N_i$ so  that the restriction of $N_i$ to each connected component  $P$ 
of $S \sm M_i$ is a union of three  non-parallel arcs connecting all pairs of boundary components of $P$. 
Furthermore, for every $k \in \Z_{> 0}$, by taking $k$ parallel copies of the arcs on all $P$, we can also take $N_i$ such that the number of the arcs of  $N_i | P$ is  3$k$ for all connected components $P$ of $S \sm M_i$.

 \begingroup
 \color{blue}
    \fontsize{12pt}{12pt}\selectfont
\begin{figure}
\begin{overpic}[scale=.5
] {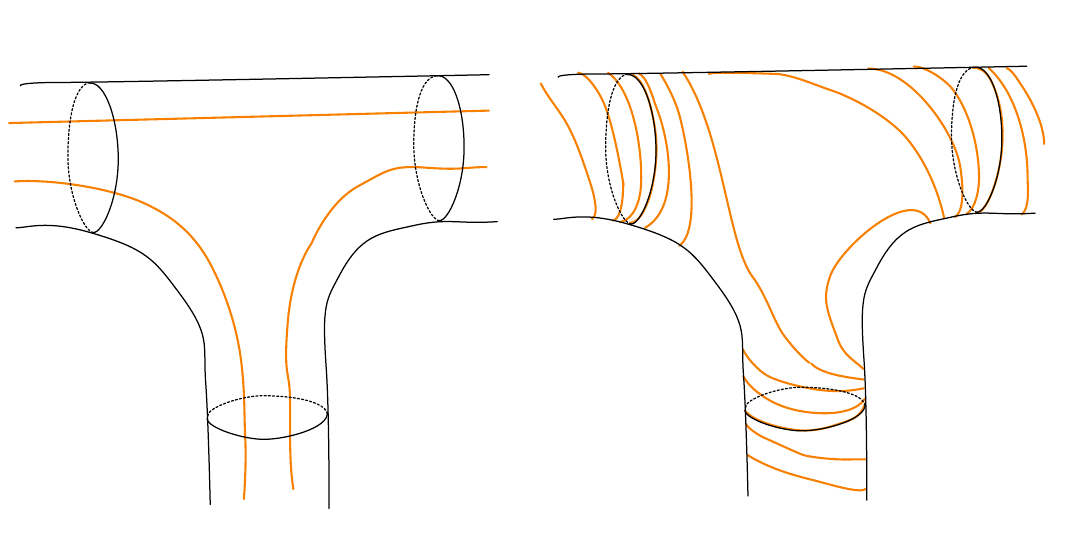} 
 \put(40 , 45){\textcolor{Black}{$M_i$}}  
 \put(20 , 28){\textcolor{Orange}{$N_i$}}  
 \put(70 , 28){\textcolor{Orange}{$\nu_i$}}  
      \end{overpic}
\caption{The maximal lamination $\nu_i$ is obtained by twisting $N_i$ along $M_i$ ``infinitely many'' times. }\label{f_nu_i}
\end{figure}

   \endgroup
 
 The following lemma guarantees that  $\nu_i$ is realized by a unique $\rho$-equivariant pleated surface.
\begin{lemma}\Label{12-18no1} 
Suppose that $\rho \cn \po(S) \to \psl$ is purely loxodromic. 
Let $\nu$ be a geodesic lamination on $S$ such that 
\begin{itemize}
\item $\nu$ is maximal, and 
\item every half-leaf of $\nu$ accumulates to a closed leaf of $\nu$.
\end{itemize}
Then there is a unique $\rho$-equivariant pleated surface realizing $\nu$. 
\end{lemma}

\begin{proof} 
Let $\Delta$ be a connected component of $S \minus |\nu|$, which is an ideal triangle.
Let $\td{S}$ be the universal cover of $S$, and
	let $\td{\nu}$ be the total lift of $\nu$ to $\td{S}$.
Let $\td{\Delta}$ be a lift of $\Delta$ to $\td{S}$.
Then $\td{\Delta}$ is an ideal triangle property embedded in $\td{S}$, and its vertices are at distinct points  on the circle at infinity $\pt_\infi S$. 
 
By the second assumption, each (ideal) vertex of $\Delta$ corresponds to a closed leaf of $\nu$. 
This loop lifts to a unique leaf of $\td{\nu}$ whose endpoint is the corresponding vertex of $\td{\Delta}$ in the boundary circle of $\td{S}$ at infinity. 
Different vertices of $\tilde{\Delta}$ correspond to different leaves of $\td{\nu}$ that cover closed leaves of $\nu$.    
Then the vertices are naturally fixed points of different elements of $\po(S) \minus \{id\}$. 
Since $\rho$ is purely loxodromic,  by Lemma \ref{7-31-12no1}, for different elements of $\po(S) \minus \{id\}$, their $\rho$-images are loxodromics sharing no fixed points. 
Then the vertices of $\til{\Delta}$ correspond to different points on $\rs$ fixed by different loxodromics, and they spans a unique ideal triangle in $\h^3$.

This correspondence defines a $\rho$-equivariant map $\beta$ from  $\td{S} \minus |\td{\nu}|$ to $\h^3$.
Note that every geodesic lamination is uniquely decomposed into isolated biinfinite leaves, closed leaves, and minimal irrational laminations (see \cite[I.4.2]{CanaryEpsteinGreen84}). 
Then,  if a leaf of $\til{\nu}$ is an isolated leaf,  then either it separates adjacent complementary ideal triangles or it covers a closed leaf of $\nu$ by the second assumption and the decomposition theorem.  
Each leaf $\ell$ of $\td{\nu}$  either separates adjacent ideal triangles or  descends to a closed leaf of $\nu$ on $S$.
Clearly the $\rho$-equivariant map continuously extends to the leaves of $\td{\nu}$ of the first type.  
If a leaf $\ell$ cover a closed leaf of $\nu$, then  there is a sequence $\{\Delta_i\}$ of ideal triangles of $\td{S} \minus |\td{\nu}|$ that converges to $\ell$ uniformly on compacts (in the Hausdorff topology).
Then, by the second assumption, if  $i \in \N$ is sufficiently large, a vertex of $\Delta_i$ must coincide with an endpoint of $\ell$.
Noting that  $\td{S} \minus |\td{\nu}|$ has only finitely many components up to $\pi_1(S)$, since $\beta$ is $\rho$-equivariant, $\beta(\Delta_i)$  must converge to the geodesic axis of the loxodromic corresponding to $\ell$.
Therefore we can continuously extend $\beta$ to the leaves of $\til{\nu}$ covering closed leaves of $\nu$ and obtain a desired a $\rho$-equivariant pleated surface $\h^2 \to \h^3$ realizing $\lam$.
\end{proof}

\subsection{Bi-infinite Sequence of geodesic laminations connecting $\nu_i$ to $\nu_{i+1}$} 
Recall that  for each $i \in \{ 0, \cdots, n -1\}$, $M_i$ and $M_{i + 1}$ are maximal multiloops on $S$ that are adjacent vertices on the pants graph of $S$.
Then let $m_i$ and $m_{i +1}$ be the loops of $M_i$ and $M_{i + 1}$, respectively,  such that $M_i \minus m_i = M_{i + 1} \minus m_{i +1}$. 
Then let $F_i$ denote the minimal subsurface of $S$ containing both  $m_i$ and $m_{i +1}$, which is either a one-holed torus or a four-holed sphere (\S \ref{8-20-12no1}). 
  
{\it Case One.} First suppose that $F_i$ is a once-holed torus.
Let $\hat{F}_i$ be the once-punctured torus obtained by pinching the boundary component of $F_i$ to a point. 
Then, every geodesic lamination on $F_i$ descends a unique geodesic  lamination on $\hat{F}_i$.
In particular, the geodesic laminations  $\nu_i$ and $\nu_{i+1}$ on $S$ restrict to geodesic laminations on $F$, then further to   
unique laminations $\hat{\nu}_i$ and $\hat{\nu}_{i+1}$, respectively, on  $\hat{F}_i$.
Since $\nu_i$ and $\nu_{i + 1}$ are maximal,    $\hat{\nu}_i$ and $\hat{\nu}_{i+1}$ are maximal on $\hat{F}_i$. 
 ( By pinching a boundary component to a point,  we can eliminate the twisting direction $v_i$ towards boundary component, as the direction is not the focus of  the following argument.)   

Let $T$ be the trivalent tree dual to the Farey tessellation (see for example \cite{Bonahon09}).
Then the vertices of $T$ bijectively correspond to the ideal triangulations of $\hat{F}_i$ and the edges to {\it diagonal exchanges} of the  ideal triangulations --- a {\it diagonal exchange} removes a diagonal of an (ideal) quadrangle and add the other diagonal of the quadrangle.

Pick a maximal lamination $\hat{\nu}_{i, 0}$ of $\hat{F}_i$
such that 
\begin{itemize}
\item for each leaf of $\hat{\nu}_{i, 0}$, its endpoints  are at the puncture of $\hat{F}_i$, and
\item  every leaf of $\hat{\nu}_{i, 0}$ intersects each of $m_i$ and $m_{i +1}$ at most in a single point (Figure \ref{3-30no1}).
\end{itemize}
Then $\hat{\nu}_i$ is obtained by the infinite iteration of the Dehn twist of $\hat{\nu}_{i, 0}$ along $m_i$ and similarly 
$\hat{\nu}_{i+1}$ by the infinite iteration of the Dehn twist along $m_{i+1}$. 
The Dehn twists along $m_i$ and $m_{i +1}$ each correspond to a composition of two diagonal exchanges of $\hat{\nu}_{i, 0}$. 
Then  we obtain a bi-infinite path in $T$ connecting $\hat{\nu}_i$ to $\hat{\nu}_{i+1}$. 
Then, there is  a corresponding bi-infinite sequence  $(\hat{\nu}_{i,j})_j$ of maximal laminations on $\hat{F}_i$ that converges to $\hat{\nu_i}$ as $j \to - \ify$ and to $\hat{\nu}_{i+1}$ as $j \to \ify$.
\begin{figure}[h]
\begin{overpic}[trim = 0mm 0mm 0mm 0mm, clip, width = 2in
]{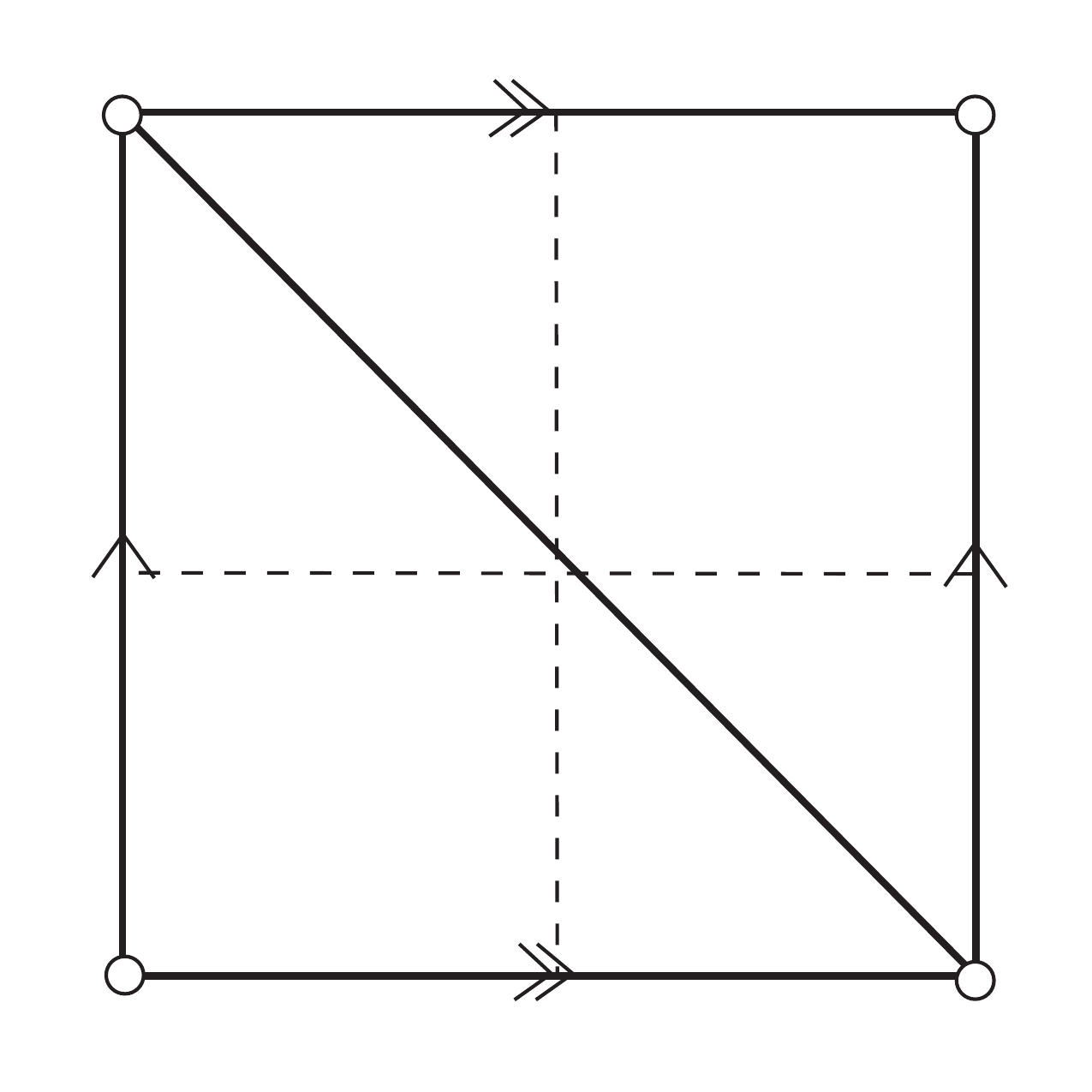}
\put(53,67){\tiny $m_{i + 1}$}
\put(20,40){\tiny $m_i $}
      \end{overpic}
\caption{An example of $\hat{\nu}_{i, 0}$, an ideal triangulation  of the once-punctured torus $\hat{F}_i$.   The arrows indicate the identification of edges of the square, so that the vertices correspond the puncture of $\hat{F}_i$.}\label{3-30no1}
\end{figure}
Let  $\nu_{i,j} \,(j \in \Z)$ be the corresponding maximal laminations on $S$ such that 
\begin{itemize}
\item[(i)] the restriction of $\nu_{i,j}$ to $F_i$  descends to $\nuh_{i,j}$ on $\hat{F}_i$, and
\item[(ii)]$\nu_{i,j}$ is isomorphic to $\nu_i$ (and $\nu_{i+1}$)  in some neighborhood, in $S$,  of the closure of $S \minus F_i$.
\end{itemize}
Note that, in (ii), we need to take a neighborhood so that $\nu_{i,j}$ spirals to the left towards $\pt F_i$ (as $\nu_i$ and $\nu_{i + 1}$ do).
Then $\nu_{i,j}$ is a bi-infinite sequence of maximal laminations of $S$ that converges to $\nu_i$ as $i \to -\ify$ and to  $\nu_{i+1}$ as $i \to \ify$.

\begin{figure}
\begin{overpic}[scale=.5,
]{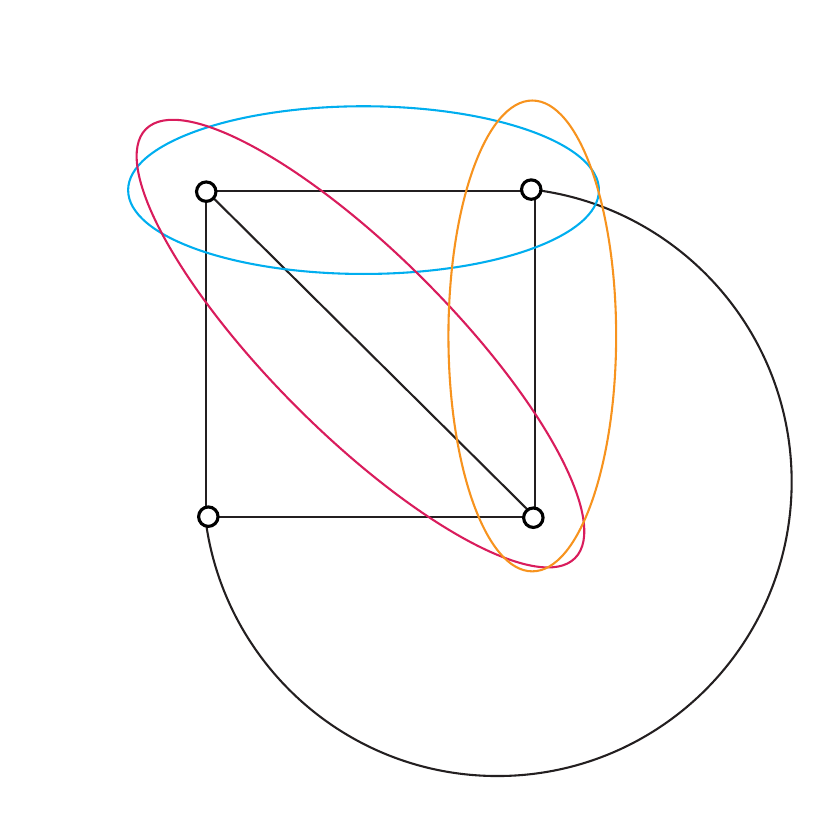}
       \end{overpic}
\caption{An ideal triangulation of a four-punctured sphere give  the edges of a tetrahedron (black).
The colored loops separate opposite edges of the tetrahedron. 
 }\label{2-20no1}
\end{figure}

{\it Case Two.}
Suppose that $F_i$ is a four-holed sphere.o
Similarly let $\hat{F}_i$ be the four-punctured sphere obtained by pinching the boundary components of $F_i$. 
Then, consider the  graph $T$ associated with $\hat{F}_i$ defined by: 
\begin{itemize}
\item The vertices of $T$ bijectively correspond to the ideal triangulations of $\hat{F}_i$ isomorphic to the triangulation of  the boundary of a tetrahedron (Figure \ref{2-20no1}). 
\item There is a (unique) edge between two vertices of $T$ if and only if the triangulation corresponding to one vertex is obtained form the other by simultaneous  diagonal exchanges of  opposite edges of the tetrahedron.  
\end{itemize}
 
 Similarly to Case One, we have
\begin{lemma}
This graph $T$ is dual to the Farey tessellation.
\end{lemma}

\begin{proof}
Each vertex of $T$ corresponds, uniquely,  to a ideal triangulation of $\hat{F}_i$. 
Then, for each pair of opposite edges of the triangulation, there is  an unique essential loop on $\hat{F}_i$ disjoint from the edges.
There are exactly three pairs of opposite edges and their corresponding loops are maximal mutually adjacent vertices of the curve graph of $\hat{F}_i$ (Figure \ref{2-20no1}).
The curve graph of the four-punctures sphere is the Farey graph (for example, see \cite[\S 5]{Schleimer-06}).
Thus, each ideal triangulation of $\hat{F}_i$ naturally corresponds to a unique triangle of the Farey tessellation, and therefore $T$ is isomorphic to the graph dual to the Farey tessellation.
\end{proof}

The maximal laminations $\hat{\nu}_i$ and $\hat{\nu}_{i+1}$ are distinct  endpoints  of the graph $T$ at infinity.
Then, similarly, there is a unique bi-infinite sequence $(\hat{\nu}_{i, j})_{j \in \Z} $ of  adjacent vertices of $T$ connecting $\hat{\nu}_i$ to $\hat{\nu}_{i+1}$. 
Indeed, there is a vertex of $T$ such that its corresponding triangulation of $\hat{F}_i$ contains two edges disjoint from $m_i$ and two edges  disjoint from $m_{i+1}$;
then $\hat{\nu}_{i}$ is obtained by the (infinite) iteration of the left Dehn twist along $m_i$, and $\hat{\nu}_{i+1}$ by the iteration of the left Dehn twists along $m_{i+1}$. 
Then, similarly, let $\nu_{i,j}$ be the lamination on $S$ satisfying (i) and (ii) in {\it Case One}.

In either Case One or Two, we have constructed  the bi-infinite sequence $\{ \nu_{i, j}\}_{j \in \Z}$ of maximal geodesic laminations on $S$ connecting $\nu_i$ to $\nu_{i +1}$,  that is,  $\{\nu_{i, j}\}$ converges to $\nu_i$ as $j \to - \In$ and  $\nu_{i +1}$ as $j \to \In$.
Then, by Lemma \ref{12-18no1}, for every  $i \in \{0,\dots, n-1\}$ and $j \in \Z$, there is a unique $\rho$-equivariant pleated surface $\beta_{i,j}\cn \tilde{S} \to \h^3$ realizing $\nu_{i,j}$. 
Then,  by Theorem \ref{12-20}, the convergence of $\nu_{i, j}$ immediately implies 
\begin{proposition} 
For each $i = 0,1, \dots, n-1$,
the pleated surface $\beta_{i,j}$ converges to $\beta_{i+1}$ as $j \to \infty$ and to $\beta_i$ as $j \to -\infty$ (in term of the closeness defined in Theorem \ref{12-20}). 
\end{proposition}

Recall that $\nu_i$ is obtained by the infinite iteration of the left Dehn twist of a multiloop $N_i$ along $M_i$. 
Recalling that  $M_i \cap M_{i + 1}$ is the set of common loops of $M_i$ and $M_{1 + i}$, we similarly have
\begin{lemma}\Label{12-31no1}
For all $i \in \{1,\dots, n-1\}$ and  $j \in \Z$,
the lamination $\nu_{i, j}$ is  obtained by the infinite iteration of the left Dean twist of some multiloop on $S$ along  $M_i \cap M_{i +1}$.
\end{lemma}
\begin{proof}
Recall that, for every $i \in \{1,\dots, n-1\}$ and  every positive integer $k$, we can take the multiloop $N_i$ inducing $\nu_i$ such that, if $P$ is a complementary pants of $M_i$ in $S$, then, for each pair of different boundary components of $P$,  the multiarc $N_i \cap P$ contains exactly $k$ parallel arcs of connecting the components.
Thus, for each for every $i \in \{1,\dots, n-1\}$, we fix such a multiloop $N_i$  given by $k = 3$ in Case One and $k = 2$ for Case Two.
Then for each $j \in \Z$, we can construct a multiloop $N_{i, j}$ such that $N_{i, j} = N_i$ in $S \minus F_i$ and the restriction of $N_{i,j}$ to $F_i$ induces $\hat{\nu}_{i,  j}$ on $\hat{F}_i$ by the integration of the left Dehn twist along $M_i$: See Figure \ref{12-29no2} for an example of $N_{i, j}$  with $k = 2$ in Case Two. 
\begin{figure}[h]
\begin{overpic}[width=3in,
]{12-29no2.pdf}
      \end{overpic}
\caption{}\label{12-29no2}
\end{figure}
\end{proof}

\section{Existence of admissible loops}\label{admissible}

Given a $\rho$-equivariant  pleated surface $\beta\col \H^2 \to \H^3$, the following theorem gives a way of finding an admissible loops on a projective structure $C$ with holonomy $\rho$ when its pleated surface is ``close'' to $\beta$. 
   
\begin{theorem}[c.f. \S 7 in  \cite{Baba-15gt}]\Label{admissibleII}
Let $\rho\cn \po(S) \to \psl$ be a homomorphism.
Let $\beta\cn \h^2 \to \h^3$ be a $\rho$-equivariant pleated surface realizing $(\sigma, \nu) \in \TT \times \gl$.
Then, there exists $\dl > 0$ such that, if a projective structure $C  \cong (\tau, L)$ with holonomy $\rho$ satisfies $\angle_\sigma(L, \nu) < \dl$ and a loop $\ell$ on $S$ satisfies $\angle_\tau(\ell, |L|) < \dl$,  then there is an admissible loop on $C$ isotopic to $\ell$.
\end{theorem}

\begin{remark}
For every $\ep > 0$,  if $\delta > 0$ is sufficiently small, then $\sigma$ and $\tau$ are $\ep$-close. 
We use $\angle_\sigma$ and $\angle_\tau$ because it seems to be natural choices. 
However (recall that) the choices of the hyperbolic structures on $S$ are not important for such angles as long as it is fixed or bounded in $\TT$.

\end{remark}

\Proof
Suppose that Theorem \ref{admissibleII} fails.
 Then there exists
 a sequence of projective structures $C_i \cong (\tau_i, L_i)$ with holonomy $\rho$  such that, letting $\lam_i = |L_i|$,  we have  $\angle (\lam_i, \nu) \to 0$ as $i \to \ify$ and a sequence of geodesic loops $\ell_i$ on $\tau_i$ with  $\angle(\ell_i, \lam_i) \to 0$ as $i \to \ify$ such that there is no admissible loop on $C_i$ homotopic to $\ell_i$.
Then, by Theorem \ref{12-20}, there are marking-preserving bilipschitz maps $ \psi_i \cn \sigma \to \tau_i$  converging to an isometry as $i \to \In$ and $\beta$ is the limit of the pleated surfaces $\beta_i$ corresponding to $C_i$ via $\psi_i$.

Since $\GL$ is compact in the Chabauty topology,  
by taking a subsequence if necessary, we can in addition assume that $\ld_i$ converges to a geodesic lamination $\ld_\If$ on $S$.
Since $\beta_i \to \beta$, thus $\nu$ is a sublamination of  $\lambda_\infty$. 
 We can also assume that the sequence of geodesic loops $\ell_i$ converges to a geodesic lamination $\ell_\If$,  taking a subsequence if necessary.
Then $\angle(\ell_\If, \ld_\If) = 0$.
Thus $\ell_\infty \cup \lambda_\infty$ is a geodesic lamination on $\sigma$.

There is a constant $K > 0$ depending on $(\sigma, \ell_\In \cup \lam_\In)$, such that for every $\ep > 0$, there is an $(\ep, K)$-nearly straight traintrack neighborhood $T_\ep$ of $\ell_\If \cup \ld_\infty$ (\cite[Lemma 7.10]{Baba-15gt}). 
Then, if $i \in N$ is sufficiently large, by the above convergences, there is a sufficiently small isotopy of  $\psi_i(T_\ep)$  into an $(\ep, K)$-straight traintrack $T_i$ on $\tau_i$ that carries both  $\ld_i$ and $\ell_i$.

For each $i \in \N$, let $\kap_i\cn C_i \to \tau_i$ denote the collapsing map, and let $\LL_i$ be the circular measured lamination on $C$ that descends to $L$ via $\kap_i$.
By \cite[Proposition 7.12]{Baba-15gt}, for sufficiently large $i$, there is a corresponding traintrack $\TT_i$ on $C_i$ diffeomorphic to $T_i$ such that 
\begin{itemize}
\item each branch of $\TT_i$ is supported on a round cylinder on $\rs$,
\item $\kap_i(|\TT_i|) = T_i$, and
\item  $\kap_i(\TT_i)$ is $\ep$-close to $T_i$, i.e. the $\kap_i$-image of each branch $\TT_i$ is $\ep$-close to a corresponding branch of $T_i$ in the hausdorff metric on $\tau_i$.
\end{itemize}

Then $\TT_i$ carries the measured lamination $\mL_i$.
Since $T_i$ carries $\ell_i$,   thus $\TT_i$ carries a corresponding loop $m_i$ so that a small homotopy transforms $\kap_i | m_i$ into $\ell$ in $|T_i|$.
Since its branches are supported on cylinders, $\TT_i$ is  foliated by vertical circular arcs  (\S \ref{traintracks}).
Then, since $m_i$ is carried by $\TT_i$, we can (further) isotope $m_i$ through loops carried by $\TT_i$, so that $m_i$ is transversal to this foliation of $\TT_i$. 
Then by Lemma \ref{123112}, $m_i$ is admissible which is a contradiction.  
 \Qed{admissibleII}
 
 Let $C \cong (\tau, L)$ be a projective structure on $S$ with holonomy $\rho$. 
 Let $\beta \col \h^2 \to \h^3$ be the $\rho$-equivariant pleated surface induced by $(\tau, L)$.
 Then, since $\angle(L, L) = 0 < \del$,  
Theorem \ref{admissibleII} immediately implies
\begin{corollary}\Label{AdmissibleLoop}
There is $\del > 0$ such that, if a geodesic loop $\ell$ on $\tau$ satisfies $\angle_\tau (\ld, \ell) < \del$, then there is an admissible loop on $C$ isotopic to $\ell$. 
\end{corollary}

In addition
\begin{corollary}\Label{12-12no1}
For every a  projective structure $C \cong(\tau, L)$ on $S$,
there exists sufficiently small  $\del > 0$ such that, if a geodesic multiloop $M$  on $\tau$ satisfies
\begin{itemize}
\item[(1)] $\angle_\tau(M, L) < \del$ and
\item[(2)] $|L|$ is contained in the $\del$-neighborhood of $M$ in $\tau$, 
\end{itemize}
then  every loop $\ell$ on $C$ satisfying $\angle_\tau(\ell, M) <  \ep$ is isotopic to an admissible loop.
\end{corollary}
\begin{proof}
Let $L = (\ld, \mu)$ denote the measured lamination, where $\ld \in \GL$ and $\mu \in \TM(\lam)$.
Then,  for every $\ep > 0$, there exists $\del > 0$, such that, if a multiloop $M$ on $\tau$ satisfies (1) and (2) for this $\del$,  then $\angle_\tau(\ell, \ld) < \ep$ for every loop $\ell$ on $\tau$ with $\angle_\tau(\ell, M) < \del$. 
Thus if $\ep > 0$ is sufficiently small, then  by Corollary \ref{AdmissibleLoop},  $\ell$ is isotopic to an admissible  loop on $C$.
\end{proof}

\subsection{Admissible loops close to $\nu_{i, j+1}$ on projective surfaces close to $\nu_{i,j}$ in Thurston coordinates}
We carry over the notations from \S \ref{5-14}.
Then
\begin{proposition}\Label{3-2no1}
For all $i \in \{ 0, 1, \dots, n-1\}$ and $j \in \Z$, 
there exists $\del> 0$,  such that, if a projective structure  $C \cong(\tau, L)$  on $S$ satisfies $L$ and $\nu_{i, j}$ are $\del$-hausdorff close, for 
  every  loop $\ell$ on $C$ such that the hausdorff distance between $\ell$ and $\nu_{i, j + 1}$ are $\del$-close,  there is an admissible loop on $C$ isotopic to $\ell$\,; indeed such a loop $\ell$ exists.
\end{proposition}
\Proof
Let $\tau_{i ,j}$ be the hyperbolic structure on $S$ such that $(\tau_{i, j}, \nu_{i, j})$ is realized by the $\rho$-equivariant pleated surface  $\beta_{i, j}$.  
Given a projective structure $C \cong (\tau, L)$ on $S$, let $L = (\lambda, \mu)$ denote the measured lamination, where $\lam \in  \GL$ and $\mu \in \TM(\lam)$.  
For every $\ep > 0$, if $\angle(\nu_{i.j},\lambda) > 0$ is sufficiently small, then there is a marking preserving $\ep$-rough isometry $\psi_{i,j} \cn \tau_{i,j} \to \tau$ given by Theorem \ref{12-20}.

Let $K$ be any positive number less than one third of the shortest closed leaf of (the geodesic representative of) $\nu_{i,  j + 1}$ on $\tau_{i, j}$. 
Then,  for every  $\ep > 0$, there is an $(\ep, K)$-nearly straight traintrack $T_{i, j +1}$ on $\tau_{i,j}$ that carries $\nu_{i,j+1}$ on $\tau_{i, j}$ (Proposition 7.11 in \cite{Baba-15gt}).  
In addition, it is easy to show that,  for every $H > 0$,   we can in addition assume, if a branch $T_{i, j+1}$ intersects no closed leaf of $\nu_{i, j}$, then its length is at least $H$ (since $K$ is determined by the lengths of closed leaves of $\nu_{i, j + 1}$).

If $\ep > 0$ is sufficiently small, each branch of $T_{i, j+1}$ is close to a geodesic segment of length at least $K$. 
Recall that $\nu_{i, j}$ and $\nu_{i, j+1}$  differ by a diagonal exchange or two simultaneous diagonal exchanges. 
Therefore, we can naturally assume that, for every leaf $d \in \nu_{i, j}$ with  $d \notin \nu_{i, j  + 1}$, there is a unique branch $R$ of $T_{i, j + 1}$ such that  $d$ transversally intersects each horizontal edge of $R$ exactly once (so that $d \cap R$ is a single geodesic segment) and such that  no other leaf of $\nu_{i, j}$ intersects $R$.
  
We show that, if $H > 0$ is sufficiently large and  $\ep > 0$ are sufficiently small, then the traintrack  $T_{i,j + 1}$ (on $\tau_{i, j}$) is \textit{admissible} on $C$. 
That is, for every loop $\ell$ carried by $T_{i,j +1}$, there is an admissible loop on $C$ isotopic to it. 
The proof is similar to that of Theorem \ref{admissibleII}, except more careful arguments for branches corresponding to leaves $d$ of $\nu_{i, j + 1}$ with  $d \notin \nu_{i ,j}$, where $d$ and $L$ are not (necessarily) close to being parallel. 

{\it Case One.} First suppose that $F_i$ is a one-holed torus.
Let $d_j$ be the leaf of $\nu_{i,j}$ removed  by the diagonal exchange of $\nu_{i, j}$ yielding $\nu_{i , j+1}$, and let $d_{j+1}$ be the leaf of $\nu_{i, j+1}$ added by the exchange. 
Then there is a unique branch $R_0$ of $T_{i, j + 1}$ such that  $d_j$ is the only leaf of $\nu_{i, j}$ intersecting $R_0$.
Since $d_{j + 1}$ is not a closed leaf, we can assume that $R_0$ has length at least $H > 0$. 
For every $\zeta > 0$, if $\del > 0$ is sufficiently small, then the distance between $\tau$ and $\tau_{i, j}$ is less than $\zeta$.
Thus, since $T_{i, j+1}$ is $(\ep, K)$-straight, we can assume that there is an $(2\ep, K)$-straight traintrack $\TTT$ also on $\tau$ obtained by a small perturbation of $\psi_{i,j}(T_{i, j+1})$.
Let $\TT = \cup_{k = 0}^m \RR_k$ be the traintrack on $C$ that descends to $\TTT$ via the collapsing map $\kp\cn C \to \tau$. 
We can assume that $\RR_0$ is the branch of $\TT$ corresponding to $R_0$.
For every  branch $\RRR$ of $\TTT$ that does not corresponding to $R_0$,   $\RRR \cap L$ is a union of geodesic segments connecting the vertical edges of $\RRR$.
Then, since $\TTT$ is $(2\ep, K)$-nearly straight with sufficiently small $\ep > 0$,  as in \cite[Proposition 7.12]{Baba-15gt}, by a small isotopy of  $\TT$  on $C$ without changing $|\TT|$, we may assume that  $\RR_k$ is a rectangle supported on a round cylinder for each $k =  1, \dots, m$. 

To complete the proof, we show there is an isotopy of $T$ on $C$ such that
\begin{itemize}
\item this isotopy is {\it supported} on $R_0$ (i.e. it fixes $T \minus R_0$), and 
\item after the isotopy, $\RR_0$ can be subdivided into three branches which are supported on three (consecutive) round cylinders.
\end{itemize}
After such an isotopy, the traintrack $\TT$ is admissible on $C$ by Lemma \ref{123112}; moreover,  by Lemma \ref{12-31no1}, there are many distinct loops carried by $\TT$.  

Let $\sigma_{i,j}$ be the subsurface of $\tau_{i,j}$ with geodesic boundary which is isotopic to the subsurface $F_i$ of $S$.
Then the boundary component of $\sigma_{i,j}$ is a closed leaf of $\nu_{i,j}$, and the lamination $\nu_{i,j}$ decomposes $\sigma_{i, j}$ into two ideal triangles.
 Let $\tilde{\sigma}_{i,j}$ be the universal cover of $\sigma_{i,j}$. 
Then $\td{\sigma}_{i, j}$ is a convex subset of $\h^2$ bounded by the geodesics which cover the boundary component of $\sigma_{i, j}$.
Then the total lift $\til{\nu}_{i,j}$ yields an ideal triangulation of $\til{\sigma}_{i, j}$.
Let  $\td{d}_j$  be a lift of  $d_j$ to $\td{\sigma}_{i, j}$.
Then $\td{d}_j$ separates adjacent  ideal triangles $\Delta_1$ and $\Delta_2$ of $\td{\sigma}_{i, j} \minus \td{\nu}_{i, j}$.
Then  $\Delta_1 \cup \Delta_2 (=: Q_{i,j})$ is a fundamental domain of $\tilde{\sigma}_{i,j}$, and it  is an ideal quadrangle.
Different vertices of $Q_{i,j}$ are endpoints of different boundary geodesics of $\tilde{\sigma}_{i,j}$, and 
	different boundary geodesics $\ell_1, \ell_2$ of $\td{\sigma}_{i, j}$ are preserved by different elements   $\gam_1, \gam_2 \in \pi_1(S) \minus \{id\}$. 
Then, since $\rho$ is purely loxodromic, by Lemma \ref{7-31-12no1}, the axes of the loxodromic elements $\rho(\gam_1)$ and $\rho(\gam_2)$ share no endpoint. 
Then, since the pleated surface $\beta_{i,j}\cn \h^2 \to \h^3$ is $\rho$-equivariant,  $\beta_{i,j}$  takes different boundary geodesics of $\td{\sigma}_{i,j}$ to distict geodesics in $\h^3$ without any common endpoint.
In particular, $\beta_{i, j}$ takes the vertices of $Q_{i, j}$ to distinct points on $\rs$.

Since the branches $\RR_1, \cdots, \RR_m$ are supported on round cylinders, the vertical edges of $\RR_0$ are circular. 

Since $d_{j + 1}$ is the only leaf of $\nu_{i, j + 1}$ intersecting $R_0 \st \sigma_{i, j}$,  there is a unique lift  $\td{R}_0$ of $R_0$  contained in $Q_{i,j}$.
Then,  if $H > 0$ is  sufficiently large and $\ep > 0$ is sufficiently small, then the opposite vertical edges of $\til{R}_0$ are sufficiently far in $Q_{i, j}$. 
Let $\td{\RR}_0$ be the branch of $\td{\TT}$ corresponding to $\til{R}_0$.
Then the vertical edges of $\til{\RR}_0$ are contained in a vertical edge of another branch of $\til{\TT}$, which is supported on a round cylinder.


In the quadrangle $Q_{i, j}$, the diagonals $\td{d}_j$ and $\td{d}_{j+1}$ intersect in a single point.
Let $v_1, v_2, v_3, v_4$ denote the vertices of  $Q_{i, j}$ so that $v_1$ is the vertex of $\Delta_1$ opposite of the diagonal $d_{i, j}$,
 $v_2$ is the vertices of $\Delta_2$ opposite of the diagonal $d_{i, j}$, and 
 $v_3, v_4$ are the endpoints of $d_{i. j}$. 

Then take disjoint round disks $D_1, D_2$ in $\rs$ so that  $D_1$ contains $\beta_{i, j}(v_1), c_1$ and $\beta_{i, j}(v_3)$ and $D_2$ contains $\beta_{i, j}(v_2), c_2$ and $\beta_{i, j}(v_4)$ (see Figure \ref{fDisks}).
 For each $i = 1, 2$, let $r_i$ be the round circle on $\rs$ bounding $D_i$.  

\begin{figure}[H]
\begin{overpic}[scale=.6
] {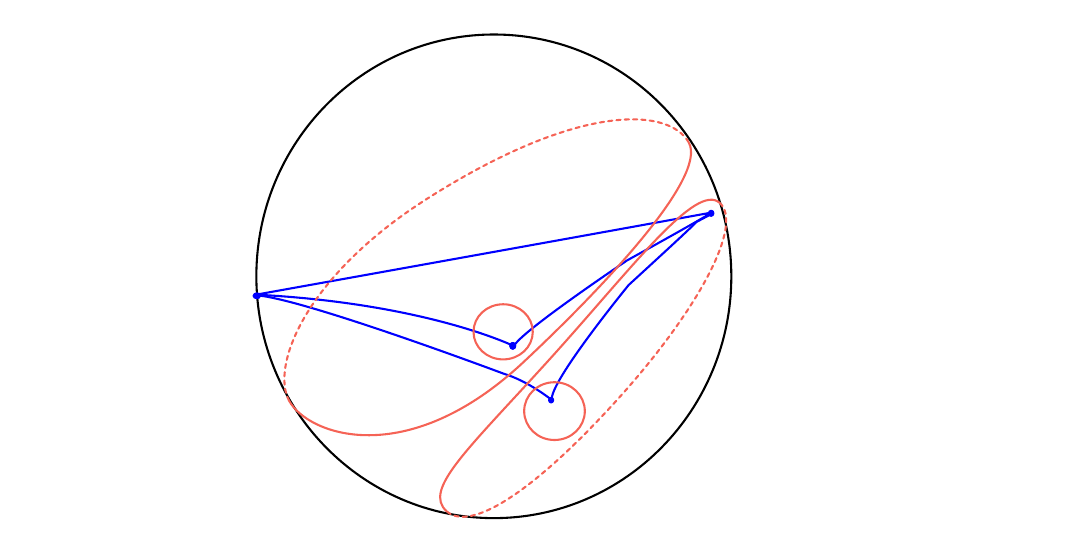} 
\put(45 , 19 ){\textcolor{blue}{\tiny $v_1$}}  
\put(50 , 10 ){\textcolor{blue}{\tiny $v_2$}}  
\put(24 , 24 ){\textcolor{blue}{\tiny $v_4$}}  
\put(65 , 27 ){\textcolor{blue}{\tiny $v_3$}}  
\put(45 , 23 ){\textcolor{Red}{$c_1$}}  
\put(54 , 14 ){\textcolor{Red}{$c_2$}}  
\put(60 , 33 ){\textcolor{Red}{$\gamma_1$}}  
\put(43 , 4 ){\textcolor{Red}{$\gamma_2$}}  
      \end{overpic}
\caption{The $\beta_{i, j}$-images of the vertices $v_i$ and the round circles separating them.}\label{fDisks}
\end{figure}
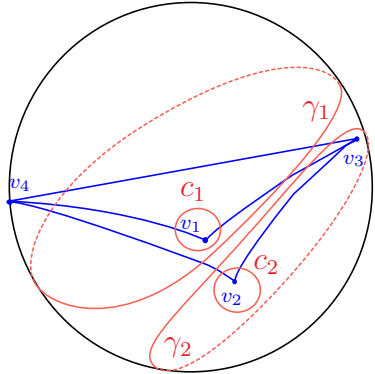

\begin{claim}
If $\del > 0$ is sufficiently small, then there is an isotopy of $\TT$ on $C$ which only moves $\RR_0$ so that $r_1$ and $r_2$ decompose $\RR_0$ into three branches supported on three consecutive round cylinders: the round cylinder $\AA_1$ bounded by $c_1$ and $r_1$, the round cylinder   $\AA$, and  the round cylinder bounded $\AA_2$ by $c_2$ and $r_2$. 
\end{claim}

\proof

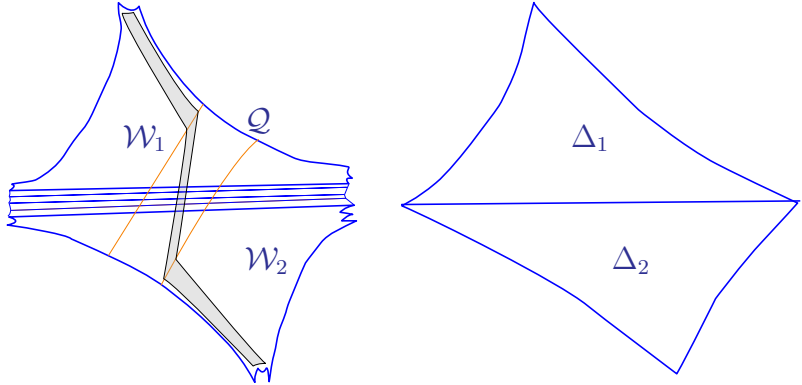
\begin{figure}[H]
\begin{overpic}[scale=.6
] {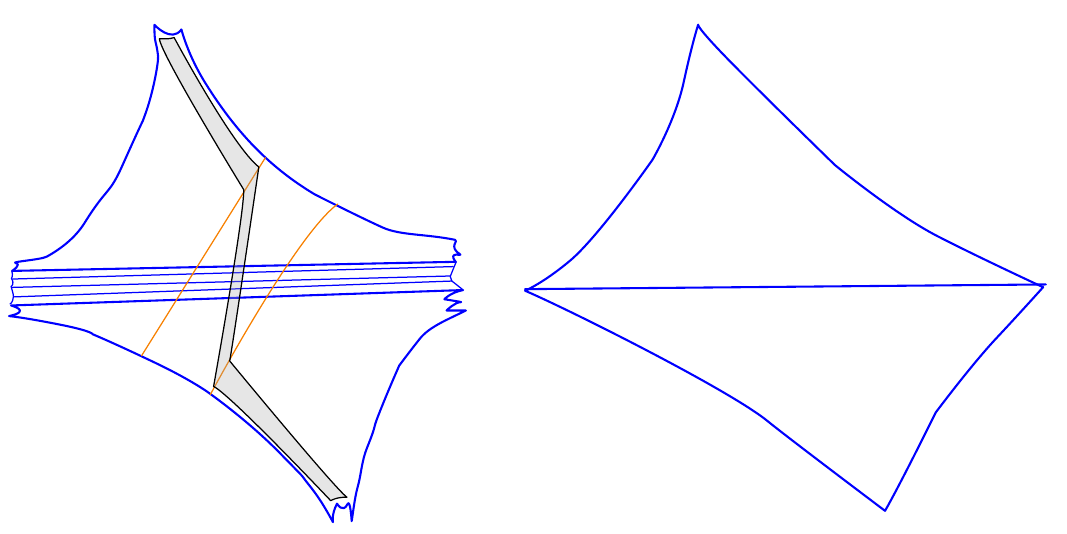}
 \put(30 ,32 ){\textcolor{Blue}{$\mathcal{Q}$}}  
 \put(15 ,30 ){\textcolor{Blue}{$\mathcal{W}_1$}}  
 \put(30 , 15 ){\textcolor{Blue}{$\mathcal{W}_2$}}  
 \put(70 ,30 ){\textcolor{Blue}{$\Delta_1$}}  
 \put(75 , 15 ){\textcolor{Blue}{$\Delta_2$}}  
      \end{overpic}
\caption{Left: $\RR_0$ after isotopy, Right: the ideal quadrangle $Q_{i, j}$}\label{fAfterIsotopy}
\end{figure}
Let $P$ be a stratum of $(\h^2, \til{L})$ that intersects $\til{R}_0$.
Then, as $L$ and $\nu_{i, j}$ are sufficiently Hausdorff-close,  $P$ is Hausdorff close to either $\Delta_1, \Delta_2$ or  $d_{i, j}$ (as a subset of  the disk $\H^2 \cup \bdr_\infi \H^2$).  
Therefore  the ideal points $\bdr_\infi P$ of $P$ map into $D_1 \cup D_2$ by $\beta_{i, j}$, and there is at least one point of $\bdr_\infi P$ maps into $D_i$ for each $i = 1, 2$.
Let $\mP$ be a  stratum of $(\tilde{C}, \tilde{\LL})$ that corresponds to $P$ by the collapsing map. 
Let $f$ be the developing map of $C$.
Then  the $f$-preimage of   $r_1 \sqcup r_2$ decomposes $\mP$ into three connected regions mapping into the different complementary regions $D_1, D_2$  and $\AA$ of $r_1 \sqcup r_2$. 

In particular, the preimage of $\AA$ is either an arc  or a rectangle supported in the round cylinder $\AA$, and it depends on whether $\mP$ is one-dimensional or two-dimensional. 
Let $\mathcal{Q}$ be the union, over all strata $\mP$ of $(\til{C}, \til{L})$ whose $\tilde{\kap}$-images intersect $\til{R}_0$; then $\mathcal{Q}$ is topologically a closed disk.
Let $U$ be the $f$-preimage of $\AA$ in $\mathcal{Q}$. 
Then $U$ is a union of a rectangles and arcs supported $\AA$, and $U$ is a projective structure on a rectangle supported on $\AA$.

Let $\WW_i$ be the stratum of $(\til{C}, \til{L})$ which map to a stratum of $(\H^2, \til{L})$ Hausdorff close to to $\Delta_i$. 
Then $\WW_i$ is bounded by circular arcs, and there is a unique boundary circular arc $a_i$ intersecting both $r_i$ and $c_i$ transversally. 
Moreover, $a_i$ is transversal to the circular foliation on $\AA_i$, and
 $f^{-1}(\AA_i) \cap \WW_i$ contains a rectangle $\RR_{0, i}$  supported on $\AA_i$ such that its vertical edge on $c_i$ matches with the vertical edge of $\RR_0$ on $c_i$. 
 
In addition,  there is a rectangle $\RR_{0,0}$ in $U$ supported on $\AA$ such that its vertical edges are the vertical edge of $\RR_{0,1}$ on $r_1$ and the vertical edge of $\RR_{0,2}$ on $r_2$.

Then it is easy to make an isotopy of $\RR_0$ to $\RR_{0,0} \cup \RR_{0,1} \cup \RR_{0, 2}$ (Figure \ref{fAfterIsotopy}).

{\it Case Two.}
Suppose that $F_i$ is a four-holed sphere. 
Then the proof is similar to Case One, and we leave the proof to the readers.  
The only difference is that  $\nu_{i,j}$ and $\nu_{i, j+1}$ differ by two diagonal exchanges (intend of one).
We accordingly deal with two branches of the traintracks more carefully than the other branches.
\Qed{3-2no1}


\section{Sequence of grafts traveling in $\gl$.}\Label{sTravelingInGL}
In this section, we prove our main theorem (Theorem \ref{main}) modulo {\tt}  Corollary \ref{3-10no3}.
This corollary will be independently proved in Part \ref{iteration}.
Recall that, in  \S \ref{5-14},  starting with two projective structures on $S$ sharing purely loxodromic holonomy $\rho\col \pi_1(S) \to \PSL$, we have constructed finitely many geodesic laminations $\nu_0, \nu_1, \dots, \nu_n$ on $S$ and, for each   $0 \leq i \leq n$,  a family of geodesic laminations $\nu_{i, j}\, (j \in \Z)$ on $S$ ``connecting'' $\nu_i$ to $\nu_{i + 1}$. 
 By Proposition \ref{3-2no1}, 
given a projective structure whose measured lamination is sufficiently Hausdorff close to  $\nu_{i, j}$, one can find an admissible loop which is sufficiently close to  $\nu_{i, j + 1}$;
then, by grafting along the admissible loop sufficiently many time, we obtain a projective structure whose measured lamination is Hausdorff close to $\nu_{i, j +1}$.
Namely
\begin{proposition}\Label{010313}
Given $0 \leq i < n$ and   $j \in \Z$, there exists $\dl_{i, j} > 0$ such that, if a projective structure $C \cong(\tau, L)$ with purely loxodromic holonomy $\rho\cn \pi_1(S) \to \psl$ satisfies $\angle(\lambda, \nu_{i,j}) < \delta_{i, j}$, then, for every $\ep_{i, j} > 0$, there is an admissible loop $\ell$ on $C$ such that, letting $\Gr^k_\ell(C) \cong (\tau_k, L_k)$ in Thurston coordinates, we have $\angle(L_k, \nu_{i, j+1}) < \ep_{i , j}$ for sufficiently large $k$.
\end{proposition}

Next, using grafting obtained by Proposition \ref{010313}, we can transform a projective structure close to $\nu_i$ to a projective structure close to $\nu_{i +1}$:
\begin{proposition}\Label{1-1no1}
For $ i = 0, 1, \cdots, n-1$, there exists $\del_i  > 0$, such that if $C \cong(\tau, L )$ is a projective structure in $\PP_\rho$ with $\angle(\lambda, \nu_i) < \del_i$, then, for every $\ep > 0$, there is a finite composition of grafts starting from $C$, 
$$C = C_0  \xrightarrow{Gr_{\ell_1}} C_1  \xrightarrow{Gr_{\ell_2}} C_2 \to \ldots  \xrightarrow{Gr_{\ell_k}} C_k$$
so that the last projective structure $C_k \cong (\tau_k, L_k)$ satisfies $\angle( L_k, \nu_{i +1}) < \ep$.
\end{proposition}

\begin{proof} 
Let  $C \cong(\tau, L)  \in \PP_\rho$ be such that  $\angle(L, \nu_i) > 0$ is  sufficiently small.
Recall that $\nu_{i, j} \to \nu_i$ as $j \to -\In$  in the Chabauty topology.
Thus, if $j \in \Z$ is sufficiently small and $\ell$ is a loop on $C$ (whose geodesic representative) is sufficiently close to $\nu_{i, j}$ on $\tau_{i, j}$, then $\angle(\ell, \nu_i) > 0$ is also sufficiently small. 
Then, by Theorem \ref{admissibleII}, a loop $\ell$ on $C$ is admissible up to an isotopy.
Set $\Gr^h_\ell(C) \cong (\tau_h, L_h)$ for each $h \in \Z_{> 0}$. 
Therefore, for every $\del > 0$, if $j \in \Z$ is sufficiently small and $\ell$ is sufficiently close to $\nu_{i, j}$, then by Corollary \ref{3-10no3}, we have $\angle(L_h, \nu_{i, j}) < \del$ for sufficiently large $h \in \N$. 

Since $\nu_{i, j} \to \nu_{i +1}$ as $j \to \In$, for every $\ep > 0$,   we can pick sufficiently large $j_\ep \in \N$ such that $\angle(\nu_{i +1}, \nu_{i, j_\ep}) < \ep$.
Then let $\del_{i, j_\ep} > 0$ be such that, if a loop $\ell$ satisfies $\angle(\ell, \nu_{i, j_\ep}) < \del_{i, j_\ep}$, then $\angle(\ell, \nu_{i + 1}) < \ep$. 

For every $j \in \Z$ with $j <  j_\ep$, inductively define $\del_{i, j} > 0$  inductively so that $\del_{i, j}$ is the constant obtained by applying Proposition \ref{010313} to $\del_{i, j+1}$.
Thus, for  $j < j_\ep$, if there is a projective structure $C' \cong (\tau', L')$ satisfies that $\angle(\nu_{i, j}, L') < \del_{i, j}$, then we can inductively construct a composition of grafts starting from $C'$, 
$$C' = C_0  \xrightarrow{Gr_{\ell_1}} C_1  \xrightarrow{Gr_{\ell_2}} C_2 \to \ldots  \xrightarrow{Gr_{\ell_k}} C_k = (\tau_k, L_k), $$
such that  $\angle( L_k, \nu_{i +1}) < \ep$.
In addition, it follows from  the first paragraph in this proof that, if $j \in \Z$ is sufficiently small, then such a projective structure $C'$ can be obtained by a finite iteration of grafts of $C$ along a fixed admissible loop. 
This completes the proof.
\end{proof}

Recall that we started with arbitrary projective structures $C_\sharp \cong (\tau_\sharp , L_\sharp)$ and $C_\flat \cong (\tau_\flat, L_\flat)$ on $S$  sharing purely loxodromic holonomy $\rho\cn \po(S) \to \psl$  and that  $L_\sharp = (\lambda_\sharp, \mu_\sharp)$ and $L_\flat = (\lambda_\flat, \mu_\flat)$ are measured laminations. 
Then
\begin{proposition}\Label{4-21no1}
For every $\ep > 0$, there exists a finite composition of grafts starting from $C_\sharp $ $$C_\sharp  = C_0  \xrightarrow{Gr_{\ell_1}} C_1  \xrightarrow{Gr_{\ell_2}} C_2 \to \ldots\to C_n$$ such that the last projective structure $C_n \cong (\tau_n, L_n)$ satisfies  $\angle(\lam_\flat, L_n) < \ep$.
\end{proposition}

\proofof{Proposition \ref{4-21no1}}
Recall that the multiloops  $M_\sharp$ and $M_\flat$ are taken to be sufficiently close to $\lam_\sharp$ and $\lam_\flat$,  respectively.
In other words, given  $\zeta > 0$,   we can assume that $\angle(M_\flat, \lam_\flat) < \zeta$, $\angle(M_\sharp, \lam_\sharp) < \zeta$ and  that $\lam_\flat$ and $\lam_\sharp$ are contained in the $\zeta$-neighborhoods of $M_\flat$ and $M_\sharp$, respectively.
Since $\nu_n$ contains $M_n=  M_\flat$, for every $\ep > 0$, if $\zeta > 0$ is sufficiently small, there exists $\del_n > 0$ such that  if a geodesic lamination $\lam$ on $S$ satisfies $\angle(\lam, \nu_n) < \del_n$, then $\angle(\lam, \lam_\flat) < \ep$.
For $i = 0, 1, \dots, n-1$, let $\del_i > 0$ be the constant given by Proposition \ref{1-1no1}.  

\begin{claim}\Label{010513}
If $\zeta > 0$ is sufficiently small, then, for loops $\ell$ sufficiently close to $\nu_0$ with the Hausdorff metric on $\tau_\sharp$, 
\begin{itemize}
\item $\ell$ is admissible on $C_\sharp$ (up to an isotopy), and

\item for $k \in \N$ letting $\Gr_\ell^{k}(C_\sharp ) \cong (\tau_{k}, L_{k})$ in Thurston coordinates, we have $\angle(L_{k}, \nu_0) < \del_0$ for sufficiently large $k$.
\end{itemize}
\end{claim}
\begin{proof}

Since $\ell$ is sufficiently (Hausdorff-)close to $\nu_0$, then $\angle(\nu_0, \ell)$ is sufficiently small. 
Since $\angle(\lam_\sharp, \nu_0) < \zeta$ is sufficiently small and $\nu_0$ is maximal,  $\angle(\lam_\sharp, \ell)$ is also sufficiently small.  
Thus, by Corollary \ref{AdmissibleLoop}, $\ell$ is admissible. 

By  Corollary \ref{3-10no3}, $\angle(L_k, \ell) \to 0$ as $k \to \In$.
Thus, since $\angle(\ell, \nu_0)$ is sufficiently small, $\angle(L_k, \nu_0) < \del_0$ for sufficiently large $i$.
\end{proof}
 Let $C_{k_0}$ be $Gr_{\ell}^{k_0}(C) \cong (\tau_{k_0}, L_{k_0})$, given by Claim \ref{010513}, with a sufficiently large $k_0 \in \N$. 
Then, since  $\angle(L_{k_0}, \nu_0) < \del_0$, by Proposition \ref{1-1no1}, there is a composition of grafts from $C_{k_0}$
 $$C_{k_0}  \rightarrow C_1  \rightarrow C_2 \to \ldots\to C_{k_1} \cong (\tau_{k_1}, L_{k_1}),$$ such that, $\angle(\nu_1, L_{k_1}) < \del_1$.
 By inductively applying Proposition \ref{1-1no1},   for each $i = 0, 1, \cdots, n$,  we can extend this composition of grafts  to 
 $$ C_{k_0}  \to \dots \to C_{k_1}  \to \dots \to C_{k_2} \to \dots \to C_{k_i}$$ 
 so that the last projective structure $C_{k_i} \cong (\tau_{k_i}, L_{k_i})$ satisfies $\angle(\nu_i, L_{k_i}) < \del_i$.
In particular,  when $i = n$, we have $\angle(\nu_n, L_{k_n}) < \delta_n$.
Hence $\angle(\lam_\flat, L_{k_n}) < \ep$.
\Qed{4-21no1}

\proofof{Theorem \ref{main}} 
Let $\del > 0$ be the constant obtained by applying Theorem \ref{8-14-12no1} to $C_\flat  \cong (\tau_\flat, L_\flat)$.
Then, by Proposition \ref{4-21no1},  there is a composition of grafts along loops, 
$$C_\sharp  = C_0  \xrightarrow{Gr_{\ell_1}} C_1  \xrightarrow{Gr_{\ell_2}} C_2 \to \ldots\to C_n,$$  such that, letting $C_n \cong (\tau_n, L_n),$ we have 
$\angle(L_n, L_\flat) < \del$.
Then we can graft $C_n$ and $C_\sharp $ along multiloops to a common projective structure.    
Hence there are admissible loops $M_n$ on $C_n$ and $M_\flat$ on $C_\flat$ such that $\Gr_{M_n}(C_n) = \Gr_{M_\flat} (C_\flat)$.
Since the grafting $\Gr_{M_n}$ of $C_n$ is naturally a composition of grafts along loops of $M_n$, which completes the proof. 

\part{Iteration of grafting along a loop}\label{4-25no1}\Label{iteration}

\section{Limit in Thurston coordinates}
 We prove
\begin{theorem}\Label{3-10no2}
Let $\ell$ be an admissible loop on a projective surface $C$.
Let $C \cong(\tau, L) \in \TT \times \ml$, and let $L = (\ld, \mu)$ with $\ld \in \GL$ and $\mu \in \mc{TM}(\lam)$.
For each $i \in \N$, let $C_i = Gr_\ell^i(C)$, the $2\pi i$-graft of $C$. 
Similarly, let $C_i \cong (\tau_i, L_i)$ and $L_i = (\ld_i, \mu_i)$.
Let $\beta_i\col\h^2 \to \h^3$ be the pleated surface corresponding to $C_i$.
Then 
\begin{itemize}
\item[(i)] $\tau_i$ converges to a hyperbolic surface $\tau_\In \in \TT$ as $i \to \If$.
\item[(ii)] $\beta_i$ converges to a $\rho$-equivariant pleated surface realizing $(\tau_\In, \lam_\In)$ for some $\lam_\In \in \GL$ containing $\ell$.
\item[(iii)] $L_i$ converges to an (heavy) measured lamination $L_\In$ supported on $\lam_\In$ such that 
$\ell$ is the only leaf with weight infinity. 
\end{itemize}
\end{theorem}

Theorem \ref{3-10no2} (i) (ii) immediately imply
\begin{corollary}\Label{3-10no3}
$\angle( L_i,\ell )$  converges to $0$ as $i \to \If$.
\end{corollary}

Arguments in Part \ref{iteration} are independent of the rest of paper, and this corollary is an important ingredient for the proof of Theorem \ref{main}, proved in Part 1.

\section{Limit of the complement of the admissible loop}\Label{6-8no1}

Let $C$ be a projective structure on $S$ with holonomy $\rho\cn \po(S) \to \psl$, and 
	let $\ell$ be an admissible loop on $C$. 
\begin{definition}\Label{4-27no3}
Let $\td{\ell}$ be a lift of $\ell$ to the universal cover $\til{C}$, and $\gam_\ell$ be the element of $\po(S)$ that represents $\ell$ and preserves $\td{\ell}$. 
Normalize the developing map $f$ of $C$ by an element of $\PSL$ so that the loxodromic $\rho(\gam_\ell) \in \PSL$ fixes $0$ and $\If$ in $\rs$. 
Pick a parametrization   $\td{\ell}\cn \R \to \til{C}$ of $\td{\ell}$.
Then its $f$-image can be written in polar coordinates $(e^{r(t)}, \theta(t)), \, t \in \R,$ so that $$f \cc \td{\ell} (t) = \exp[r(t) + i \theta(t)]$$ where $r\cn \R \to \R_{> 0}$ and $\theta\cn \R \to \R$ are continuous functions.  
Then we say that $\ell$  \textit{spirals} if $\theta$ is an unbounded  function. 
Otherwise it is called \textit{roughly circular}. 
\end{definition} 

Indeed, an admissible loop $\ell$ is roughly circular if and only if there is a  homotopy (or an isotopy) between $f \cc \td{\ell}$ and a circular arc on $\rs$ connecting the fixed points of $\rho(\gam_\ell)$ such that the homotopy is equivariant under the restriction of $\rho\cn \po(S) \to \psl$  to the infinite cyclic group generated by $\gam_\ell$.

In the setting of Theorem \ref{3-10no2},
since $C_i = \Gr_\ell^i(C)$, then $C \sm \ell$ isomorphically embeds into $C_i$ so that the complement of $C \sm \ell$ in $C_i$ is the cylinder inserted by the grafting $\Gr_\ell^i$.
 This cylinder is naturally cut, along parallel isomorphic copies of $\ell$,  into  $i$ isomorphic copies of a grafting cylinder (of ``length $2\pi$'').
Let $M_i$ be the union of $i +1$ parallel copies of $\ell$ on $C_i$ that decomposes  $C$ into the $i$ cylinders and $C \sm \ell$.
Then let $\ell_i$ be, if $i + 1$ is odd, the middle loop of $M_i$ and, if $i$ is odd, a boundary component of the middle grafting cylinder. 
Let $\CC_i = C_i \sm \ell_i$.
Then, there is a natural isomorphic embedding  of $\CC_i$  into $\CC_{i+1}$.
Therefore, we let $$\CC_\In = \lim_{i \to \If} (C_i \sm \ell_i).$$
Then $\CC_\In$ is a projective structure on $S \sm \ell$, and its holonomy is the restriction of $\rho$ to $\po(S \sm \ell)$.
(Equivalently  $\CC_\In$ is obtained by attaching a half-infinite grafting cylinder along each boundary component of $C \sm \ell$.)

Recall that the holonomy $\rho$ of $C$ is non-elementary \cite{Gallo-Kapovich-Marden}.
Noting that $C \minus \ell$ has either one or two connected components, we have 
\begin{lemma}\Label{4-20no2}
Let $P$ be a connected component of $C \minus \ell$.
Then, the restriction of $\rho$ to  $\pi_1(P)$ is non-elementary.
\end{lemma}

\begin{proof}
Suppose, to the contrary, that $ \rho( \po(P) )$ is elementary. 
Then, since $\rho(\ell)$ is loxodromic, the limit set $\Lambda$ of $\rho( \po(P) )$ contains only the two fixed points of the loxodromic $\rho(\ell)$.
Then the domain of discontinuity of $ \rho( \po(P) )$ is identified with $\C \minus \{0\}$, and it admits a complete Euclidean metric given by the exponential map $\exp\col \C (\cong \R^2) \to \C \minus \{0\}$, so that the $ \rho( \po(P) )$-action on the domain is isometric.

Let $\CC_P$ denote the  connected component  of $\CC_\In$ corresponding to $P$.
Consider the inverse image of  $\Lambda$ under $dev(\CC_P)$.
Since $\Lambda$ is in particular a discrete subset of $\rs$,  the $dev(\CC_P)$-inverse image of $\Lambda$ is a discrete subset preserved by $\po(P)$, and it descends to a set of finitely many points on $\CC_P$.
By pulling back the Euclidean metric  via the developing map,  $\CC_P$ minus the finitely many points carries a complete Euclidean metric.
This is a contradiction since the Euler characteristic of $P$ is negative.
\end{proof}

Let $\td{\CC}_i$ and $\td{\CC}_\In$ denote the universal covers of $\CC_i$ and $\CC_\In$, respectively.  
Then, by Lemma \ref{4-20no2}, the holonomy of every connected component of $\CC_\In$ is non-elementary. 
Thus, by Corollary  \ref{6-3no2}, let $\CC_\In \cong (\sigma_\If, N_\If)$ be the Thurston coordinates, where $\sigma_\In$ is a convex hyperbolic surface with geodesic boundary whose interior is homeomorphic to $S \sm \ell$ and $N_\In$ is a (possibly heavy) measured geodesic lamination on $\sigma_\If$.
(Recall from Theorem \ref{6-3no1},  each boundary component of $\sigma_\If$ may be either open  or closed, i.e. entirely contained in $\sigma_\If$ or not contained in $\sigma_\If$ at all; moreover a closed boundary is must be a leaf of the Thurston lamination and its weight is $\infi$, which corresponds to a half infinite cylinder. )
We have
\begin{proposition}\Label{4-27no2}
The boundary of  $\sigma_\infi$ is the union of   two  geodesic loops  corresponding the boundary circles of $S \minus \ell$. 
The lengths of boundary components are the translation length of $\rho(\ell)$.
Furthermore
\begin{itemize}
\item[(i)]  Suppose that $\ell$ is roughly circular.
Then $\sigma_\infi$ contains both boundary geodesic loops (i.e. closed boundary), and 
they are isolated leaves of $N_\If$ with weight infinity.  
\item[(ii)]  Suppose that $\ell$ spirals.
Then $\sigma_\In$ contains no boundary geodesic (i.e. open boundary).  
Leaves of the lamination $N_\In$ spiral towards each boundary component of $\sigma_\In$ in the same direction with respect to the orientation on          the boundary components of $\tau$ (induced by the orientation of $S$);  see Figure \ref{6-7no2}.
In particular $N_\In$ contains no heavy leaves.   

\end{itemize}

\begin{figure}[h]
\begin{overpic}[scale=.1,
]{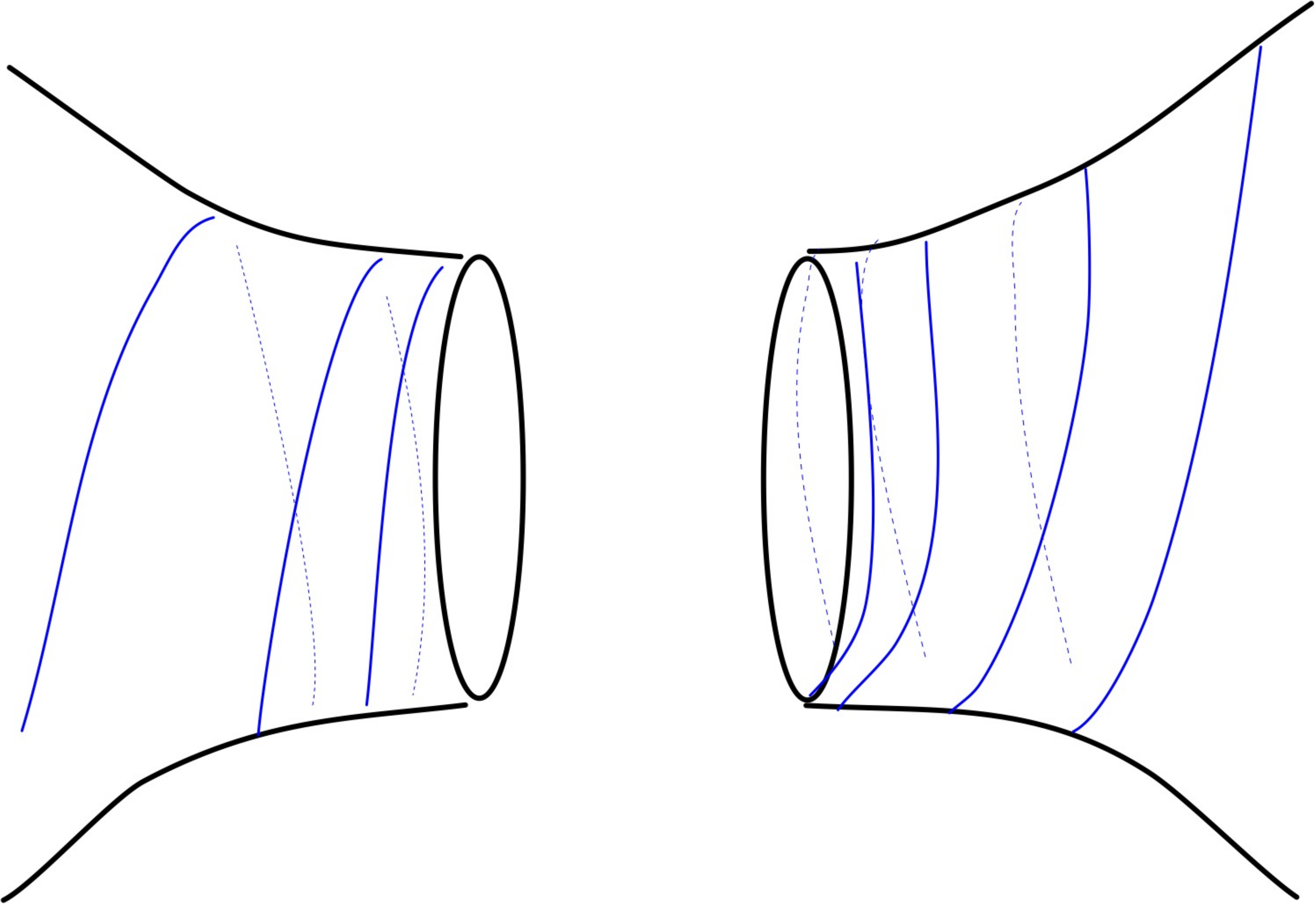}
      \end{overpic}
\caption{Geodesics spiraling to the left towards both boundary components  (when you stand on the surface facing toward boundary).}\label{6-7no2} 
\end{figure}
\end{proposition}
\begin{remark}
\begin{itemize}
\item In (ii), the metric completion of $\sigma_\In$ is  the union of $\sigma_\In$ and the boundary loops. 
 Then $N_i$ naturally extends to yet a heavy measured lamination on the completion, so that both boundary loops are leaves of weight infinity.
\item  Let $\zeta_\If\cn \td{\sigma}_\If \to \h^3$ be the pleated surface associate with $\CC_\If$, where $\td{\sigma}_\If$ is the universal cover of $\sigma_\In$ (note that, if $\ell$ is separating, $\td{\sigma}_\In$ has two connected components). 
Then,
let $m$ be a boundary geodesic of $\td{\sigma}_\In$ and let $\gam_m$ be a non-trivial deck transformation preserving $m$.
Then,  in  both Case (i) and (ii), $\zeta_\If$ isometrically takes  $m$ onto the axis of the loxodromic element $\rho(\gm_m)$.

\end{itemize}
\end{remark}

\proof[Proof of Proposition \ref{4-27no2}]
{\it Case (i).}
Suppose that $\ell$ is roughly circular. 
Let $\AA$ be a component of $\CC_\In \minus \CC_0$, which is a half infinite (grafting) cylinder attached a boundary component of $\CC_0$. 
As $\ell$ is roughly circular, the developing map of $\AA$ is the restriction of  the exponential map $\exp\col \C \to \C^\ast$ to a region $\til{\AA}$ bounded by a property embedded curve in $\C$ whose imaginary coordinate is bounded form below and above. 
Then we can in addition assume that $\pi_1(\AA) \cong \Z$ acts on $\til{\AA}$ by translations by integers (conjugating by an element of $\PSL$).  
Then we can find  a circular loop $\ap$ in $\til{\AA}$, which lifts to a horizontal line in $\R$.
Then $\ap$ bounds  a (smaller) half-infinite cylinder $\AA'$ isotopic to $\AA$, and $\AA'$ is uniquely foliated by circular loops. 
Let $\ell'$ be such a circular loop in $\AA'$.

Let $\NN_\If$ be the circular measured lamination on $\td{\CC}_\In$, which descends to $N_\In$.
 In  $\til{\AA}$, one can easily find two horizontal parallel lines of distance  $\pi$ apart  $\til{\AA}$. 
Then the  regions bounded by such two lines is a maximal ball, and its core is the horizontal line in the middle of them.
 The core descends to  a closed leaf  $\ell'$ of $\NN_\In$.
In addition we can assume that  $\AA'$ is a maximal cylinder in $\CC_\infi$ that is isotopic to $\AA$  such that $\AA'$ is foliated by closed leaves of $\NN_\If$.
(Then $\AA'$ may not be contained in $\AA$ anymore.)
Then $\AA'$ is still half-infinite and the total transversal measure of $\NN_\If$ on $\AA'$ is infinite. 
Let $\iota_\In\cn \CC_\In \to \sigma_\If$ be the collapsing map. 
The $\iota_\In$ takes $\AA'$ to a boundary geodesic loop of $\sigma_\In$ of infinite weight.
Conversely the inverse-image of the geodesic loop is $\AA'$ since $\AA'$ is maximal.
Since $\NN_\If$ has infinite measure on $\AA'$, the boundary component is a leaf of $N_\In$ with infinite weight. 
Then, since the transversal measure of $\NN_\In$ is locally finite,   no leaf of $\NN_\In$ spirals towards the boundary loop of $\AA'$.
 Thus the boundary geodesic of $\sigma_\In$ is an isolated leaf of $N_\In$.

(ii).
Suppose that $\ell$ spirals. 
Since $\ell$ is admissible, the restriction of  $dev(C)$ to a lift  $\td{\ell}$ to $\td{S}$ is a simple curve  on $\rs$,  and we can assume that it connects  $0$ and $\In$.
Then, as in Definition \ref{4-27no3},  it lifts to a curve  
\begin{eqnarray*}
{\bf l} \cn \R &\to& \R^2   \\
t &\mapsto& (\theta(t), r(t))
\end{eqnarray*}
through $\exp\cn \C \cong \R^2 \to \C \minus \{0\}$, where $\theta \cn \R \to \R$ and $r\cn \R \to \R$ are continuous functions, so that $\exp (r(t) + i \theta(t))$ is the curve $dev(C) | \td{\ell}$.
Since $\ell$ is admissible,  ${\bf l}$ is a simple curve.
Since $dev(C)|\td{\ell}$ is preserved by the loxodromic $\rho(\ell)$, 
accordingly ${\bf l}$ is preserved by a nontrivial translation of $\R^2$ (along a geodesic).
Since $\rho(\ell)$ is loxodromic and $\ell$ is  spiraling,  the axis of this translation intersects both $\theta$ and $r$-axes transversally.
Thus one connected component of $\R^2 \sm {\bf l}$ lies {\it above} ${\bf l}$, i.e. it contains $\{ 0 \} \times  [R, \In)$ for sufficiently large $R > 0$, and the other component lies {\it below}. 
Let $c$ be a boundary component of $\CC_0$, which is isomorphic to $\ell$. 
Let $\td{c}$ be a lift of $c$ to the universal cover $\td{\CC}_0$, so that $\til{c}$ is isomorphic to $\td{\ell}$. 
Then we can assume that $dev(\CC_0)$ takes  the small neighborhood of $\td{c}$ into the region below ${\bf l}$ in $\R^2$, if necessary, by exchanging $0$ and $\If \in \rs$ by an element of $\psl$.

Cearly $\td{\CC}_0$ is isomorphically embedded in  $\td{\CC}_\In$.
Then in particular $\td{c}$ is embedded in $\td{\CC}_\In$ and we can regard  the endpoints of $\td{c}$ as distinct ideal points of both $\td{\CC}_0$ and $\td{\CC}_\In$.
Then let ${\tt p}^+$ and ${\tt p}^-$ be the ideal points corresponding to $\If$ and $0$, respectively, via ${\bf l}$.
\begin{lemma}\Label{5-3no1}
There is a maximal ball $B$ in $\td{\CC}_\In$ such that  ${\tt p}^+$ is an ideal point of  $B$.
\end{lemma}
\begin{proof}
Let $A$ be the connected component of $\CC_\In \minus \CC_0$ bounded by $c$, so that $A$ is a half-infinite grafting cylinder.
Let $\td{A}$ be the corresponding connected component of $\td{\CC}_\In \sm \td{\CC}_0$ bounded by $\td{c}$, so that $\td{A}$ covers $A$.
Then $dev(A)$ lifts, through $\exp$,   to an embedding onto the component of $\R^2 \minus {\bf l}$ above ${\bf l}$. 
 A {\it round ball} on $\rs$ is a ball bounded by a round circle.
Thus we can find a round ball $B$ contained in $\td{A}$ such that ${\tt p}^+$ is an ideal point  of $B$ ; see Figure \ref{6-6no1}.
Since $\td{A} \st \td{\CC}_\In$,  there is a desired maximal ball in $\td{\CC}_\In$  that contains $B$. 
\end{proof}
 
\begin{figure}[h]
\begin{overpic}[scale=.2,
]{4-27no1.pdf}
     \put(43, 45){$\In$}
     \put(8,20){$f \cc \td{\ell}$}
     \put(45, 25){$B$}
      \end{overpic}
\caption{}\label{6-6no1}
\end{figure}

\begin{lemma}\Label{6-7no3}
There is no maximal ball $B$ in $\td{\CC}_\In$ such that ${\tt p}^+$ and ${\tt p}^-$ are both its ideal points.
\end{lemma}

\begin{proof}
Suppose, to the contrary, that there is a maximal ball $B$ in $\td{\CC}_\In$ such that $\pt_\If B$ contains both ${\tt p}^+$ and ${\tt p}^-$.
Then its core $\Core(B)$  contains a circular arc $\ap$ connecting ${\tt p}^+$ and ${\tt p}^-$. 
The surface $\td{C}_\In$ is of hyperbolic type (topologically). 
Then since $\ap$ and $\td{c}$ share their ideal end points, they must project to isotopic  loops on $C_\In$. 
Thus $\ap$ covers a circular loop on $\CC$ isotopic to $c$.
This contradicts that $c$ spirals. 
\end{proof}

Next we show that $\sigma_\In$ has an open boundary component that is a closed geodesic  homotopic to $c$.
Let $B$ be a maximal ball of $\td{\CC}_\In$ given by Lemma \ref{5-3no1} so that $\pt_\If B \ni {\tt p}^+$.
Then the collapsing map $\td{\iota}_\In\col \til{\CC}_\infi \to \til{\sigma}_\infi$ projects  $\Core(B)$ onto a convex subset $X$ of $\td{\sigma}_\In$. 
Moreover $\td{\iota}_\infi$ continuously extends to a map from the ideal boundary of $\til{\CC}_\infi$ to the ideal boundary of $\til{\sigma}_\infi$.
Since $c$ is an essential loop, there are  points $p^+$ and $p^-$ on the ideal boundary of $\td{\sigma}_\In$ corresponding to ${\tt p}^+$ and ${\tt p}^-$, respectively.
Then, by Lemma \ref{6-7no3},  $\pt_\In X \,(\st \pt_\In \h^2)$ contains $p^+$ but not $p^-$. 

By regarding $\td{\sigma}_\In$ as a convex subset of $\h^2$, there is a unique  geodesic  $g$  connecting $p^+$ and $p^-$ in $\h^2$.
Let $\gam_c$ be a deck transformation corresponding to $c$ so that $\gam_c$ preserves $g$.
Then we see that $(\gam_c)^j X$ converges to $g$ uniformly on compacts as $j \to \In$, if necessary, changing $\gam_c$ to its inverse (Figure\ref{6-7no1}).
By Lemma \ref{6-7no3},  the geodesic $g$ is not contained in $\td{\sigma}_\In$.
Thus $g$ descends to a desired open boundary component of $\sigma_\If$.
Let $\zeta_\If\cn \td{\sigma}_\In \to \h^3$ be the pleated surface for $C_\In$.
Then since $\zeta_\If$ is equivariant and 1-Lipschitz,  the continuous extension of $\zeta_\If$ takes  $g$ isometrically onto  the axis of the loxodromic $\rho(c)$.
Therefore the translation length of $\rho(c)$ is the length of the boundary component of $\sigma_\In$ homotopic to $c$. 
In addition the convergence $(\gam_c)^j X \to g$ implies that leaves of $N_\infi$ spiral towards $c$.

\begin{figure}[h]
\begin{overpic}[scale=.3
]{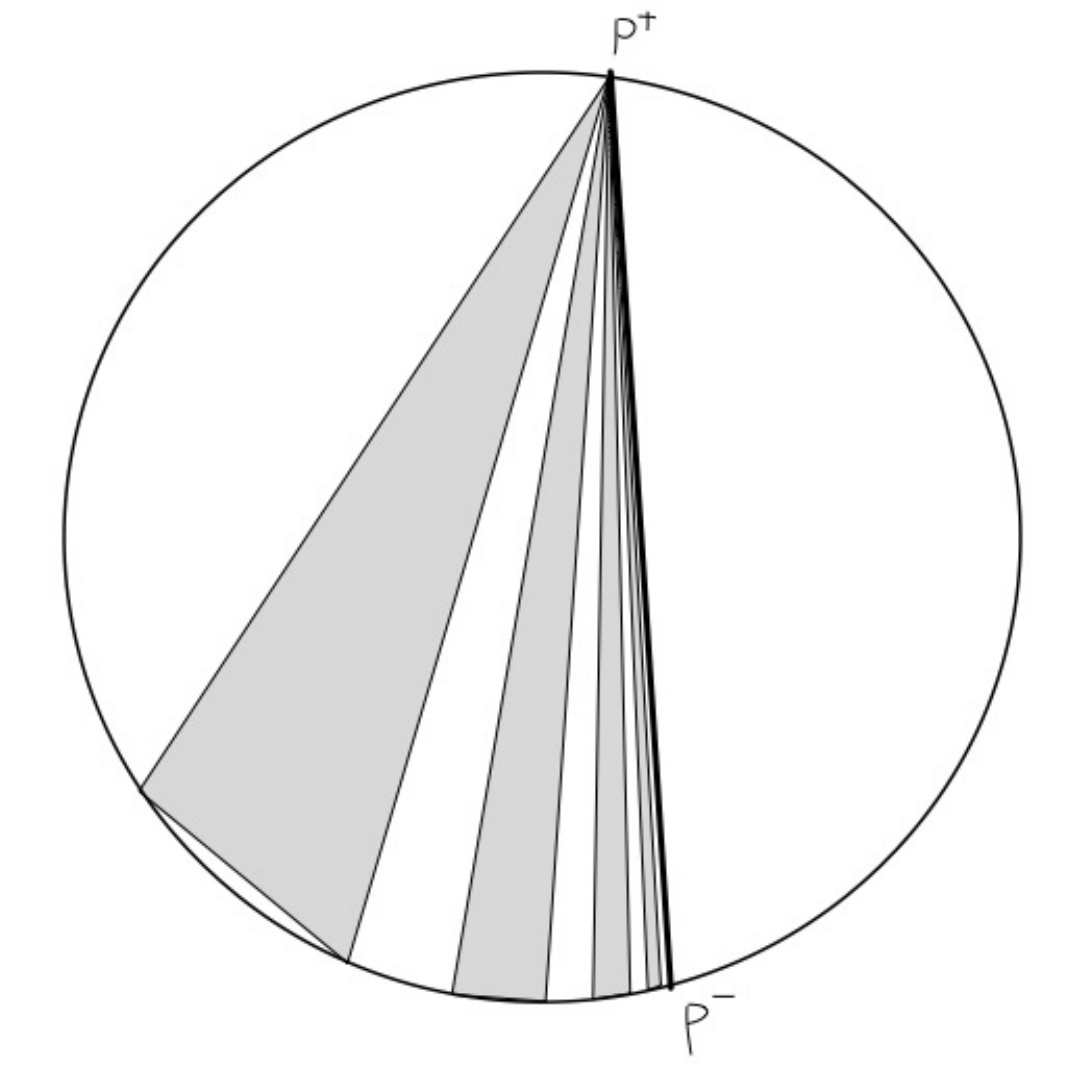}
      \put(64,50){$g$}
      \put(20,30){$(\gam_c)^j X$}
     
      \end{overpic}
\caption{}\label{6-7no1}
\end{figure}

Recall that  a small neighborhood of the boundary component $c$ of $C_0$ in $\AA$ develops above ${\bf l}$, which is used to distinguish $p^+$ and $p^-$. 
Letting $c'$ be the other boundary component of $C_0$, then $dev (C_0)$ takes a small neighborhood of $c$ to the region below ${\bf l}$ (with respect to  $c = c'$ on $C$).
Then it follows that the labels $p^+$ and $p^-$ are opposite for  $c'$ and $c$. 
However, since the normal directions of $c$ and $c'$ of $C_0$ are the opposite on $C$, leaves of
$N_\In$ spiral towards both boundary components in the same direction with respect to the normal directions of the boundary components.
\Qed{4-27no2}

\section{Identification of  boundary components of the limit structure}\Label{identification}
We have obtained the Thurston coordinates $(\sigma_\If, N_\In)$ of $\CC_\In$ (Proposition \ref{4-27no2}). 
In particular the boundary components of $\sigma_\If$ are two closed geodesic whose lengths are equal to the translation length of the loxodromic $\rho(\ell)$ corresponding to the admissible loop $\ell$ on $C$.
Thus we can identify the boundary components of $\sigma_\If$ and obtain a hyperbolic structure $\tau_\In$ on $S$ and a heavy measured lamination $L_\In$ with a unique heavy leaf homotopic to $\ell$. 
Although this (isometric) identification is a priori unique up to sharing along the heavy leaf,  in fact
\begin{lemma}\Label{7-17-12no2}
There is a unique identification of the boundary components of $\sigma_\If$ so that the resulting pair $(\tau_\In, L_\In)$ is realized by a $\rho$-equivariant  pleated surface that coincides, in the complement of $\ell$,  with  the pleated surface $\iota_\In$ corresponding to $\CC_\In$. 
\end{lemma}

\begin{proof}
Let $\td{\ell}$ be a lift of $\ell$ to $\St$. 
Let $\hat{\ell}$ be the total lift of $\ell$ to $\St$. 
Let $\gm_\ell$ be the element of $\po(S)$ that corresponds to $\ell$ and preserves $\td{\ell}$.
Let $P_1$ and $P_2$ be the connected components of $\St \sm \hat{\ell}$ that are adjacent along $\td{\ell}$.

For each $k = 1,2$, let $\iota_{P_k}\cn X_k \to \h^3$ denote the pleated surface for the connected component of $\CC_\In$ corresponding to $P_k$ so that  $\iota_{P_k}$ is equivariant under the restriction of $\rho$ to the subgroup $\po(S)$ that preserves $P_k$.
Let $g_k$ be the boundary geodesic of $X_k$  corresponding to $\td{\ell}$.
Then $\iota_{P_k}$ isometrically takes $g_k$ to the axis of $\rho(\gm_\ell)$.
Thus there is a unique identification of $g_1$ and $g_2$ so that $\iota_{P_1}$ and $\iota_{P_2}$ continuously extends the union $X_1 \cup X_2$ given by the identification.
Then, by quotienting out $X_1 \cup X_2$ by the infinite cyclic group generated by $\gm_\ell$, the identification of $g_1$ and $g_2$ descends to a unique isometry between the boundary components of $\sigma_\In$.

Since $\CC_\In$ is obtained by grafting  $C$ and $dev(C)$ is $\rho$-equivariant, the identification of the boundary components of $\sigma_\In$ is independent of the choice of the lift $\td{\ell}$.
Then, applying the identification of boundary components for all adjacent components of $\til{S} \minus \hat{\ell}$,  we obtain a $\rho$-equivariant pleated surface from $\h^2$ to $\h^3$ realizing $(\tau_\In, L_\In)$.
\end{proof}

\section{Convergence of canonical neighborhoods under grafting}\Label{sNeighborhoodsConverging}

Recall, from \S \ref{6-8no1}, that $C_i = \Gr_\ell^i(C)$, where $\ell$ is an admissible loop on a projective structure $C$ on $S$, and  $\CC_i = C_i \minus \ell_i$ where $\ell_i$ is an isomorphic copy of $\ell$ which sits in the ``roughly middle''' of the cylinder inserted by $\Gr_\ell^i$.
For each $i \in \n$, let $e_i\cn \CC_i \to \CC_\If$ denote the canonical isomorphic embedding. 
Then  the embeddings $e_i$ give an exhaustion of $\CC_\If$,
  $$\CC_1 \st \CC_2 \st  \CC_3 \st \cdots (\st \CC_\If).$$
Let $p$ be a point on $\CC_\In$, and let $\td{p}$ be a lift of $p$ to the universal cover $\td{\CC}_\If$. 
Let $U_\If(\td{p}) \st \td{\CC}_\In$ be the canonical neighborhood (\S\ref{5-4no1}) of $\td{p}$.

For each $i \in \n$, let $\td{C}_i$ be the universal cover of $C_i$.
If $\ell$ is non-separating, then let $\hat{C}_i$ be the quotient of  $\td{C}_i$  by   $\pi_1(S \minus \ell)$.
If $\ell$ is separating, for each connected component  $F$ of $S \minus \ell$, quotient $\td{C}_i$ by  $\pi_1(F)$, and let $\hat{C}_i$ be the disjoint union of both quotients. 
Then $\CC_i$ is isomorphically embedded in $\Ch_i$.
\begin{displaymath}
\xymatrix{
\Ch_i \ar[d] \\
C_i & ~\CC_i~ \ar@{_{(}->}[l]  \ar@{_{(}->}[ul] \ar@{^(->}[r]^{e_i} & \CC_\If
}
\end{displaymath}

For sufficiently large $i$, we have  $p \in \CC_i \st C_i$.
Accordingly $\td{p} \in \mCt_i \st \Ct_i$.
Let $U_i$ and $\UU_\In$  be the canonical neighborhoods of the point $\td{p}$ in $\Ct_i$ and  $\td{\CC}_\If$, respectively.
Fix any metric on $\rs$ inducing the (standard) topology of $\rs$ (e.g. a spherical metric).
Note that canonical neighborhoods embed into $\rs$ by developing maps.  
We consider  a version of the Hausdorff metric: 
For two proper subsets $X$ and $Y$ of $\rs$, $X$ and $Y$ are
  \textit{$\ep$-close} if
 \begin{itemize}
 \item the Hausdorff distance of $X$ and $Y$ is less than $\ep$ and
 \item the Hausdorff distance of $\pt X$ and $\pt Y$ is less than $\ep$. 
\end{itemize}
With this distance on the subsets on $\rs$, we have
\begin{proposition}\Label{4-23no2}
$U_i$  converges to $\mc{U}_\If$ as $i \to \If$.
\end{proposition}

\proof[Proof of Proposition \ref{4-23no2}]
Let $\UU_i$ be the canonical neighborhood of $\td{p}$ in $\td{\CC}_i$.
Then, since $\td{\CC}_i$ embeds into $\Ct_i$ and $\td{\CC}_\If$, canonically $\UU_i \st U_i$ and $\UU_i \st \UU_\If$.

\begin{lemma}\Label{7-18-12no3}
$\UU_i$ converges to $\UU_\If$   as $i \to \In$.
\end{lemma}

\begin{proof}
For every $\ep > 0$, there exits finitely many closed round balls $B_1, B_2 \dots, B_n$ in $\td{\CC}_\If$ containing $\td{p}$ such that $\cup_{k = 1}^n B_k$ is $\ep$-close to $\UU_\If$ in $\rs$.
Since $\{\td{\CC}_i \}$ exhausts $\td{\CC}_\If$,  thus each $B_k$ is also contained in $\td{\CC}_i$ for sufficiently large $i$. 
Thus $\td{\CC}_i$ contains $\cup_{k =1}^n B_k$,  and therefore $\cup_{k =1}^n B_k$ is contained in $\UU_i$. 
\end{proof}
Since canonical neighborhoods are topologically open balls,  by  $\UU_i \sub U_i$ and Lemma \ref{7-18-12no3}, it suffices to show that for every $\ep > 0$, if $i$ is large enough, then $\pt U_i$ is contained in the (honest) $\ep$-neighborhood of $\pt \UU_\If$.

\begin{lemma}\Label{6-24no2}
Given any $x \in \pt_\If \UU_\If$ and any neighborhood $\VV_x$ of $x$ in $\rs$, then  $\VV_x$ is not a subset of $U_i$ for sufficiently large $i \in \N$.  
\end{lemma}

\begin{proof}
Suppose that the assertion fails; then there is a neighborhood $\VV_x$ of $x$ in $\rs$, such that, for every  $n \in \N$, there is  $i > n$ with $\VV_x \st U_i$.
Let $\NN_\In$ be the circular lamination of $\CC_\In$ that descends to $N_\In$, and
 let $\td{\NN}_\In$ be the total lift of $\NN_\In$ to the universal cover $\td{\CC}_\In$.

Since $\NN_\infi$ is nonempty, the endpoints of leaves of $\til{\NN}_\infi$ is dense in the ideal boundary $\bdry_\infi \UU_\infi (\sub \bdry_\infi \til{\CC}_\infi)$. 
Therefore we can in addition assume that $x$ is an endpoint of a leaf of $\til{\NN}_\infi$. 
Then the leaf contains a ray $\td{r}\cn [0, \In) \to \td{\CC}_\In$ ending at $x$. 
Let $r$ be the projection of $\td{r}$ to $\CC_\In$.

For every $s > 0$, since $r| [0, s]$ is a compact subset of  $\CC_\infi$, thus, for sufficiently large $i$, it is also a circular curve in $\CC_i$ and thus in $C_i$.
Accordingly $\til{r}| [0,s]$ is a circular arc embedded in $\til{C}_i$.
Since $\til{r}$ ends at $x$, if $s > 0$ is sufficiently large,  $\td{r}\vert [s, \In)$ is contained in $\VV_x$.
Therefore $\td{r}$ is contained in a compact subset of $U_i$ for sufficiently large $i$. 
Thus we induce a contradiction, showing that $x$ is an ideal point of  $\td{C}_i$ for sufficiently large $i$. 

First suppose that $r$ stays in the compact subset of $\CC_\In$.
Then, since every compact subset of $\CC_\If$ naturally embeds into $C_i$ for sufficiently large $i$, accordingly $\td{r}$ is naturally a circular ray in $\td{C}_i$  limiting to a point of $\pt_\In\td{C}_i$.

Next suppose that no compact subset of $\CC_\In$ contains $r$. 
Then the admissible $\ell$ loop must spiral--- otherwise $\ell$ is roughly circular, and every leaf of $\NN_\If$ is contained in a compact subset of $\CC_\In$.
The projection of $r$ to $\sigma_\In$ is an (eventually simple) geodesic ray spiraling towards  a boundary component $c$ of $\sigma_\If$.
Then $x$ is a fixed point of the corresponding loxodromic element.
For each $j \in \N$,  let $b_j$ be the boundary component of $\CC_j$ isotopic to $c$, so that $b_j$ isomorphic to $\ell$.
Then $b_j$ are parallel in $\CC_\infi$ and  $r \in \CC_\infi$ intersects $b_i$ for sufficiently large $i$.
Let  $\td{b}_j\cn \R \to \td{\CC}_j$ be a (parametrized) lift of $b_j$ so that $\td{b}_j(t)$ limits to $x$ as $t \to \If$.
Then $\VV_x$ contains $b_j(t)$ for all $t  > t_j$ with some $t_j$.
Therefore $U_i$ must contain  $b_j(t)~ (t > t_j)$ as well.   
Since $x$ is also an ideal point of $\til{C}_i$. 

 \begin{figure}[h]
\begin{overpic}[scale=.2
]{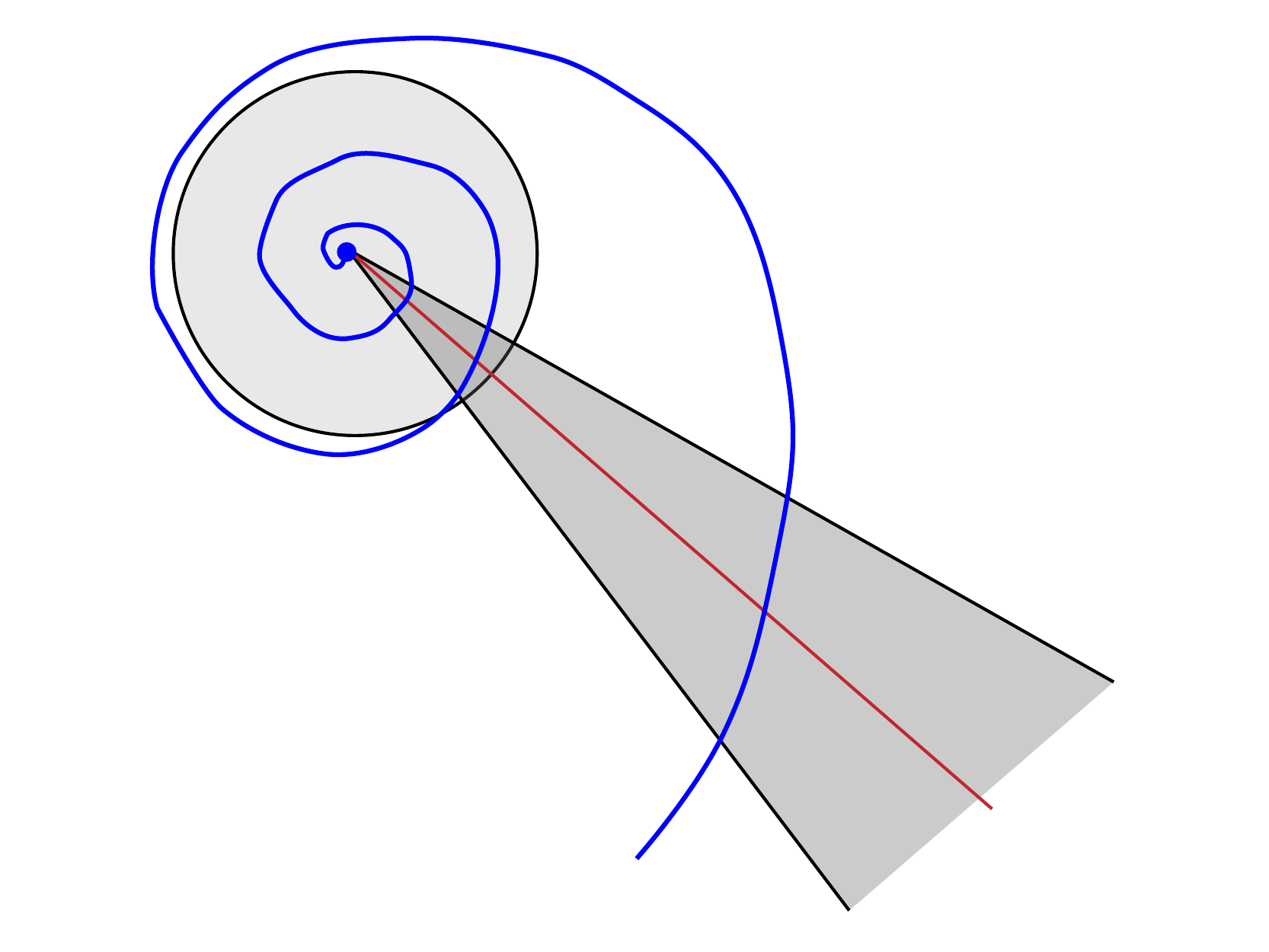}
         \put(62 ,11 ){\textcolor{black}{P}}
        \put(62,50){\textcolor{blue}{$\td{b}_j$}}
        \put(18, 50 ){\textcolor{blue}{$x$}}
        \put(70, 20){\textcolor{BrickRed}{$r$}}
       \put(41, 60){\textcolor{black}{$\VV_x$}} 
      \end{overpic}
\caption{}\label{  }
\end{figure}

\end{proof}

Then Lemma \ref{6-24no2} implies that 
\begin{corollary}\Label{011213}
For every $\ep > 0$, if $i \in \N$ is large enough, then the $\ep$-neighborhood of $\UU_\If$ contains all maximal balls $B_i$  in $\Ct_i$ containing $\td{p}$. 
\end{corollary}

\begin{proof}
The point  $\td{p}$ is contained in  $\rs \minus \pt_\If \UU_\If$.
Note $ \pt_\If \UU_\If$ is a compact subset of $\rs$.
Then $\rs \minus \pt_\If \UU_\If$ carries a canonical projective structure, and  $\UU_\If$  naturally is isomorphic to the canonical neighborhood of $\td{p}$ in the complement. 

For every $\del > 0$, take finitely many points  $x_1, \dots, x_{n(\del)}$ in $\pt_\If \UU_\If$ so that their $\del$-neighborhoods $\VV_1, \dots, \VV_{n(\del)}$ cover $\pt_\If \UU_\If$. 
Then their union $\VV_1 \cup \dots \cup \VV_{n(\del)}$ converges to $\pt_\In \UU_\If$ as $\del \to 0$ in the Hausdorff topology.
By Lemma \ref{6-24no2}, if $i$ is sufficiently large, there is a point $y_k$ in $\VV_k$ that is not contained in $U_i$ for each $k = 1, \dots, n(\del)$. 
Then $U_i$ is contained in the canonical neighborhood of $\td{p}$ in  the punctured sphere $\rs \minus \{y_1, \dots,  y_{n(\del)}\}$.
Since $\{y_1,  \dots, y_{n(\del)}\}$ converges $\pt_\If \UU_\If$ as $\del \to 0$, for every $\ep > 0$, if $i$ is sufficiently large,  $U_i$ is contained in the $\ep$-neighborhood of $\UU_\In$.
\end{proof}
Corollary \ref{011213} immediately implies that $U_i$ is contained in the $\ep$-neighborhood of $\UU_\In$. 
Then $\pt U_i$ is contained in the $\ep$-neighborhood of $\pt \UU_i$, since $\UU_i \st U_i$ and $\UU_i \to \UU_\In$.
\Qed{4-23no2}

%
%
%
%
%
%
%
\section{Convergence of domains in $\rs$ and Thurston coordinates}\Label{6-30no1}

Let $R$ be a subset of $\C$ homeomorphic to an open disk.
By Theorem \ref{6-3no1}, the projective structure on $R$ has Thurston coordinates  $(\h^2, L)$, where $L$ is a measured lamination on $\h^2$ (note that $L$ contains no heavy leaf since $R$ is embedded in $\rs$).   
Let $\beta\cn \h^2 \to \h^3$ denote the corresponding pleated surface. 
Let $\mL = (\nu, \omega)$ be the circular measured lamination on $R$ that descends to $L = (\ld, \mu)$ via the collapsing map $\kp\cn R \to \h^2$, where $\ld = |L|$ and $\mu \in \mathcal{TM}(\ld)$.

Fix a conformal identification of  $\rs$ with $\s^2$ in oder to fix a spherical Riemannian metric on $\rs$. 
Then let $\{R_i\}$ be a sequence of regions  in $\rs$ homeomorphic to an open disk, such that $\bdry R_i \to \bdry R$ and $\rs \minus R_i \to \rs \minus  R$ in the Hausdorff topology.
For each $i$, we similarly let  $(\h^2, L_i )$ denote Thurston coordinates of the projective structure on $R_i$\,; let $\beta_i\cn \h^2 \to \h^3$ be the corresponding pleated surface; let $\mL_i = (\nu_i, \omega_i)$ be the circular measured lamination on $R_i$ that descends to $L_i = (\ld_i, \mu_i)$ via the collapsing map $\kp_i\cn R_i \to \h^2$.

Every compact subset of $R$ is also a compact subset of $R_i$ for sufficiently large $i$.
In particular, for every compact subset $K$ of the target $\h^2$ of $\beta$, take $i$ is large enough,  so that $R_i$ contains the compact set $\kap^{-1}(K)$ of $R$.
For each $x \in K$, pick a point in $y_x \in \kap^{-1}(x)$. 
(Note that if $x$ in on a leaf of $L$ with atomic measure, then $\kap^{-1}(x)$ is a circular arc.) 
Then define  $\psi_i\col K \to \H^2$ by $\psi_i(x) = \kap_i(y_x)$.
Note that this  map $\psi_i$ is not necessarily unique or continuous. 

\begin{theorem}\Label{4-23no4}

\begin{itemize}

\item[(i)] $\mL_i$ converges to $\mL$, uniformly on compacts, via the convergence of $R_i$ to $R$. 

\item[(ii)] $L_i$ converges to $L$ pointwise.
 
\item[(iii)] $\psi_i$ converges to an isometry  uniformly on compacts; $\beta_i \cc \psi_i\col \h^2 \to \h^3$ converges to $\beta\cn \h^2 \to \h^3$  uniformly on compacts. 
\end{itemize}
\end{theorem}

\begin{remark}\Label{rem:convergence}
In (i), by the uniform convergence, we mean that,  for every $\ep > 0$ and every compact subset $K$ of $R$, if $i$ is sufficiently large, then given any $p, q \in K$,  $\omega(p,q)$  is $\ep$-close to $\omega_i(p,q)$,  where  $\omega_i(p,q)$ and $\omega(p,q)$ denote the transversal measures of the geodesic segments connecting $p$ to $q$ on $R_i$ and $R$, respectively,  in the Thurston metric.
In (ii), by the pointwise convergence, for any $p, q \in R$ not on leaves with positive weight,  $\mu_i(p,q) \to \mu(p,q)$ as $i \to \infi$.
In (iii), for every compact subset $K$ of $\h^2$,  $\psi_i$ is $\ep_i$-rough isometry with the sequence $\ep_i > 0$ converging to $0$. 
The convergence, $\beta_i \circ \psi_i \to \beta$ is with respect to the sup norm.
\end{remark}

Note that (i) implies (ii) by the definition of $\psi_i$.
The rest of \S\ref{6-30no1} is the proof of Theorem \ref{4-23no4}.
For each point $x \in R$, let $B(x)$ be the maximal ball in $R$ centered at $x$.
For sufficiently large $i$, we have $x \in R_i$.
Thus let  $B_i(x)$ be the maximal ball  in $R_i$ centered at $x$.

\begin{proposition}\Label{4-26no2} 
(i) For every  $\ep > 0$ and every compact subset $K$ of $R$, if $i \in \n$ is sufficiently large, then $B_i(x)$ is $\ep$-close to $B(x)$ for every $x \in K$.
(ii) For every $\ep > 0$ and $x \in R$, there is a neighborhood $U_x$ of $x$ in $R$, such that, if $i$ is sufficiently large, then the ideal boundaries $\pt_\If B_i(y)$ and $\pt_\In B(y)$ are contained in the $\ep$-neighborhood of $\pt_\If B(x)$ in $\rs$ for all $y \in U_x$.
 \end{proposition}

\proof[Proof of \ref{4-26no2}]
 (See also the proof of Theorem 4.4 in \cite{Kullkani-Pinkall-94}.)
For every compact subset $X$ of the Euclidean plane $\R^2$, there is a unique closed round ball $D = D(X)$ of least radius containing $X$.
Let $\pt_X (D)$ be the intersection of $X$ with the boundary circle of $D$.
Then the minimality implies that the convex hull of $\pt_X(D)$ (for the Euclidean metric) contains the center of $D$.
In addition, the uniqueness of $D$ implies that $D$ changes continuously when $X$ changes continuously in the Hausdorff metric. 
Therefore, for every $\ep > 0$,  there exists $\del > 0$ such that if $Y$ is a compact subset of $\R^2$ that is $\del$-close to $X$, then, letting $D_Y$ be the round ball of least radius containing $Y$,
 the $\ep$-neighborhood of $\pt_X(D)$ contains $\pt_Y(D_Y)$.

If $x$ be a point in $R$, regarding $\rs = \R^2 \cup \{\In\}$,  we can assume $x = \{\In \}$ by the $\PSL$-action on $\rs$.
Note that the round $B(x)$ is the complement of $D(\rs \minus R)$ in $\rs$.
In addition $\pt_\In B(x) = \pt_X D( \rs \minus R)$.
If $U$ is a neighborhood of the identity element in $\psl$, then $U x$  and $U^{-1} x$ are  neighborhoods of $x = \{\In \}$ in $\rs$. 
Thus for every $\del > 0$, if the neighborhood $U$ is sufficiently small, for every $\gm \in U$, $R$ and $\gm R$ are $\del$-close. 
Therefore, it follows from the preceding paragraph that, for $\ep > 0$, if $U$ is sufficiently small, then, letting $y = \gm^{-1} x$, $B(x)$ and $B(y)$ are $\ep$-close and the $\ep$-neighborhood of  $\pt_\In B(x)$ contains $\pt_\In B(y)$. 

Since $\rs \minus R_i \to \rs \minus R$ and $\gm R_i$  changes continuously in $\gm \in U$, for every $\del > 0$, if $i$ is sufficiently large and $U$ is sufficiently small, then $\R^2 \minus \gm R_i$ is $\del$-close to $\R^2 \minus R$.
Therefore,  for every $\ep > 0$, we can assume that the maximal balls $B(y)$ and $B_i(y)$ are $\ep$-close to $B(x)$ and the $\ep$-neighborhood of $\pt_\In B(x)$ contains $\pt_\In B(y)$, which proves (ii). 

 Thus  $B(y)$ and $B_i(y)$ are $2\ep$-close for all $y$ in the small neighborhood $U$ of $x$.
 Since $K$ is compact, this implies (i). 
\Qed{4-26no2}

Recalling $\rs \minus R_i$ converges to $\rs \minus R$ in $\rs$ as $i \to \In$,
let $e\cn R \cap R_i \to R$ and $e_i\cn R \cap R_i \to R_i$ be the trivial embeddings. 
Let $\phi = \beta \cc \kp \cc e\cn R \cap R_i \to \h^3$ and  $\phi_i =  \beta_i \cc \kp_i \cc e_i \cn R_i \cap R \to \h^3$.

\centerline{
\xymatrix{R \cap R_i \ar[r]^{e}  \ar[rd]^{e_i} & R \ar[r]^{\kp}
 &\h^2 \ar[d]^{\psi_i} \ar[dr]^{\beta}&  \\
& R_i \ar[r]^{\kp_i} & \h^2 \ar[r]^{\beta_i}& \h^3
}}

Note that $\nu$ and $\nu_i$ are circular and we can measure their intersection angle with respect to the spherical Riemannian metric on $\rs$.
Then 
\begin{corollary}\Label{6-26no1}
$\phi_i$ converges to $\phi$ uniformly on compacts in $R$ as continuous maps;  
therefore,  $\beta \cc \psi_i$ converges to $\beta_i$ uniformly on compacts as $i \to \In$ in the sup norm. 
\end{corollary}

\begin{proof}
By Proposition \ref{4-26no2} (i),  for every $\ep > 0$ and every compact subset $K$ of $R$, if $i \in \n$ is sufficiently large, then for every $p \in K$,  the maximal balls $B_i(p)$ and $B(p)$ of $R_i$ and $R$, respectively, at $p$ are $\ep$-close. 
Thus, for sufficiently large $i$, 
the orthogonal projection of $p$ into the totally geodesic hyperplane in $\h^3$ bounded by $\pt B_i$ is $\ep$-close to that into the hyperplane bounded by $\pt B_i$  for all $p \in K (\st \rs)$. 
Since these projections of $p$ are $\phi_i(p)$ and $\phi(p)$, the first assertion holds. 
Then the second assertion immediately follows from the definition of $\psi_i$.
\end{proof}

\begin{proposition}\Label{10-11no1} 
Let $K$ be an arbitrary compact subsurface in $R$.
Then $\angle_K (\nu_i , \nu) \to 0$ as $i \to \If$.
\end{proposition}
\begin{proof}

Since $K$ is compact, it suffices to show that, for every $\ep > 0$ and  $x \in R$,  if an open neighborhood $U_x$ of $x$ in $R$ is sufficiently small, then $\angle_{U_x}(\nu_i, \nu)  < \ep$ for sufficiently large $i$.

Suppose that $x$ is contained in a leaf  $\ell_x$ of $\nu$.
Let $\ell$ and $\ell_i$ be  leaves of $\nu$ and $\nu_i$, respectively, that intersect in $U_x$.
Let $B(\ell)$ be the maximal ball in $R$ whose core contains $\ell$, and 
let $B_i(\ell_i)$ be the maximal ball in  $R_i$ whose core contains $\ell_i$. 
Then, it follows from  Proposition \ref{4-26no2} (ii) that, for every $\ep > 0$, if $U_x$ is  small enough and $i$ is sufficiently large, 
then the endpoints of $\ell$ and $\ell_i$ are sufficiently close to the end points of $\ell_x$ so that  $\angle_{U_x}(\ell_i, \ell) < \ep$. 
Hence $\angle_{U_x}(\nu_i, \nu) < \ep$.

Suppose that $x \in R \minus |\nu|$.
Then take $U_x$ disjoint from $\nu$.
 Then $\angle_{U_x}(\nu_i, \nu) = 0$
\end{proof}

Let ${\sf d}\cn R \times R \to \R_{\geq 0}$  be  the continuous map obtained by  pulling back of the hyperbolic distance $\h^2 \times \h^2 \to \R_{\geq 0}$ via $\kp\cn R \to \h^2$. 
Then ${\sf d}$ is a pseudometric on $R$. 
In fact ${\sf d}$ coincides with the Thurston metric on each stratum of $(R, \LL)$. 
On the other hand, ${\sf d}$ does not measure the part of Thurston metric corresponding to the transversal measure of $\LL$.
In particular, the Euclidean region of $R$,  circular arcs orthogonal to $\nu$ have ``length'' zero since they map to single points on $\h^2$.
Similarly ${\sf d}_i\cn R_i \times R_i \to \R_{\geq 0}$ be the pseudometric on $R_i$ obtained via $\kp_i\cn R_i \to \h^2$.
\begin{proposition}\Label{5-17no1}
Let $K$ be a compact subset of $R$.
Then, for every $\ep > 0$, if $i \in \n$ is sufficiently large, then ${\sf d}$ and ${\sf d}_i$ are \textit{$\ep$-close} on $K \times K$, i.e.
 $| {\sf d}(x, y) - {\sf d}_i(x,y) | < \ep$ for all $x, y$ in $K$.
\end{proposition}

In the sense of Remark \ref{rem:convergence}, we have
\begin{corollary}
$\psi_i$ converges to an isometry uniformly on compacts; Theorem  \ref{4-23no4} (iii) holds.
\end{corollary}

\proof[Proof of Proposition \ref{5-17no1}]

First suppose that $\kap(K)$ and $\kap_i(K)$ do not intersect the Euclidean regions.
Then $\kap|K$ and $\kap_i|K$ are $C^1$-diffeomorphism onto their images, and ${\sf d}$ and ${\sf d}_i$ are both hyperbolic metrics on $K$.
Thus, by Proposition \ref{4-26no2} (i), for every $\ep > 0$, if $i$ is sufficiently large,  then for every unit tangent vector $v$ at a point $x$ in $K$,  the length of the derivative $\kap'(v)$ is $\ep$-close to that of $\kap_i'(v)$. 
Thus ${\sf d}$ and ${\sf d}_i$ are $\ep$-close in $K$.
This special case extends to  general Thurston metrics by:

\begin{proposition}\Label{5-17no2}
Let $P$ be a topological open disk in $\rs$ with $P \ncong \C$, so that $P$ has  Thurston coordinates (by Proposition \ref{6-3no1}).
Then, for every $\ep > 0$ and  every compact subset $K$ of $P$ homeomorphic to a closed disk, there is another topological open disk $Q$ in $\rs$ containing $K$ such that 
\begin{itemize}
\item[(i)] $P$ and $Q$ are $\ep$-close on $\rs$ in the Hausdorff metric, 
\item[(ii)]  in Thurston coordinates,  its lamination (on $\H^2$) has no leaf with atomic transversal  measure, and 
\item[(iii)] letting   ${\sf d}_P\cn P \times P \to \R_{\geq 0}$ be the pseudo metric and  ${\sf d}_Q\cn Q \times Q \to \R_{\geq 0}$ be the metric defined as above, then 
   ${\sf d}_P$ and ${\sf d}_Q$ are $\ep$-close in $K \times K$.
\end{itemize}
\end{proposition}

Indeed, by Proposition \ref{5-17no2},  we can take $Q$ and $Q_i$ for each $i$ that are close to $R$ and  $R_i$ on $\rs$, respectively, so that ${\sf d}_Q$ and ${\sf d}_{Q_i}$ are sufficiently close to ${\sf d}_R$ and  ${\sf d}_{R_i}$ on $K \times K$.
Then, by (ii),  the general case is reduced to the case that the Thurston laminations have no atomic measure.

\proof[Proof of Proposition \ref{5-17no2}]
For every $\del_1 > 0$, pick another open topological disk $Q$ contained in $P$ such that 
\begin{itemize}
\item[(1)] when projected to $\C$ by stereographic projection, $\pt Q$ is  a smooth loop embedded in $\rs$ and the sign of its curvature  changes at most finitely many times in the Euclidean metric, and \note{probably does not depend on }
\item[(2)] the Hausdorff distance of $\rs \minus P$ and  $\rs \minus Q$ is less than $\del_1$ ((i)). 
\end{itemize}

We can in addition assume that $Q$ also contains $K$ by taking sufficiently small $\del_1$.
Let $(\h^2, L_Q)$ be the Thurston coordinates of the projective surface $Q$. 
Set $\mL_Q = (\nu_Q, \omega_Q)$ to be the circular measured  lamination on $Q$ descending to $L_Q$.
By  the smoothness in (1), for each point of $\pt Q$, there is a unique maximal ball in $Q$ tangent, at the point, to $\bdry Q$.  
Thus every leaf of  $L_Q$  has no atomic measure ((ii)).
Moreover 
\begin{lemma}
The two-dimensional strata of $(Q, \LL_Q)$ are isolated. 
Therefore two-dimensional strata of $(\H^2, L_Q)$ are isolated.
\end{lemma}
\begin{proof}
Let $R$ be a two-dimensional stratum of $\LL_Q$, and let $B_R$ be the maximal ball in $Q$ whose core is $R$. 
Then, by the curvature condition, $B_R$ has only finitely many ideal points $p_1, \dots, p_n$.
Let $I_1, \dots, I_n$ be sufficiently small neighborhoods of $p_1, \dots, p_n$ in $\bdr Q$.
Since $\bdr Q$ is smooth, at every point  $x$ of $\partial Q$,  there is a unique maximal $B_x$ ball in $Q$ such that $x$ is an ideal point of $B_x$. 
Suppose that $x$ is in $I_i \minus p_i$ for some $i \in \{1, \dots, n\}$. 
Then as  $I_i$ is sufficiently small, $x$ is the unique ideal point of $B_x$ in $I_i$. 
Let $P_x$ be the connected component of $Q \minus \Core(R)$ whose boundary contains $x$. 
 Let $\ell$ be the circular boundary segment of $R$ bounding $P_x$, so $x$ is close to one of the end point of $\ell$ (Figure \ref{fIsolatedStratum}).
Then the ideal points of $B_x$ must be in $\bdr P_x \minus \ell$. 
In addition they must be contained in a small neighborhood of $\ell$ by the continuity. 
Thus, there is exactly one more ideal point of $P_x$ near the other endpoint of $\ell$. 
\begin{figure}[H]
\begin{overpic}[scale=.6
] {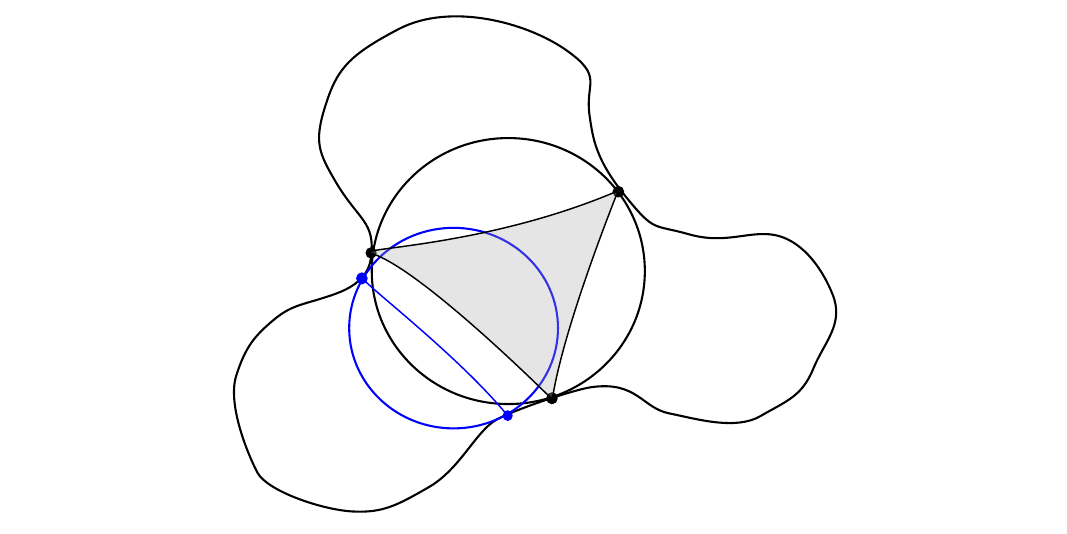} 
\put( 48 , 9 ){\textcolor{blue}{\small $x$}}  
\put( 45 , 20 ){\textcolor{Black}{\small $\ell$}}  
      \end{overpic}
\caption{}\label{fIsolatedStratum}
\end{figure}
\end{proof}

Let $\on{pr}_P\cn \h^3 \to \Conv(\rs \sm P)$ denote the nearest point projection onto the convex hull of $\rs \sm P$ in $\h^3$.
Let $(P, \LL_P) \to (\h^2, L_P)$ denote the collapsing map of $P$.
Then $\bdry \Conv (\rs \minus P)$ is the pleated surface induced by $(\h^2, L_P)$.

For ${\del_2} > 0$, consider the ${\del_2}$-neighborhood of $\Conv(\rs \sm P)$. 
Then its boundary surface  $S_{\del_2}$ is $C^1$-smooth and it carries an intrinsic Riemannian metric induced from  $\h^3$  (see \cite[II.1.3.6, II.1.5]{Epstein-Marden-87}).
Similarly let $\on{pr}_{\del_2}\cn P \to S_{1\del_2}$ denote the orthogonal projection along geodesics in $\h^3$; then $\on{pr}_{\del_2}$ is a $C^1$-diffeomorphism. 
Consider the Riemannian metric on $P$ obtained by pulling back the Riemannian metric on $S_{\del_2}$ via $\on{pr}_{\del_2}$, let $d_{\delta_2}\col P \times P \to \R_{\geq 0}$ be the associated distance function of $P$.
Then (iii) follows from:
\begin{claim}\Label{8-5no2}
For every  $\ep > 0$, if $\del_1 > 0$ and ${\del_2} > 0$ are sufficiently small  then 
 \begin{itemize}
 \item[(1)] $d_{\del_2}$ and ${\sf d}_Q$ are $(1 + \ep)$-bilipschitz on $K \times K$, i.e. $$1 - \ep < d_{\del_2}(x, y)/ {\sf d}_Q(x, y) < 1 +\ep$$ for all distinct $x, y$ in $K$. 
\item[(2)]$d_{\del_2}$ and ${\sf d}_P$ are $\ep$-close on $K \times K $.
\end{itemize}
 \end{claim}

\begin{proof}
(1) 
For each point $x \in P$, let $H_{\del_2}(x)$ be the unique hyperbolic plane in $\h^3$  tangent to $S_{\del_2}$ at $\on{pr}_{\del_2}(x)$. 
Then the boundary circle $\pt H_{\del_2} (x)$ is contained in $P$.
The boundary of the maximal ball $B_P(x)$ bounds another hyperbolic plane supporting, at $\on{pr}_P(x)$,  the pleated surface bounding $\Conv(\rs \sm P)$.
Those two hyperbolic planes are perpendicular to the geodesic through at $\on{pr}_{\del_2}(x)$ and $\on{pr}_P(x)$.  
Then the distance between the planes is exactly $\del_2$.

Let $B_{\del_2} (x)$ be the round open ball in $P$ bounded by $\pt H_{\del_2}(x)$.
Then $B_{\del_2}(x)$ contains $x$. 
Let $B_P(x)$ and $B_Q(x)$ be the maximal balls in $P$ and $Q$, respectively, centered at $x$.
Then, for every $\ep > 0$, if ${\del_2} > 0$ is sufficiently small, then $B_P(x)$ is $\ep$-close to $B_{\del_2}(x)$ on $\rs$ for every $x \in P$.
In addition, by Proposition \ref{4-26no2} (i), if $\del_1 > 0$ is sufficiently small then,  $B_P(x)$ and $B_Q(x)$ are $\ep$-close for all $x \in K$.
Then $B_Q(x)$ and $B_{\del_2}(x)$ are $2\ep$-close.
Therefore if $\del_1 > 0$ and ${\del_2} > 0$ are sufficiently small then,  for every unite tangent vector $v$ at a point in $K$, the derivatives $d \on{pr}_{\del_2}(v)$ and $d\on{pr}_P(v)$ are tangent vectors in $\h^3$, that are $\ep$-close.  
Thus $d_{\del_2}$ and ${\sf d}_Q$ are $\ep$-bilipschitz on $K$.

(2)
Let $H$ be the subsurface of $P$ where the Thurston metric is negatively curved. 
Then $\on{pr}_P$ takes $H$ isometrically onto its image in $\pt\Conv(\rs \minus P)$ with the intrinsic metric induced by $\h^3$. 
By identifying $L_P$ and its image on $\pt\Conv(\rs \minus P)$, then,  $x \in H$ if and only if $\on{pr}_P(x)$ is not on a leaf of $L_P$ with positive weight.
Then $\on{pr}_P$ is $C^1$-smooth on $H$. 
Thus similarly to (1),  for every $\ep > 0$, if $\del_2 > 0$ is sufficiently small, then  $d_{\del_2}$ and ${\sf d}_P$ are $\ep$-bilipschitz on each connected component of $H$.


Each connected component $E_\ell$ of the Euclidean subsurface of $P$ corresponds to a leaf $\ell$ of $L_P$ with positive weight, so that $E_\ell = \on{pr}_P\iv(\ell)$.
we show that, for every $\ep > 0$, if $\del_2 > 0$ is sufficiently small, then, for every leaf $\ell$ of $L_P$ with weight $w(\ell) > 0$,  ${\sf d}_P$ and $d_{\del_2}$ are $(1 + \ep, w(\ell) \del_2)$-quasi isometric on $E_\ell$.  

Then $E_\ell$ is, in the Thurston metric,  an infinite Euclidean strip with width $w(\ell)$. 
Thus we may regard $E_\ell$ as a subset of $\R^2$ so that it is infinite in the vertical direction. 
On $\rs$, the strip $E_\ell$ is regarded as a {\it wedge}, i.e. a region bounded by two circular arcs sharing both endpoints.
Consider the $\del_2$-neighborhood $M$, in $\h^3$, of the geodesic $m$ connecting the vertices of $E_\ell$\,--- it is an infinite solid cylinder invariant under any hyperbolic translation along $m$.
Then $d_{\del_2}$ on $E_\ell$ is given by pulling back the intrinsic metric on $\bdry M$ by $\pr_{\del_2}$. 
The boundary of $M$ is foliated by round loops bounding (geometric) disks of radius $\del_2$ orthogonal to $m$ in $\h^3$.  
Then, by the nearest point projection $\h^3 \to m$,  each loop map to a single point on $m$. 
In addition there is another foliation of $\bdry M$ by straight lines (with its intrinsic metric) that are orthogonal to the round loops. 
Then each straight line diffeomorphically projects onto $m$ by the projection $\h^3 \to m$.  

For different points $p,q$ in $E_\ell$, a geodesic connecting $p$ to $q$ with  ${\sf d}_P$ can be realized as  a union of a vertical geodesic segment and horizontal geodesic  segment. 
On vertical lines in $E_\ell$, for every $\ep > 0$, if $\del_2 > 0$ is sufficiently small, then ${\sf d}_P$ and $d_{\del_2}$ are $(1 + \ep)$-bilipschitz. 
On the other hand, on the horizontal lines, ${\sf d}_P$ and $d_{\del_2}$ are $w(\ell) \del_2$ rough isometric. 
Therefore for even $\ep > 0$, if $\del_2 > 0$ is small, then for every leaf $\ell$ of $L_P$ of positive weight, the projection ${\sf d}_P$ and $d_{\del_2}$ are  $(1 + \ep, w(\ell) \del_2)$-quasiisometric on $E_\ell$.  

The total transversal measure on $K$ given by $\LL_P$ is finite. 
Therefore, for every $\ep > 0$, if $\del_2 > 0$ is sufficiently small, then ${\sf d}_P$ and $d_{\del_2}$ are $(1 + \ep, \ep)$-quasiisometric on $K$. 
Since $K$ is a compact subset, we can in addition assume that they are $\del_2$-rough isometric. 
\end{proof}
\Qed{5-17no2}

In the rest of this section, we show the convergence of  the transversal measures for Theorem \ref{4-23no4} (i).

\begin{proposition}\Label{7-25-12no1}
Let $p_0, p_1$ be (distinct) points  in a single stratum of $(R, \LL)$.
The $\omega_i(p_0, p_1) \to \omega(p_0, p_1) = 0$ as $i \to \In$, where $\omega(p_0, p_1)$ and $\omega_i(p_0, p_1)$ are the transversal measures  of, in Thurston metrics, the geodesic segment from $p_0$ to $p_1$  on $R$ and $R_i$, respectively.
\end{proposition}

\begin{proof}
Let $P$ be the strata of $(\RR, \LL)$ containing $p_0$ and $p_1$.
Then let  $\ap\cn [0, 1] \to P$ be the geodesic segment from $p_0$ to $p_1$, i.e. $\alpha(0) = p_0$ and $\alpha(1)= p_1$.
Let $Q = \kp(P)$, the corresponding strata of $(\h^2, L)$.
We can naturally identify $Q$ and $\beta(Q)$. 
For each $j = 1,2$, let $q_j = \kap(p_j)$ and $q_{i, j} = \kap_i(p_j)$ for all sufficiently large $i \in \N$. 
Then, for each $j$, the point $\beta_i (q_{i, j})$ converges to the point $\beta(q_j)$ as $i \to \If$ (Corollary \ref{6-26no1}).
Let $N$ be the totally geodesic hyperplane in $\h^3$ orthogonally intersecting $\beta(Q)$ in the geodesic segment from  $\beta(q_0)$ to $\beta(q_1)$.

For each $t \in [0,1]$,  considering  the maximal ball in $R_i$ centered at $\ap(t)$, let $H_{i, t}$ be the totally hyperbolic plane in $\h^3$  bounded by the boundary of the maximal ball.
By Proposition \ref{4-26no2}, $\sup_{t \in [0,1]} \angle(N, H_{i, t}) \to \pi/2$ as $i \to \In$.
Thus, for every $\ep > 0$, if $i$ is sufficiently large, then $H_{i, s}$ and $H_{i, t}$ intersect and $\angle(H_{i, s}, H_{i, t}) < \ep$  for all $s,  t \in [0,1]$.
This implies that $\mu_i(q_0, q_1) < \ep$ (see \cite{Epstein-Marden-87}).
Therefore $\omega_i(p_0, p_1) < \ep$.
\end{proof}

Let $d_R$ denote Thurston metric on $R$. 
Then
\begin{proposition}\Label{5-24no1}
For every $\ep > 0$, there exists $\dl > 0$ 
such that, if  $p_0,p_1 \in R$ are points contained in different strata of $(R, \LL)$ satisfying
$d_R(p_0, p_1) < \del$ and $\angle([p_0, p_1], \LL) < \del$, then 
 $$1 - \ep < \frac{\omega( \,p_0,p_1\, )}{\omega_i(p_0, p_1)} <  1 + \ep,$$
 for all sufficiently large $i \in \n$.
  \end{proposition}
\pf{5-24no1}
Let $p_0, p_1$ be points in $R$ satisfying the assumptions.
Let $q_0 = \kap(p_0)$ and $q_1 = \kap(p_1)$.
Then for every $\ep > 0$, $\del > 0$ is sufficiently small, then there is a  hyperbolic plane $N$ in $\h^3$ passing $\beta(q_0), \beta(q_1)$ so that $N$ is $\ep$-nearly orthogonal to $\beta(Q)$ for all strata $Q$ of $(\h^2, L)$ intersecting $[q_0, q_1]$.

Let $\ap\cn [0, 1] \to R$ be the geodesic connecting $p_0$ to $p_1$.
Then, for each $t \in [0,0]$, the maximal ball in $R$ centered at $\ap(t)$ shares its boundary circle with a unique hyperbolic plane  $H_t$  in $\h^3$. 
Note that, if $d_R(p_0, p_1) > 0$ is sufficiently small,  then $H_s$ an $H_t$ must intersect for all $s, t \in [0,1]$. 
Let $\theta$ be a subdivision of $[0,1]$ as  $0 =t_0 < t_0 < \dots < t_{n_\theta} = 1$.
Let $|\theta|$ be the maximal width  of the subintervals $[t_k, t_{k+1}]~(0 \leq k < n_{\theta} -1)$. 
Then the transversal measure $\omega([p_0, p_1])$ is the limit of   
\begin{displaymath}
 \Sigma_{k = 1}^{n_\theta} \angle_{\h^3} (H_{t_k}, H_{t_{k+1}})
\end{displaymath} 
as $|\theta| \to 0$, where $\angle_{\h^3} (H_{t_k}, H_{t_{k+1}})$ be the angle taking its value in $[0, \pi/2]$ between the hyperbolic planes $H_{t_k}$ and $H_{t_{k+1}}$ in $\h^3$.
Note that this summation decreases when the subdivision $\theta$ is refined (\cite[II.1.10]{Epstein-Marden-87} ).  

For $s, t  \in [0,1]$, the geodesic $H_s \cap N$ intersect the geodesic $H_t \cap N$\,; 
 let $\angle_N(H_s, H_t) \in [0, \pi/2]$ denote their intersection angle in $N$.
 Since $\angle_{\h^3} (H_t, N)$ is $\ep$-close to $\pi/2$ for all $t \in [0,1]$, 
 by taking sufficiently small $\del > 0$, we can assume that
$$\displaystyle 1 - \ep < \frac{\angle_{\h^3}(H_s, H_t)}{\angle_N(H_s, H_t)} < 1 + \ep,$$
for all $s ,t \in [0,1]$.
Thus  
$$\displaystyle 1 - \ep <\frac{\Sigma_{k = 1}^{n_\theta}  \angle_{\h^3}(H_{t_k}, H_{t_{k+1}})}{\Sigma_{k =1}^{n_\theta} \angle_N(H_{k}, H_{k+1})} < 1 + \ep.$$

Since $R_i$ contains the geodesic segment $\ap$ for sufficiently large $i$,  similarly let $H_{i, t}$ be a copy of $\h^2$ such that $\pt_\In H_{i, t}$ bounds the maximal ball in $R_i$ centered  at $\ap(t)$.
Then the transversal measure $\omega_i(p_0, p_1)$ is the limit of   
\begin{displaymath}
\Sigma_{k = 1}^{n_\theta} \angle_{\h^3} (H_{i, t_k}, H_{i, t_{k+1}})
\end{displaymath} 
as $|\theta| \to 0$.
For every $\ep > 0$, if $i$ is sufficiently large, the hyperbolic planes $H_t$ and $H_{i,t}$ are $\ep$-close for all $t \in [0,1]$.
Thus, if  $\del > 0$ is sufficiently small and $i$ is sufficiently large, then  $H_{i, t}$ intersects $N$ at an angle $\ep$-close to $\pi/2$.
Thus we can in addition assume that 
$$\displaystyle 1 - \ep <\frac{\Sigma_{k=1}^{n_\theta}  \angle_{\h^3}(H_{t_k}, H_{t_{k+1}})}{\Sigma_{k =1}^{n_\theta} \angle_N(H_{t_k}, H_{t_{k+1}})} < 1 + \ep,$$
for any subdivision $\theta$.

Therefore it remains to show that, if $|\theta|$ is sufficiently small, then 
\begin{eqnarray}
 - \ep < \Sigma_{k =1}^{n_\theta} \angle_N(H_{t_k}, H_{t_{k+1}}) - \Sigma_{k=1}^{n_\theta} \angle_N(H_{t_k}, H_{t_{k+1}})  < \ep. \label{eqn:approximation}
\end{eqnarray}
Consider the convex subset $X_\theta$ of $\h^3$ bounded by the hyperbolic planes $H_{t_1}, \dots, H_{t_n}$ so that $X_\theta$ contains $\Conv(\rs \minus R)$.
Then $N$ intersects $\pt X_\theta$ nearly orthogonally, and the intersection is a piecewise geodesic that is a convex bi-infinite curve through $\beta(q_0)$ and $\beta(q_1)$, and its non-smooth points are between $\beta(q_0)$ and $\beta(q_1)$.
Pick  a segment $\eta_\theta$ of this curve that is slightly larger than the segment from $\beta(q_0)$ to $\beta(q_1)$ so that the interior of $\eta_\theta$ contains $\beta(q_0)$ and $\beta(q_1)$. 
Then $\Sigma_{k =1}^{n_\theta} \angle_N (H_{t_k}, H_{t_{k+1}})$ is equal to the sum of the exterior angles of $\eta_\theta$.

Similarly let $X_{i, \theta}$ be the convex subset of $\h^3$ bounded by $H_{i, t_1}, \dots, H_{t_n}$, such that $X_{i, \theta}$ contains  $\Conv(\rs \minus R_i)$. 
Then $\pt X_{i, \theta} \cap N$ is a piecewise geodesic convex curve in $N$, which converges to the convex curve $\pt X_\theta \cap N$ above as $i \to \In$.
 For sufficiently large $i$, each endpoint $\eta_\theta$ has a unique closest point on $\pt X_{i, \theta} \cap N$.
 Then those closest points cut off a segment $\eta_{i, \theta}$  of  $\pt X_{i, \theta} \cap N$ that contains all non-smooth points.
 Then $\Sigma_{k=1} \angle_N (H_{i, t_k}, H_{i, t_{k+1}})$  is the sum of the exterior angles of $\eta_{i, \theta}$.

 Consider the loop $\ell_i$ that obtained by connecting the corresponding endpoints of $\eta_\theta$ and $\eta_{i, \theta}$ by geodesic segments. 
 Then, since $\eta_{i, \theta}$ converges to $\eta_\theta$ as $i \to \In$,  the area in $N$ bounded by $\ell_i$ converges to $0$ as $i \to \In$.
 By applying Gauss-Bonnet Theorem to  $\ell_i$ in the hyperbolic plane $N$, we obtain (\ref{eqn:approximation}).
\Qed{5-24no1}
\begin{proposition}\Label{7-25-12no2}
For all $p, q \in R$,
$\omega_i(p, q) \to\omega(p, q)$
as $i \to \In.$
\end{proposition}
\begin{proof}
For every $\dl > 0$, pick a simple piecewise geodesic path $\eta = \cup_{k=1}^{n}[p_k, p_{k+1}]$ in $R$ connecting $p$ to $q$, where $p_k$ are points in $R$, such that   $d_R(p_k, p_{k+1}) < \del$ and 
$\pi/2 -\del < \angle(\LL, [p_k, p_{k+1}]) < \pi/2 - \dl$ for all $k = 0,1,\dt, n-1$. 
  By Proposition \ref{5-24no1}, if $\del > 0$ is sufficiently small, then if $p_k$ and $p_{k+1}$ are in different strata of $(R, \LL)$, then 
    $$1 - \ep < \frac{\omega(p_k, p_{k+1})}{\omega_i(p_k, p_{k+1})} <  1 + \ep$$
For sufficiently large $i$.
If $p_k$ and $p_{k + 1}$ are in a single stratum of $(R, \LL)$, then by Proposition \ref{7-25-12no1}, $\omega_i(p_k, p_{k + 1}) \to 0 = \omega(p_k, p_{k+1})$.
Clearly $\omega_i (p, q)= \Sigma_{k = 1}^n \omega_i(p_k, p_{k+1})$ and $\omega (p, q)= \Sigma_{k =1}^n \omega(p_k, p_{k+1})$. 
Thus for every $\ep > 0$, if $\del > 0$ is sufficiently small, then $|w(p,q) - w_i(p,q) | < \ep$ for sufficiently large $i$. 
\end{proof}

\begin{corollary}
Let $K$ be a compact subset  of $R$.
Then for every $\ep > 0$, if $i \in \N$ is sufficiently large, then $ -\ep < \omega(p, q) - \omega_i(p, q)< \ep$ for all $p, q \in K$.  
\end{corollary}

\begin{proof}
For every point $x \in K$, there is a neighborhood $U_x$ such that it follows from Proposition \ref{4-26no2} that, for every $\ep > 0$, if $i$ is sufficiently large, then  $-\ep < \omega(y, z) - \omega_i(y, z) < \ep$ for all $y, z \in U_x$ .
Since $K$ is compact, there are finitely many points $x_1, \dots, x_n$ such that $U_{x_1}, \dots, U_{x_n}$ cover $K$.
Applying Proposition \ref{7-25-12no2} to all pairs of points in $x_1, \dots, x_n$, we have $-\ep < \omega(x_j, x_k) - \omega_i(x_j, x_k) < \ep$ for all $0 < j, k \leq n$.
Then the Triangle Inequality implies the corollary. 
\end{proof}

\section{Proof of Theorem \ref{3-10no2}}\Label{Proof}

Let $\KK$ be a compact connected surface $\pi_1$-injectively embedded in  $\CC_i$. 
Recalling the natural embedding $e_i \col\CC_i \to \CC_\infi$, we let $\KK_i = e_i\iv(\KK)$ for each $i \in \n$.  
Since $\CC_1 \sub \CC_2 \sub \dots$ exhausts $\CC_\infi$,  if $i$ is sufficiently large,  $\KK_i$ is isomorphic to $\KK$ by $e_i$, and thus  $\KK_i$ is a compact subsurface of $\CC_i$. 
Since $\CC_i \st C_i$, naturally $\KK_i \sub C_i$. 
Recall that  $\tau_i$ and $\tau_\infi$ are homeomorphic to $S$ and that $\tau_\infi$ is obtained by identifying the boundary geodesics of $\sigma_\In$.
Then, since $\kap_i\cn C_i \to \tau_i$ and $\iota_\In\cn \CC_\In \to \sigma_\In$ are collapsing maps,  when $\KK_i$ is isomorphic to $\KK$, then $\kap_i| \KK_i $ and $\iota_\In | \KK$ are homotopic as maps to $S$. 
Let  $\kpt_i(\td{\KK}_i)$ and $\td{\iota}_\If(\td{\KK})$ denote the universal covers of $\kp_i(\KK_i)$  and $\iota_\If(\KK)$, respectively. 
Then, recalling that $\NN_\infi$ is the canonical circular lamination on  $\CC_\infi$,  we have

\begin{proposition}\Label{4-22no1}
There exists a sequence of (not necessarily continuous) maps $\psi_i\cn \iota_\If(\KK) \to   \kp_i(\KK_i)$ for $i \in \N$,  such that, letting $\psit_i\cn\td{\iota}_\If(\td{\KK}) \to \kpt_i(\td{\KK}_i)$ be the lift of $\psi_i$, which commutes with deck transformations,  we have
\begin{itemize}
\item[(i)] $\LL_i$ on $\KK_i$ converges to $\NN_\In$ on $\KK$ uniformly,
\item[(ii)] $\psi_i$ converges to an isometry uniformly as $i \to \If$, 
\item[(iii)]  the sup distance between $\kp_i \cc e_i^{-1}$ and $\psi_i \cc \iota_\In$  converges to zero on $\KK$ as $i \to \In$ (Figure \ref{050813}),
\item[(iv)] the sup distance between $\beta_i \cc \td{\psi}_i$ and $\beta_\If$ converges to $0$ on $\iota_\In(\KK_\If)$ as $i \to \If$,

\end{itemize}
and therefore 
\begin{itemize}
\item[(v)]  for $x, y \in \iota_\In (\KK)$  not on leaves of positive atomic measure, let $[x, y]$ be a geodesic segment connecting $x$ to $y$ in $\sigma_\In$ and let $[\psi_i(x), \psi_i(y)]$ be the geodesic segment on $\tau_i$ that is homotopic to  $\psi_i([x, y])$ with its endpoints fixed;
then the transversal measure on $[\psi_i(x), \psi_i(y)]$ by $L_i$ converges to    the transversal measure on the geodesic segment $[x ,y]$ by $N_\In$.  

\end{itemize}
More precisely, in  (i), we mean that for every $\ep > 0$,  if $i$ is sufficiently large, then for all $x, y \in \KK$, then the transversal measure of $[x, y]$  given by $\LL_i$ is $
\ep$-close to that given by $\NN_\In$.
 By (ii), we mean that for every $\ep > 0$, if $i$ is sufficiently large, then $$- \ep < \dist_{\h^2}( \td{\psi}_i(x), \td{\psi}_i(y) ) -\dist_{\h^2}(x, y) < \ep$$ for every $x, y \in \td{\iota}_\In(\td{\KK})$.

\end{proposition}

\begin{figure}
\begin{displaymath}
\xymatrix{\CC_\In \supset \KK \ar[d]^{\iota_\In} \ar@{->}[r]^{e_i^{-1}} & \KK_i \subset C_i \ar[d]^{\kap_i}\\
\sigma_\In \ar[r]^{\psi_i}	& \tau_i}
\end{displaymath}
\caption{}
\label{050813}
\end{figure}

\pf{4-22no1}
It suffices to show that, for every $p$ in $\td{\CC}_\In$, there is a compact neighborhood of $p$ with the desired properties.
Consider all maximal balls in $\til{\CC}_\infi$ containing $p$\,; then the union of their cores is a neighborhood of $p$ contained in the canonical neighborhood of $p$.
Thus we can assume  that $\KK$ is a simply connected region contained  this union.

For  sufficiently large $i \in \ \n$,
let $p_i \in \td{\CC}_i$ such that  $\td{e}_i(p_i) = p$.
Let $U_i$ be the canonical neighborhood of $p_i$ in $\til{C}_i$.
By Proposition \ref{4-23no2}, $U_i$ converges to $\UU_\If$ and  $\pt U_i$ converges to $\pt \UU_\If$ on $\rs$ as $i \to \If$ in the Hausdorff metric for $\rs$ (by fixing  a natural metric on $\rs$).
Hence, by Theorem \ref{4-23no4} (including the definition of $\phi_i$) and Proposition \ref{4-23no3}, we have (i) - (iv).
\Qed{4-22no1}

Let $l_\In$ and $l_i$ be the geodesic representatives  of $\ell$ in $\tau_\If$ and $\tau_i$, respectively. 
We first show that $\tau_i \to \tau_\In$ and $\beta_i \to \beta_\In$ as $i \to \In$.
Let $\sigma_i$ be $\tau_i \minus l_i$. 
Then, by Proposition \ref{4-22no1} (ii), $\sigma_i$  converges to $\sigma_\In (= \tau_\If \sm l_\If)$ as $i \to \If$.
In other words, $\tau_i$  converges to $\tau_\If$ possibly up to a ``twist'' along $\ell_\If$.
By Proposition \ref{4-22no1} (iii), the restriction of $\beta_i$ to a lift of  $\sigma_i (\sub \tau_i)$ to $\h^2$ converges to the restriction of $\beta_\In$ to the corresponding lift of $\sigma_\infi (\sub \tau_\In)$ to $\h^2$.
Since $\beta_i$ and $\beta$ are both $\rho$-equivariant, $\beta_i$ must converge to $\beta_\In$ (c.f. \S \ref{identification}) as $i \to \infi$, which proves (ii).
Therefore  $\tau_i$ must converge to $\tau_\If$.

Last we show the convergence of $L_i$.
By Proposition \ref{4-22no1} (v),  the restriction of $L_i$ to $\sigma_i$ converges to the restriction of $L_\infi$ to $\sigma_\In$ as $i \to \If$ uniformly on compacts.
Thus it is left to show that the transversal measure of $L_i$ near $\ell_i$ must diverges to $\infi$.
Each connected component of $\CC_\In \minus \CC_0$ is a half-infinite grafting cylinder.
Then this cylinder has infinite total transversal measure given by  $\NN_\In$. 
Thus, by Proposition \ref{4-22no1} (iv),  for any fixed $j \in \N$, the total transversal measure on $C_i \minus \CC_j$ given by $\LL_i$ diverges to $\In$ as $i \to \infi$. 
Let $\ap_\In$ be a smooth arc on $\tau_\In$ transversal to $L_\In$ such that $\ap_\In$ intersects $\ell_\In$ in a single point.
Then the  transversal measure of $\ap_\In$ by $L_\infi$ is infinite.
By the convergence  $\tau_i \to \tau$, we have $|L_i| \to |L_\In|$ (in $\GL$) as $i \to \In$.  
Thus let $(\alpha_i)$ be a sequence of arcs $\alpha_i$ on $\tau_i$ smoothly converges to an arc $\ap_\In$, so that $\alpha_i$ is transversal to $L_i$ for sufficiently large $i$.
Since the total transversal measure of $C_i \minus \CC_j$ diverges as $i \to \infi$ as above, the divergence,  accordingly the transversal measure of $\alpha_i$ given by $L_i$ must diverge to $\In$.
Therefore $L_i$ converges to $L_\In$ as $i \to \In$.

\part{Appendix: density of holonomy map fibers $\PP_\rho$ in $\pml$.}
Recall Thurston coordinates $\PP \cong \mathscr{T} \times \ML$ (\S \ref{thurston}) on  the space $\PP$ of all (marked) projective structures on $S$.
This gives an obvious projection from $\PP$ to $\ML$. 
Then the obvious projection $\ML \minus \{\emptyset\} \to \pml$ extends to $\Phi\cn\ML \to \PML \sqcup \{\emptyset\}$ so that  the empty lamination $\emptyset$ maps to  $\emptyset$.
 
Recall  from \S \ref{S:intro} that $\PP_\rho$ is the set of all projective structures with fixed holonomy $\rho\cn\po(S) \to \psl$ and that  $\PP_\rho$ is a discrete subset of $\PP$.
 If $\rho$ is fuchsian, letting $\tau \in \mathscr{T}$ be the corresponding hyperbolic structure, we have  $$\PP_\rho \cong \{\,(\tau, M)\, \vert~ {\rm multiloops~ M ~ with ~2 \pi\on{-}multiple ~weights} \},$$
in Thurston coordinates (\cite{Goldman-87}, c.f. \cite{Baba-15gt}).
Thus $\Phi(\PP_\rho)$ is the union of  $\emptyset$ and  a dense subset of $\PML$. 
Note that a projective structure $C \in \PP$ maps to $\emptyset$ via $\Phi$ if and only if $C$ is a hyperbolic structure (\cite{Goldman-87}). 
Thus, for almost all $\rho\colon \pi_1(S) \to \PSL$, we have  $\Phi(\PP_\rho) \not\ni  \emptyset $. 
Then
\begin{theorem}\Label{8-15-12no1}
Given arbitrary $\rho\cn \pi_1(S) \to \psl$, if  $\PP_\rho$ is non-empty, then  $\Phi(\PP_\rho) \minus \{\emptyset\}$ is a dense subset of  $\pml$.
\end{theorem}

A \textit{Schottky decomposition} of a representation $\rho\cn \pi_1(S) \to \psl$ is a decomposition of $S$ into pairs of pants $P_k$ along a (maximal) multiloop $M$ such that
the restriction of $\rho$ to $\pi_1(P_k)$ is an isomorphism onto a Schottky group for each pants $P_k$.  
A \textit{Schottky decomposition} of a projective structure $C = (f, \rho)$ is a decomposition of  $C$ into pairs of pants along a multiloop $M$ on consisting of admissible loops such that  $M$ realizes a Schottky decomposition of $\rho$. 
\begin{proposition}\Label{8-17-12no1}
Let $\rho\cn \pi_1(S) \to \psl$ be the holonomy representation of some projective structure on $S$.
Then for every uniquely ergodic measured lamination $L$, there is a sequence of projective structures $C_i$ with holonomy $\rho$ such that there is, for each $i$,  a Schottky decomposition of $C_i$ along some admissible multiloop containing a loop $\ell_i$ and $[\ell_i] \to [L]$ in $\PML$ as $i \to \In$.
\end{proposition}

\begin{proof}
Given a non-elementary  representation $\rho\cn\po(S) \to \PSL$ that lifts to $\pi_1(S) \to \operatorname{SL}(2,\C)$, Gallo Kapovich and Marden gave a Schottky decomposition of $\rho$ along a multiloop $M$ and then constructed a projective structure $C$ with holonomy $\rho$ that admits a Schottky decomposition along $M$ (\cite[\S 4, 5]{Gallo-Kapovich-Marden}).  
We sketch their construction and explains how it implies the Proposition, following the notations in \cite{Gallo-Kapovich-Marden}.  

Let  $a$ be an arbitrary element of  $\pi_1(S)$ representing an essential loop on $S$; then
modify $a$ in several steps to another loop (namely $d^n x$ in  \cite[p.650]{Gallo-Kapovich-Marden}) that represents to  a loop $d'$. 
Then $d'$ extends to a multiloop realizing a Schottky decomposition of $\rho$. 
Thus,  for a sequence of $a_i \in \pi_1(S)$ representing simple loops $\ap_i$ with $[\ap_i] \to [L]$ as $i \to \In$,
letting $d'_i$ be the loop given by appropriately applying the above contraction to $d_i$,
we claim that  $d'_i$  also converges to $[L]$ because each modification step in the construction preserves the convergence property.

We can assume that $\ap_i$ are non-separating loops, replacing $\ap_i$ by a non-separating loop disjoint from $\ap_i$.
Then the convergence $[\ap_i] \to [L]$ still holds.
By abusing notation, we let elements of $\pi_1(S)$ also denote their corresponding loops on $S$. 
Then 
given $\rho\cn\pi_1(S) \to \psl$, a \textit{handle} is a pair of elements $a, b \in \po(S)$ such that 
\begin{itemize}
\item  $a$ and $b$ are simple loops on $S$ intersecting in a single point, and
\item $\rho(a), \rho(b)$ are loxodromic, and they generate a non-elementary subgroup of $\psl$.  
\end{itemize}

By modifying $a_i \in \pi_1(S)$, possibly in a few steps, we obtain a handle $H_i$ by Proposition 3.1.1 in \cite{Gallo-Kapovich-Marden} for each $i$. 
 if necessary after changing $a_i$,  let $H_i = \langle a_i, b_i \rangle$ with some $b_i \in \pi_1(S)$.
By  the proof of the proposition, we can assume that  the projective classes $[a_i]$ and $[b_i]$ also converge to $[L]$ as $i \to \In$ in $\pml$ :  Basically each  modification is given by chaining a loop $\ell$ to another loop with a bounded intersection number with $\ell$ or to a loop ``twisted'' along $\ell$ many times. 

For each handle $H_i$, we let $\ap_i = \rho(a_i)$ and $\beta_i = \rho(b_i) \in \psl$, which are loxodromic elements.
Then we may in addition assume that $\beta_i$ does not take a fixed point of $\ap_i$ to the other (\cite{Gallo-Kapovich-Marden},\S 4.2). 
Indeed this modification is done by, if necessary, replacing $\langle a_i, b_i \rangle$ by a new handle of the form either  $\langle a_i b_i^q, b_i \rangle$ or $\langle b_i, a_i b_i^q \rangle$.
This modification also preserves the convergence to $[L]$ since $a_i$ and $b_i$ intersects in a single point. 

Pick  another pair of non-separating loops $x_i, y_i$ in $\pi_1(S)$  such that $x_i, y_i$ intersect in a single point and they are disjoint from $a_i$ and $b_i$( \cite{Gallo-Kapovich-Marden},\S 4.3).
Clearly $[x_i]$ and $[y_i]$ converge to $[L]$ as $i \to \In$. 
Then the induced multiloop for the Schottky decomposition of $\rho$ contains a loop of the form $d_i^{n_i} x_i$, where $d_i = y_i b_i a_i^{k_i}$, for some $k_i, n_i \in \Z$. 
Then $d_i^{n_i} x_i$ is a non-separating loop disjoint from $b_i a_i^{k_i}$ (\cite{Gallo-Kapovich-Marden},\S 4.5).
Since $\langle a_i, b_i\rangle$ is a handle, the loop $b_i a_i^{k_i}$ intersects the loop $a_i$ in a single point, and thus the projective class $[b_i a_i^{k_i}]$ also converges to $[L]$ as $i \to \In$. 
Hence $[d_i]$ and thus $[d_i^{n_i} x_i]$ converge to $[L]$ as $i \to \In$.

\end{proof}

\proofof{Theorem \ref{8-15-12no1}}
Let $L$ be a uniquely ergodic measured lamination on $S$. 
By Proposition \ref{8-17-12no1},  there are sequences of projective structures $C_i$ with holonomy $\rho$ and  admissible loops $\ell_i$ on $C_i$  converging to $[L]$ in $\pml$.

For  $n_i \in \N$, consider the projective structure $\Gr_{\ell_i}^{n_i}(C_i)$ obtained by $n_i$ times grafting $C_i$ along $\ell_i$.
Then its measured lamination $L_{i, n_i}$, in Thurston coordinates, converges to $[\ell_i]$ as $n_i \to \In$ in $\pml$ by Theorem \ref{3-10no2}.
We can pick sufficiently large $n_i$ for each $i$ so that  $[L_{i, n_i}]$ converges to $[L]$ as $i \to \In$.
Therefore $[L]$ is an accumulation point of $\Phi(\PP_\rho)$.
Almost all measured laminations are uniquely ergodic lamination and in particular they are dense in $\pml$.
Thus $\Phi(\PP_\rho)$ is dense in $\pml$.
\Qed{8-17-12no1}
\bibliographystyle{amsalpha}

\bibliographystyle{plain}
\bibliography{../../../Reference}

\end{document}